\documentclass[11pt]{article}
\pdfoutput=1

\usepackage[letterpaper,margin=1in,bottom=1.4in]{geometry} 
\usepackage[format=plain,font=small,labelfont=bf,textfont=up]{caption}

\usepackage[utf8]{inputenc} 
\usepackage[T1]{fontenc}    
\usepackage[hidelinks]{hyperref}       
\usepackage{url}            
\usepackage{booktabs}       
\usepackage{amsfonts}       
\usepackage{nicefrac}       
\usepackage{microtype}      
\usepackage[affil-it]{authblk} 
\usepackage{tikz}
\usepackage{lmodern}
\usepackage{float}

\usepackage[linesnumbered,ruled,vlined]{algorithm2e}

\SetCommentSty{mycommfont}

\usepackage{amsmath}
\usepackage{amssymb}
\usepackage{amsthm}
\usepackage{mathtools}
\mathtoolsset{showonlyrefs,showmanualtags}
\allowdisplaybreaks
\usepackage{bbm}
\usepackage{graphicx}
\usepackage{subcaption}
\usepackage{thmtools, thm-restate}  
\usepackage[toc,page,header]{appendix}
\usepackage{minitoc}
\usepackage{enumitem}

\newcommand{\figpath}{figs}

\newcounter{spacesave}
\usepackage[style=authoryear-comp,sorting=nyt,sortcites=false, natbib=true, backend=bibtex, maxbibnames=99, doi=false, url=false]{biblatex}
\renewbibmacro{in:}{} 

\DeclarePairedDelimiter\paren\lparen\rparen

\DeclarePairedDelimiter\braces\lbrace\rbrace

\DeclarePairedDelimiter\abs\lvert\rvert

\providecommand{\bbone}{\mathbf{1}}
\DeclarePairedDelimiterXPP\indicator[1]{\bbone}{\lbrack}{\rbrack}{}{#1}

\DeclarePairedDelimiterXPP\expf[1]{\exp}{\lparen}{\rparen}{}{#1}
\DeclarePairedDelimiterXPP\logf[1]{\log}{\lparen}{\rparen}{}{#1}
\DeclarePairedDelimiterXPP\maxf[1]{\max}{\lparen}{\rparen}{}{#1}
\DeclarePairedDelimiterXPP\minf[1]{\min}{\lparen}{\rparen}{}{#1}

\DeclareMathOperator*{\sgn}{sgn}
\DeclarePairedDelimiterXPP\sgnf[1]{\sgn}{\lparen}{\rparen}{}{#1}

\DeclarePairedDelimiterXPP\func[2]{#1}{\lparen}{\rparen}{}{#2}

\DeclareMathOperator*{\atan}{atan}
\DeclarePairedDelimiterXPP\atanf[1]{\atan}{\lparen}{\rparen}{}{#1}
\DeclarePairedDelimiterXPP\tanf[1]{\tan}{\lparen}{\rparen}{}{#1}

\DeclarePairedDelimiter\setb\lbrace\rbrace

\newcommand{\Reals}{\mathbb{R}}

\DeclareMathOperator*{\argmin}{arg\,min}

\newcommand{\mat}[1]{\boldsymbol{#1}}
\renewcommand{\vec}[1]{\boldsymbol{#1}}

\makeatletter
\newcommand*{\tran}{{\mathpalette\@tran{}}}
\newcommand*{\@tran}[2]{\raisebox{\depth}{$\m@th#1\intercal$}}
\makeatother

\DeclarePairedDelimiter\norm\lVert\rVert
\DeclarePairedDelimiterXPP\tnorm[1]{}{\lVert}{\rVert_{1}}{}{#1}
\DeclarePairedDelimiterXPP\enorm[1]{}{\lVert}{\rVert_{2}}{}{#1}
\DeclarePairedDelimiterXPP\inorm[1]{}{\lVert}{\rVert_{\infty}}{}{#1}
\DeclarePairedDelimiterXPP\pnorm[2]{}{\lVert}{\rVert_{#1}}{}{#2}
\DeclarePairedDelimiterXPP\opnorm[1]{}{\lVert}{\rVert_{op}}{}{#1}

\DeclarePairedDelimiterXPP\detf[1]{\det}{\lparen}{\rparen}{}{#1}

\DeclarePairedDelimiterXPP\kerf[1]{\ker}{\lparen}{\rparen}{}{#1}

\DeclareMathOperator{\trsym}{tr}
\DeclarePairedDelimiterXPP\tr[1]{\trsym}{\lparen}{\rparen}{}{#1}

\DeclareMathOperator{\vspansym}{span}
\DeclarePairedDelimiterXPP\vspan[1]{\vspansym}{\lparen}{\rparen}{}{#1}

\DeclareMathOperator{\diagsym}{diag}
\DeclarePairedDelimiterXPP\diag[1]{\diagsym}{\lparen}{\rparen}{}{#1}

\DeclareMathOperator{\ranksym}{rank}
\DeclarePairedDelimiterXPP\rank[1]{\ranksym}{\lparen}{\rparen}{}{#1}

\DeclareMathOperator{\vectorizesym}{vec}
\DeclarePairedDelimiterXPP\vectorize[1]{\vectorizesym}{\lparen}{\rparen}{}{#1}

\DeclareMathOperator*{\esssup}{ess\,sup}
\DeclarePairedDelimiterXPP\esssupf[1]{\esssup}{\lparen}{\rparen}{}{#1}

\let\Prsym\Pr
\let\Pr\relax
\DeclarePairedDelimiterXPP\Pr[1]{\Prsym}{\lparen}{\rparen}{}{%
	#1}

\DeclarePairedDelimiterXPP\Prsub[2]{\Prsym_{#1}}{\lparen}{\rparen}{}{%
	#2}

\DeclareMathOperator{\Esym}{E}
\DeclarePairedDelimiterXPP\E[1]{\Esym}{\lbrack}{\rbrack}{}{%
	#1}

\DeclarePairedDelimiterXPP\Esub[2]{\Esym_{#1}}{\lbrack}{\rbrack}{}{%
	#2}

\DeclareMathOperator{\Varsym}{Var}
\DeclarePairedDelimiterXPP\Var[1]{\Varsym}{\lparen}{\rparen}{}{%
	#1}

\DeclarePairedDelimiterXPP\Varsub[2]{\Varsym_{#1}}{\lparen}{\rparen}{}{%
	#2}

\DeclarePairedDelimiterXPP\EstVar[1]{\widehat{\Varsym}}{\lparen}{\rparen}{}{%
	#1}

\DeclareMathOperator{\Covsym}{Cov}
\DeclarePairedDelimiterXPP\Cov[1]{\Covsym}{\lparen}{\rparen}{}{%
	#1}

\DeclarePairedDelimiterXPP\Covsub[2]{\Covsym_{#1}}{\lparen}{\rparen}{}{%
	#2}

\DeclareMathOperator{\Corrsym}{Corr}
\DeclarePairedDelimiterXPP\Corr[1]{\Corrsym}{\lparen}{\rparen}{}{%
	#1}

\makeatletter
\newcommand{\indep}{\protect\mathpalette{\protect\@indep}{\perp}}
\newcommand*{\@indep}[2]{\mathrel{\rlap{$#1#2$}\mkern3mu{#1#2}}}
\makeatother

\newcommand{\bigOsym}{\mathcal{O}}
\DeclarePairedDelimiterXPP\bigO[1]{\bigOsym}{\lparen}{\rparen}{}{#1}
\DeclarePairedDelimiterXPP\bigOt[1]{\widetilde{\bigOsym}}{\lparen}{\rparen}{}{#1}

\newcommand{\littleOsym}{o}
\DeclarePairedDelimiterXPP\littleO[1]{\littleOsym}{\lparen}{\rparen}{}{#1}

\newcommand{\bigOpsym}{\bigOsym_p}
\DeclarePairedDelimiterXPP\bigOp[1]{\bigOpsym}{\lparen}{\rparen}{}{#1}

\newcommand{\littleOpsym}{\littleOsym_p}
\DeclarePairedDelimiterXPP\littleOp[1]{\littleOpsym}{\lparen}{\rparen}{}{#1}

\newcommand{\bigOmegasym}{\Omega}
\DeclarePairedDelimiterXPP\bigOmega[1]{\bigOmegasym}{\lparen}{\rparen}{}{#1}

\newcommand{\littleOmegasym}{\omega}
\DeclarePairedDelimiterXPP\littleOmega[1]{\littleOmegasym}{\lparen}{\rparen}{}{#1}

\newcommand{\bigThetasym}{\Theta}
\DeclarePairedDelimiterXPP\bigTheta[1]{\bigThetasym}{\lparen}{\rparen}{}{#1}

\DeclarePairedDelimiterXPP\cosf[1]{\cos}{\lparen}{\rparen}{}{#1}
\DeclarePairedDelimiterXPP\sinf[1]{\sin}{\lparen}{\rparen}{}{#1}

\newcommand{\quadtext}[1]{\quad\text{#1}\quad}

\newcommand{\quadand}{\quadtext{and}}

\newcommand{\newvar}[2]{
	\expandafter\newcommand\csname #1\endcsname{#2}
}

\newcommand{\newvars}[2]{
	\expandafter\newcommand\csname #1\endcsname[1]{#2_{##1}}
}

\newcommand{\newvarss}[2]{
	\expandafter\newcommand\csname #1\endcsname[2]{#2_{##1,##2}}
}

\newcommand{\newvarsss}[2]{
	\expandafter\newcommand\csname #1\endcsname[3]{#2_{##1,##2,##3}}
}

\newcommand{\newvarssss}[2]{
	\expandafter\newcommand\csname #1\endcsname[4]{#2_{##1,##2,##3,##4}}
}

\newcommand{\newfunconly}[2]{
	\expandafter\DeclarePairedDelimiterXPP\csname #1\endcsname[1]{#2}{\lparen}{\rparen}{}{##1}
}

\newcommand{\newfuncs}[2]{
	\expandafter\DeclarePairedDelimiterXPP\csname #1\endcsname[2]{#2_{##1}}{\lparen}{\rparen}{}{##2}
}

\newcommand{\newfuncss}[2]{
	\expandafter\DeclarePairedDelimiterXPP\csname #1\endcsname[3]{#2_{##1,##2}}{\lparen}{\rparen}{}{##3}
}

\newcommand{\newfuncsss}[2]{
	\expandafter\DeclarePairedDelimiterXPP\csname #1\endcsname[4]{#2_{##1,##2,##3}}{\lparen}{\rparen}{}{##4}
}

\newcommand{\newfuncssss}[2]{
	\expandafter\DeclarePairedDelimiterXPP\csname #1\endcsname[5]{#2_{##1,##2,##3,##4}}{\lparen}{\rparen}{}{##5}
}

\newcommand{\varpartialeval}[3]{
	\expandafter\newcommand\csname #2\endcsname{\csname #1\endcsname{#3}}
}

\newcommand{\funcpartialeval}[3]{
	\expandafter\newcommand\csname #2\endcsname[1][]{\csname #1\endcsname[##1]{#3}}
}

\newcommand{\varevali}[2][]{
	\varpartialeval{#2#1}{#2i}{i}
	\varpartialeval{#2#1}{#2j}{j}
}

\newcommand{\varevalk}[2][]{
	\varpartialeval{#2#1}{#2k}{k}
	\varpartialeval{#2#1}{#2l}{\ell}
}

\newcommand{\varevals}[2][]{
	\varpartialeval{#2#1}{#2s}{s}
	\varpartialeval{#2#1}{#2r}{r}
}

\newcommand{\funcevali}[2][]{
	\funcpartialeval{#2#1}{#2i}{i}
	\funcpartialeval{#2#1}{#2j}{j}
}

\newcommand{\funcevalk}[2][]{
	\funcpartialeval{#2#1}{#2k}{k}
	\funcpartialeval{#2#1}{#2l}{\ell}
}

\newcommand{\varevalii}[2][]{
	\varevali[#1]{#2}
	\varevali{#2i}
	\varevali{#2j}
}

\newcommand{\varevalik}[2][]{
	\varevali[#1]{#2}
	\varevalk{#2i}
	\varevalk{#2j}
}

\newcommand{\varevaliik}[2][]{
	\varevalii[#1]{#2}
	\varevalk{#2ii}
	\varevalk{#2ij}
	\varevalk{#2ji}
	\varevalk{#2jj}
}

\newcommand{\varevaliis}[2][]{
	\varevalii[#1]{#2}
	\varevals{#2ii}
	\varevals{#2ij}
	\varevals{#2ji}
	\varevals{#2jj}
}

\newcommand{\varevaliikk}[2][]{
	\varevaliik[#1]{#2}
	\varevalk{#2iik}
	\varevalk{#2ijk}
	\varevalk{#2jik}
	\varevalk{#2jjk}
	\varevalk{#2iil}
	\varevalk{#2ijl}
	\varevalk{#2jil}
	\varevalk{#2jjl}
}

\newcommand{\funcevalii}[2][]{
	\funcevali[#1]{#2}
	\funcevali{#2i}
	\funcevali{#2j}
}

\newcommand{\funcevalik}[2][]{
	\funcevali[#1]{#2}
	\funcevalk{#2i}
	\funcevalk{#2j}
}

\newcommand{\funcevaliik}[2][]{
	\funcevalii[#1]{#2}
	\funcevalk{#2ii}
	\funcevalk{#2ij}
	\funcevalk{#2ji}
	\funcevalk{#2jj}
}

\newcommand{\funcevaliikk}[2][]{
	\funcevaliik[#1]{#2}
	\funcevalk{#2iik}
	\funcevalk{#2ijk}
	\funcevalk{#2jik}
	\funcevalk{#2jjk}
	\funcevalk{#2iil}
	\funcevalk{#2ijl}
	\funcevalk{#2jil}
	\funcevalk{#2jjl}
}

\newcommand{\funcevalX}[3]{
	\expandafter\newcommand\csname #1v#2\endcsname{\csname #1v\endcsname{#3}}
}

\newcommand{\funcevaliX}[3]{
	\expandafter\newcommand\csname #1v#2\endcsname{\csname #1v\endcsname{#3}}
	\expandafter\newcommand\csname #1v#2e\endcsname[1]{\csname #1ve\endcsname{##1}{#3}}
	\expandafter\newcommand\csname #1v#2i\endcsname{\csname #1vi\endcsname{#3}}
	\expandafter\newcommand\csname #1v#2j\endcsname{\csname #1vj\endcsname{#3}}
}

\newcommand{\funcevalikX}[3]{
	\expandafter\newcommand\csname #1v#2\endcsname{\csname #1v\endcsname{#3}}
	\expandafter\newcommand\csname #1v#2e\endcsname[2]{\csname #1ve\endcsname{##1}{##2}{#3}}
	\expandafter\newcommand\csname #1v#2i\endcsname[1]{\csname #1vi\endcsname{##1}{#3}}
	\expandafter\newcommand\csname #1v#2j\endcsname[1]{\csname #1vj\endcsname{##1}{#3}}
	\expandafter\newcommand\csname #1v#2ik\endcsname{\csname #1vik\endcsname{#3}}
	\expandafter\newcommand\csname #1v#2il\endcsname{\csname #1vil\endcsname{#3}}
	\expandafter\newcommand\csname #1v#2jk\endcsname{\csname #1vik\endcsname{#3}}
	\expandafter\newcommand\csname #1v#2jl\endcsname{\csname #1vjl\endcsname{#3}}
}

\newcommand{\newvari}[2]{
	\newvar{#1}{#2}
	\newvars{#1e}{\csname #1\endcsname}
	\varevali[e]{#1}
}

\newcommand{\newvark}[2]{
	\newvar{#1}{#2}
	\newvars{#1e}{\csname #1\endcsname}
	\varevalk[e]{#1}
}

\newcommand{\newvarik}[2]{
	\newvar{#1}{#2}
	\newvarss{#1e}{\csname #1\endcsname}
	\varevalik[e]{#1}
}

\newcommand{\newvarii}[2]{
	\newvar{#1}{#2}
	\newvarss{#1e}{\csname #1\endcsname}
	\varevalii[e]{#1}
}

\newcommand{\newvariik}[2]{
	\newvar{#1}{#2}
	\newvarsss{#1e}{\csname #1\endcsname}
	\varevaliik[e]{#1}
}

\newcommand{\newvariis}[2]{
	\newvar{#1}{#2}
	\newvarsss{#1e}{\csname #1\endcsname}
	\varevaliis[e]{#1}
}

\newcommand{\newvariikk}[2]{
	\newvar{#1}{#2}
	\newvarssss{#1e}{\csname #1\endcsname}
	\varevaliikk[e]{#1}
}

\newcommand{\newfunc}[2]{
	\newvar{#1}{#2}
	\newfunconly{#1v}{\csname #1\endcsname}
}

\newcommand{\newfunci}[2]{
	\newvari{#1}{#2}
	\newfunconly{#1v}{\csname #1\endcsname}
	\newfuncs{#1ve}{\csname #1\endcsname}
	\funcevali[e]{#1v}
}

\newcommand{\newfunck}[2]{
	\newvark{#1}{#2}
	\newfunconly{#1v}{\csname #1\endcsname}
	\newfuncs{#1ve}{\csname #1\endcsname}
	\funcevalk[e]{#1v}
}

\newcommand{\newfuncik}[2]{
	\newvarik{#1}{#2}
	\newfunconly{#1v}{\csname #1\endcsname}
	\newfuncss{#1ve}{\csname #1\endcsname}
	\funcevalik[e]{#1v}
}

\newcommand{\newfuncii}[2]{
	\newvarii{#1}{#2}
	\newfunconly{#1v}{\csname #1\endcsname}
	\newfuncss{#1ve}{\csname #1\endcsname}
	\funcevalii[e]{#1v}
}

\newcommand{\newfunciik}[2]{
	\newvariik{#1}{#2}
	\newfunconly{#1v}{\csname #1\endcsname}
	\newfuncsss{#1ve}{\csname #1\endcsname}
	\funcevaliik[e]{#1v}
}

\newcommand{\newfunciikk}[2]{
	\newvariikk{#1}{#2}
	\newfunconly{#1v}{\csname #1\endcsname}
	\newfuncssss{#1ve}{\csname #1\endcsname}
	\funcevaliikk[e]{#1v}
}

\theoremstyle{plain}
\newtheorem{theorem}{Theorem}[section]
\newtheorem{conjecture}{Conjecture}
\newtheorem{corollary}[theorem]{Corollary}
\newtheorem{lemma}[theorem]{Lemma}
\newtheorem{proposition}[theorem]{Proposition}

\newtheorem*{mainresult}{Main Result}

\theoremstyle{definition}

\newtheorem{definition}{Definition}

\theoremstyle{remark}

\newcommand{\ate}{\tau}
\newcommand{\eate}{\hat{\ate}}

\newcommand{\neigh}[1]{N(#1)}
\newcommand{\neighi}{\neigh{i}}
\newcommand{\extneigh}[1]{\widetilde{N}(#1)}
\newcommand{\extneighi}{\extneigh{i}}

\newcommand{\model}{\mathcal{M}(G)}
\newcommand{\momodel}[1]{\mathcal{M}(G,\ate,#1)}
\newcommand{\momodelq}{\momodel{q}}

\newcommand{\expset}[1]{\Delta_{#1}}
\newcommand{\expseti}{\expset{i}}
\newcommand{\expmap}[1]{h_{#1}}
\newcommand{\expmapi}{\expmap{i}}

\newcommand{\risk}{\mathcal{R}}
\newcommand{\mirisk}{\risk_{\maxindset}}

\newcommand{\cgraph}{\mathcal{H}}
\newcommand{\cvert}{V_{\cgraph}}
\newcommand{\cedges}{E_{\cgraph}}

\newcommand{\allgraphs}{\mathbb{H}}

\newcommand{\indset}{\mathcal{I}}
\newcommand{\maxindset}{\mathcal{MI}}

\newcommand{\lamH}{\lambda(\cgraph)}

\newcommand{\davgH}{\davg(\cgraph)}
\newcommand{\dmaxH}{\dmax(\cgraph)}

\newcommand{\ccgraph}[1]{\cgraph( \cvert \setminus #1 )}
\newcommand{\ccgraphS}{\ccgraph{S}}
 
\newcommand{\cI}{\mathcal{I}}
\newcommand{\cH}{\mathcal{H}}
\newcommand{\cM}{\mathcal{M}}

\newcommand{\cF}{\mathcal{F}}
\newcommand{\cR}{\mathcal{R}}

\newcommand{\cU}{\mathcal{U}}
\newcommand{\cB}{\mathcal{B}}

\newcommand{\cA}{\mathcal{A}}
\newcommand{\cG}{\mathcal{G}}

\newcommand{\tv}{d_\mathrm{TV}}

\newcommand{\R}{\mathbb{R}}
\newcommand{\F}{\mathbb{F}}

\newcommand{\contrastz}[2]{\vec{z}_{#1,(#2)}}

\newcommand{\design}{\mathcal{D}}

\newcommand{\constant}{8.92}

\newcommand{\riskconst}{2.64 \cdot 10^{-6}}

\newcommand{\lowerboundconst}{1.36 \cdot 10^{-3}}

\newcommand{\localconst}{6.6 \cdot 10^{-7}}

\newcommand{\dmax}{d_{\mathrm{max}}}
\newcommand{\dmin}{d_{\mathrm{min}}}
\newcommand{\dhar}{d_{\mathrm{harm}}}
\newcommand{\davg}{d_{\mathrm{avg}}}

\definecolor{myred}{RGB}{234,108,123}
\definecolor{myblue}{RGB}{113,154,225}

\usepackage{expcmd}
\newcommand\mainref\ref
\newcommand\suppref\ref


\title{A Design-Based Minimax Theory for Network Experiments}
\author[1]{Vardis Kandiros}
\author[1]{Christopher Harshaw}
\author[2]{Fredrik S{\"a}vje}
\affil[1]{Columbia University}
\affil[2]{Uppsala University}
\date{\today}

\begin{document}
	
	\makeatletter%
	
	\begin{NoHyper}\gdef\@thefnmark{}\@footnotetext{\hspace{-1em}We thank
James M. Robins
for insightful discussions which helped to shape this work.
The authors gratefully acknowledge support from the following agencies:
NSF Grant MMS-2316335;
Jan Wallander, Tom Hedelius \& Tore Browaldh foundations Grant P25-0067;
Swedish Research Council Grant 2025-04794.
Part of this research was conducted at the Isaac Newton Institute for Mathematical Sciences.}\end{NoHyper}%
	\makeatother%
	
	\maketitle
	\thispagestyle{empty}
	
	\begin{abstract}
		Network experiments are used throughout the social and medical sciences to investigate causal effects under the presence of interference.
While a large body of work has developed improved statistical procedures, the fundamental limits of statistical estimation in these settings is less well understood.
In this paper, we develop and investigate a design-based theory of minimax risk for network experiments under an arbitrary neighborhood interference model.
Our notion of minimax risk describes the optimal precision among all statistical procedures for investigating a particular causal effect on the observed interference network.
We show that the minimax risk is a function of the corresponding conflict graph, which captures inherent unobservability of estimand-relevant potential outcomes given the observed interference network.
Our main contribution is a series of upper and lower bounds on the minimax rate in terms of local and global connectivity properties of the conflict graph.
To illustrate their utility, we apply these general results to obtain minimax analyses for two commonly studied effects: the direct treatment effect and global average treatment effect.

	\end{abstract}

	\newpage
	
	\pagenumbering{roman}
	
	\doparttoc 
	\faketableofcontents 
	\part{} 
	\parttoc 
	
	\newpage

	\pagenumbering{arabic}
	
	\section{Introduction} \label{sec:intro}

Network interference experiments are used in the social and medical sciences to investigate spillover effects and other causal interactions between units.
An important problem for the execution of such an experiment is the choice of the experimental design and the causal effect estimator, which we collectively refer to as the \emph{statistical procedure}.
The choice is crucial because it determines how much can be learned from the experiment.
An ill-advised choice can introduce both bias and imprecision, often rendering the estimator inconsistent unless the underlying network is very sparse.

A growing body of literature has sought to shine light on this question.
The typical line of attack is the development of a new statistical procedure for network experiments that can be shown to provide improved rates of convergence under suitable circumstances and restrictions on the underlying interference network.
Examples include experimental designs based on cluster randomization \citep{Ugander2013Graph, Ugander2023Randomized, Viviano2026Causal}, independent sets \citep{Karwa2018Systematic, Jagadeesan2020Designs, Fatemi2020Network},
and conflict-resolution schemes \citep{kandiros2024conflict}.

While significant progress has been made with this approach, the literature is stifled by a lack of a theoretical benchmark: what are the statistically optimal rates of estimation for network experiments?
The question of rate optimality is not only of theoretical interest, but it also has practical implications.
Without an understanding of the fundamental statistical limits of what can be learned from these experiments, both statisticians and practitioners remain unaware of how near or far existing statistical procedures are from being optimal.
If it is found that existing procedures are far from the fundamental optimal limit, we should be incentivized to continue working on new procedures, possibly trying novel approaches.

In this paper, we develop and investigate a design-based theory of minimax optimality for network experiments.
The proposed notion of minimax optimality describes the best precision among all statistical procedures for investigating a particular causal effect $\ate$ on an observed interference network $G$, providing a fundamental statistical limit of the empirical problem at hand.
Our notion of minimax optimality has two key features.
First, optimality is considered with respect to a single observed interference network $G$, considering all potential outcomes obeying the restrictions imposed by the network.
Second, optimality is with respect to all possible choices of experimental design and estimator without any restrictions.
We believe that these two features reflect the challenges faced by practitioners, who are often able to select any statistical procedure after measuring the interference network.
Moreover, this fills a gap in the design-based literature; to the best of our knowledge, minimax analyses which account for both experimental design and estimator have not previously been investigated in design-based causal inference.

The main contribution of the paper is an analysis of the minimax rate of causal effect estimation under the assumption of arbitrary neighborhood interference (ANI), where the outcome of a unit depends arbitrarily on the treatment of its neighbors.
Central to our analysis is the notion of an exposure-level conflict graph, which is closely related to the unit-level conflict graph described by \citet{kandiros2024conflict}.
The vertices in the conflict graph $\cgraph$ are all potential outcomes relevant for the estimand $\ate$ at hand, and an edge is drawn between two vertices if the corresponding potential outcomes cannot be observed simultaneously given the interference network $G$.
We show that the minimax risk of the network experiment problem is in fact a function of the conflict graph $\cgraph$, albeit a complicated one.
Thus, the central question becomes: what aspects of the conflict graph determine the minimax rate?

Our primary analysis appears in a series of lower and upper bounds on the minimax risk presented throughout the paper.
We collect these results here, presented in a slightly simplified version.

\begin{mainresult}[Informal] \label{thm:informal-main-result}
	Let $G$ be an interference network, $\ate$ be a causal estimand, and $\cgraph$ be the corresponding conflict graph.
	The minimax risk for estimating $\ate$ on network $G$ is bounded as
	\[
	\max\paren[\Bigg]{\frac{1}{|\cI(\cH)|}, \frac{\sqrt{d^*(\cH)}}{n} } 
	\lesssim 
	\inf_{\design,\eate} \sup_{y} \E[\Big]{\paren[\big]{\eate - \ate}^2} 
	\lesssim
	\frac{\min_{S} Q(S,\cH)}{n}
	\enspace,
	\]
	where the infimum is taken over all possible designs $\design$ and estimators $\eate$, and the supremum is taken over all potential outcomes $y$ having bounded second moments and satisfying the network interference assumption.
	The expectation is taken with respect to the experimental design $\design$.
\end{mainresult}

The lower bounds on minimax risk depend on two quantities describing the conflict graph: its largest independent set $\indset(\cgraph)$ and its critical degree $d^*(\cgraph)$, which represent global and local connectivity information of the graph.
More precisely, $\indset(\cgraph)$ is the largest subset of vertices for which no pair is adjacent in $\cgraph$ and $d^*(\cgraph)$ is the largest integer $k$ such that at least $\sqrt{k}$ vertices have degree at least $k$.
The upper bound is the solution to a regularized induced subgraph problem, where for any subset $S$ of vertices, the objective function $Q(S,\cgraph) = \lambda(\cgraph \setminus S) + |S|$ has two terms: the largest eigenvalue of the induced subgraph obtained by removing $S$ and the size of $S$.
This upper bound reflects a bias-variance trade-off and improves over the rate attained by \citet{kandiros2024conflict}.
The full analysis considers general moment restrictions, and it uncovers an interplay between the strength of the moment restrictions and degree heterogeneity for the minimax rates, where stronger moment restrictions can accommodate more heterogeneity.

The primary technical innovation facilitating this work is the development of techniques to derive statistical lower bounds in a design-based setting.
Our approach begins by adopting Le Cam's method within the design-based framework, where a mixture distribution over potential outcome functions is constructed to certify the statistical lower bound.
In the setting of network experiments, several novel challenges arise necessitating the construction of mixture distributions which are both network-specific and design-adaptive.
The lower bound techniques we develop may be of independent interest, and may be useful in other high-dimensional network problems.

\subsection{Illustration of Main Results}\label{sec:applications}

To illustrate the utility of our general minimax rate analyses, we apply the main result to two causal estimands previously considered in the literature.
In what follows, we denote the minimax rate as $\risk(G,\ate)$ for a network $G$ and causal effect $\ate$.
For simplicity, we here focus on second moment restrictions and present results without considering the aforementioned bias-variance trade-off, effectively setting $S = \emptyset$.
All formal definitions are described in Section~\ref{sec:minimax_theory} and proofs of the corollaries in this subsection appear in Section~\ref{sec:applications_appendix} of the appendix.

We first consider estimation of the Direct Treatment Effect (DTE).
The DTE estimand is the effect of treating a single experimental subject when all other subjects are assigned control treatment, averaged over all subjects.
The estimand isolates the effect that treatment has directly on the treated unit, removing any potential spillover and interference effects.
The following corollary uses the main results above to establish bounds on the minimax optimal rate of estimating DTE in terms of the underlying interference graph $G$.

\begin{restatable}{corollary}{dterate} \label{corollary:direct-rate}
	The minimax rate of estimating the Direct Treatment Effect $\ate_{\mathrm{DTE}}$ on a network $G$ is bounded as
	\[
	\max\paren[\Bigg]{ \frac{1}{|\indset(G)|} , \frac{\sqrt{d_{\mathrm{avg}}(G)}}{n} } 
	\lesssim 
	\risk(G,\ate_{\mathrm{DTE}}) 
	\lesssim
	\frac{\lambda(G)}{n}
	\enspace,
	\]
	where $\indset(G)$ is the maximum independent set in $G$, $\davg(G)$ is its average degree, and $\lambda(G)$ is the largest eigenvalue of its adjacency matrix.
\end{restatable}

The bounds are similar to those that appeared in the general main results, but with the underlying interference network $G$ in place of the conflict graph $\cH$.
This simplification is possible because $\cH$ is tightly coupled with $G$ for the DTE.

The minimax rate is precisely characterized when the upper and lower bounds match; however, there are graphs $G$ for which these bounds do not match.
In Section~\ref{sec:comparisons}, we show that both bounds are tight even within the restricted class of $d$-regular graphs, in the sense that there exist graphs such that the true minimax rate coincides with the lower bound, and other graphs such that it coincides with the upper bound.
In addition, we show that these bounds characterize the minimax risk for random $d$-regular graphs, up to logarithmic factors.

We next consider estimation of the Global Average Treatment Effect (GATE).
The GATE estimand is the effect of treating all subjects versus treating no one.
The estimand captures both the direct effect of treatment and spillover effects.
The following corollary uses the main results above to establish bounds for the GATE in terms of the two-hop network $G^2$, which is obtained by connecting all subjects that are within distance $2$ in the original interference network $G$.

\begin{restatable}{corollary}{gaterate} \label{corollary:gate-rate}
	The minimax rate of estimating the Global Average Treatment Effect $\ate_{\mathrm{GATE}}$ on a network $G$ is bounded as
	\[
	\frac{\sqrt{\davg(G^2)}}{n}
	\lesssim 
	\risk(G,\ate_{\mathrm{GATE}})
	\lesssim
	\frac{\lambda(G^2)}{n}
	\enspace,
	\]
	where $G^2$ is the two-hop interference graph, $\davg(G^2)$ is the average degree of the two-hop graph, and $\lambda(G^2)$ is the largest eigenvalue of its adjacency matrix.
\end{restatable}

Because the two-hop network $G^2$ is typically much more densely connected than the original network $G$, the minimax analyses of Corollary~\ref{corollary:gate-rate} aligns with the commonly held belief that the DTE is easier to estimate than the GATE.
In Section~\ref{sec:comparisons}, we formalize this intuition by showing that the minimax rate of estimating the DTE is at most that of estimating the GATE for random $d$-regular graphs.

The fact that the derived upper and lower bounds on the minimax rates for network experiments do not always match, in the sense that they are the same up to constant terms, means that they do not fully characterize the fundamental limit of the statistical problem.
Although the minimax risk is a function of the conflict graph $\cgraph$, this exact relationship depends on the structure of maximal independent sets of $\cgraph$, which are likely too complex to be succinctly described.
In Section~\ref{sec:conjecture}, we formalize this possibility through a set of well-motivated conjectures about the computational hardness of deriving sharper analyses of the minimax risk and discuss their practical consequences.
Nevertheless, we believe the bounds in this paper provide an important step towards the understanding of network experiments as a statistical problem.

	\subsection{Related Work} \label{sec:related_work}

We consider experimental design and estimation in network experiments under \emph{arbitrary neighborhood interference}, allowing the outcome of a unit to depend on the treatments of its neighbors in a known graph.
This is an instance of the design-based exposure mapping framework developed by \citet{aronow2017estimating}, which generalizes the ideas developed by \citet{Sobel2006What} and \citet{hudgens2008toward} to allow for outcomes to depend on arbitrary summaries of the treatment vector.
The concept of effective treatments described by \citet{manski2013identification} is an alternative formulation.

Experimental design under arbitrary neighborhood interference has been considered previously.
 \citet{Ugander2013Graph} and \citet{Ugander2023Randomized} consider estimation of the Global Average Treatment Effect (GATE) under uniformly bounded outcomes and describe designs that split the graph in clusters used for the assignment of treatment.
 The work closest to the current paper is \citet{kandiros2024conflict}, who consider arbitrary contrastive effects under bounded second outcome moments and describe an experimental design based on a conflict resolution scheme.
Compared to this previous work, our contribution is to derive lower bounds on the minimax rate, which has previously been absent from the literature, and improved upper bounds.
We also formalize the study of the minimax rate in the setting of network experiments.

Several authors have considered network experiments under restrictions on the potential outcomes beyond those imposed by the neighborhood interference model itself.
For example, \citet{Toulis2013Estimation} describe two designs for assigning treatments in graphs under a linear outcome model.
A long line of work has explored other types of additional outcome assumptions in network experiments \citep{Sussman2017,Fatemi2020Network,Jagadeesan2020Designs,Cai2024Independent,Viviano2026Causal}.
The purpose of these assumptions is to make potential outcomes that are irrelevant under the neighborhood interference model alone informative of the causal effect of interest, simplifying the experimental design and estimation problems.

Our concept of minimax risk and optimality differs from previously considered concepts in the experimental design literature.
Previously, analyses have typically fixed the estimator and considered a restricted class of designs \citep{Wu1981Robustness,Baird2018Optimal,Kallus2018Optimal,Leung2022Rate,Bojinov2023Design,Candogan2024Correlated}.
Some analyses instead restrict the design and consider a larger class of estimators \citep[see, e.g.,][]{CortezRodriguez2023Exploiting}. 
This type of analysis demonstrates optimality within the restricted class of designs or estimators, but not optimality in a universal sense.
We consider the choice of estimator and design simultaneously without restrictions, making the analysis informative of universal optimality.
Concepts of optimality not building on minimax performance have also been considered.
For example, \citet{Basse2018Model} consider designs that are optimal with respect to a generative working model, and \citet{karwa2023admissibility} consider admissibility of linear exposure effect estimators under interference.

Within the superpopulation framework for causal inference, statistical optimality and minimax analysis have been studied extensively.
The data generating process is beyond the researcher's control here, so optimality is considered only over possible estimators.
The predominate focus in this literature is on \emph{local minimax risk} \citep{van2000asymptotic}, which considers the best possible asymptotic variance that can be achieved by an asymptotically unbiased estimator.
In semiparametric efficiency theory, this is done using the efficient influence curve \citep{bickel1993efficient}. 
The approach is local because it considers perturbations of the parameter of interest locally in the parameter space.
We instead use a global notion of minimax risk, considering the worst-case performance of a statistical procedure for all potential outcome functions in some class of functions.
The global notion is suitable to capture the optimal \emph{rate} of estimating an effect, without consideration for the exact constants involved.
A global notion of minimax risk has been considered also in the superpopulation setting to study optimal rates \citep[see, e.g.,][]{robins2008higher,rotnitzky2021characterization,foster2023orthogonal,Hu2023Off,kennedy2024minimax}. 

Within the survey sampling literature, several minimax optimality results for finite populations have been developed \citep{RINOTT2009523}.
For example, \citet{aggarwal1959bayes,bickel1981minimax} established that the sample mean is minimax optimal under the constraint that the variance of the outcomes is bounded, and \citet{hodges1982minimax} showed that a slightly biased version of the sample mean is minimax optimal under the constraint that all outcomes are uniformly bounded.
More recently, \citet{aronow2026minimax} derived lower bounds in a setting with unequal inclusion probabilities under uniformly bounded outcomes.
While the underlying structure of the statistical problems are similar, the techniques for establishing minimax optimality in survey sampling are not directly applicable to causal inference problems.
The survey sampling techniques rely on classical decision theory arguments based on symmetry of the problem with respect to the unit, allowing the proof to be divided into two parts, establishing the optimal design separately from the optimal estimator \citep{blackwell1979theory,gabler2012minimax}.
No such symmetry exists in the causal inference setting, and we must consider the choice of design and estimator simultaneously.
For this reason, there is no obvious way to apply these approaches to the problems we consider in this paper.
The techniques we describe here are developed from first principles and unrelated to those used in the survey sampling literature.

	\section{Minimax Theory for Network Experiments}\label{sec:minimax_theory}

\subsection{Preliminaries: Model and Estimands} \label{sec:preliminaries}

We consider an experimental study with $n$ subjects indexed by integers $\setb{1, 2, \ldots, n} \triangleq [n]$.
The experimenter selects a random intervention $z$ from a set of interventions $\Omega$ and the randomized selection mechanism is referred to as the experimental design.
Formally, the \emph{interventions} are a measurable space $(\Omega, \cF)$ and the \emph{experimental design} is a probability measure $\design : \cF \to [0,1]$.
We focus on the setting where intervention is a binary treatment to each experimental subject, i.e. $z = (z_1, z_2, \dots z_n)$ with $z_i \in \setb{0,1}$ so that $\Omega = \setb{0,1}^n$.

Each subject $i \in [n]$ has an associated \emph{potential outcome function} $y_i : \Omega \to \R$ which describes its responses under each intervention.
For notational convenience, we collect these $n$ functions into an \emph{entire potential outcome function} $y : \Omega \to \R^n$ defined as $y(z) = (y_1(z), \dots ,y_n(z))$.
In the randomized experiment, the experimenter draws a random intervention $Z$ according to the design $\design$ and observes the $n$ outcomes $y_1(Z), \dots ,y_n(Z)$.

In order to facilitate meaningful inference of causal effects, the experimenter must place further restrictions on the potential outcome functions.
A common choice is the arbitrary neighborhood interference (ANI) model, which depends on an observed interference network.
We presume that the experimenter has access to an undirected network $G = (V,E)$ whose vertices $V$ are experimental subjects with edge set $E$.
For a subject $i \in [n]$, we define its \emph{neighborhood} in the network to be the other subjects to which it is connected: $\neighi =\setb{ j \in [n] : (i,j) \in E }$.
For  notational convenience, we define the \emph{extended neighborhood} as $\extneighi = N(i) \cup \{i\}$, which is just a unit's neighbors and itself.
The ANI model stipulates that a subject's potential outcome function can depend arbitrarily on treatments assigned to subjects in its extended neighborhood, but is not affected by treatment assigned to other subjects.

\begin{definition}
	The \emph{arbitrary neighborhood interference} model $\model$ contains all functions $y : \Omega \to \Reals^n$ satisfying the following condition:
	$y_i(z) = y_i(z')$
	for all subjects $i \in [n]$ and interventions $z, z' \in \Omega$ with
	$z_j = z_j'$ 
	for all
	$j \in \extneigh{i}$.
\end{definition}

The ANI model is an example of an exposure mapping model \citep{aronow2017estimating}.
It will sometimes be convenient to work with the underlying exposure mappings, so we define them now.
For each subject $i \in [n]$, its \emph{exposure set} $\expseti$ is defined to contain all subsets of its extended neighborhood, i.e. $\expseti = \mathcal{P}(\extneighi)$.
The exposure mapping $\expmapi : \Omega \to \expseti$ maps an intervention $z$ to the subset of extended neighbors which were treated under $z$, i.e.
$
\expmapi(z) = \setb{ j \in \extneighi : z_j = 1 }
$.
In an exposure mapping model, a unit's potential outcome function is constant across level sets of the exposure mapping, i.e. if $\expmapi(z) = \expmapi(z')$ then $y_i(z) = y_i(z')$.
In this way, we can see that the ANI model is an exposure mapping model, where the exposure is the specific configuration of extended neighbors which are treated.
In line with the previous literature, we overload the potential outcome notation and write $y_i(e)$ to denote the potential outcome that would be observed for unit $i$ under exposure $e \in \expseti$.

We consider the class of contrastive causal effects, which are defined as the average over subjects of the contrast between two different potential outcomes.
Formally, a \emph{contrastive effect} $\ate$ is a linear functional $\ate : \model \to \Reals$ which is specified by $2n$ unit-specific exposures $\setb{ (e_{i,1}, e_{i,0}) }_{i=1}^n$ and defined as
\[
\ate(y) = \frac{1}{n} \sum_{i=1}^n y_i(e_{i,1}) - y_i(e_{i,0})
\enspace.
\]
To avoid notational clutter, we do not parameterize $\ate$ by the choice of exposures.
We review two of the most well-studied examples of contrastive effects below:

\begin{figure}[t!]
	\centering
	\begin{subfigure}[t]{0.45\textwidth}
		\centering
		\resizebox{0.8\linewidth}{!}{%
            \tikzset{every picture/.style={line width=0.75pt}} 

\begin{tikzpicture}[x=0.75pt,y=0.75pt,yscale=-1,xscale=1]

\draw  [fill=myred  ,fill opacity=1 ][line width=1.5]  (147,144) .. controls (147,130.19) and (158.19,119) .. (172,119) .. controls (185.81,119) and (197,130.19) .. (197,144) .. controls (197,157.81) and (185.81,169) .. (172,169) .. controls (158.19,169) and (147,157.81) .. (147,144) -- cycle ;
\draw  [fill=myred  ,fill opacity=1 ][line width=1.5]  (102,69.5) .. controls (102,60.94) and (108.94,54) .. (117.5,54) .. controls (126.06,54) and (133,60.94) .. (133,69.5) .. controls (133,78.06) and (126.06,85) .. (117.5,85) .. controls (108.94,85) and (102,78.06) .. (102,69.5) -- cycle ;
\draw  [fill=myred  ,fill opacity=1 ][line width=1.5]  (207,69.5) .. controls (207,60.94) and (213.94,54) .. (222.5,54) .. controls (231.06,54) and (238,60.94) .. (238,69.5) .. controls (238,78.06) and (231.06,85) .. (222.5,85) .. controls (213.94,85) and (207,78.06) .. (207,69.5) -- cycle ;
\draw  [fill=myred  ,fill opacity=1 ][line width=1.5]  (113,226.5) .. controls (113,217.94) and (119.94,211) .. (128.5,211) .. controls (137.06,211) and (144,217.94) .. (144,226.5) .. controls (144,235.06) and (137.06,242) .. (128.5,242) .. controls (119.94,242) and (113,235.06) .. (113,226.5) -- cycle ;
\draw  [fill=myred  ,fill opacity=1 ][line width=1.5]  (75,150.5) .. controls (75,141.94) and (81.94,135) .. (90.5,135) .. controls (99.06,135) and (106,141.94) .. (106,150.5) .. controls (106,159.06) and (99.06,166) .. (90.5,166) .. controls (81.94,166) and (75,159.06) .. (75,150.5) -- cycle ;
\draw  [fill=myred  ,fill opacity=1 ][line width=1.5]  (237,149.5) .. controls (237,140.94) and (243.94,134) .. (252.5,134) .. controls (261.06,134) and (268,140.94) .. (268,149.5) .. controls (268,158.06) and (261.06,165) .. (252.5,165) .. controls (243.94,165) and (237,158.06) .. (237,149.5) -- cycle ;
\draw  [fill=myred  ,fill opacity=1 ][line width=1.5]  (199,226.5) .. controls (199,217.94) and (205.94,211) .. (214.5,211) .. controls (223.06,211) and (230,217.94) .. (230,226.5) .. controls (230,235.06) and (223.06,242) .. (214.5,242) .. controls (205.94,242) and (199,235.06) .. (199,226.5) -- cycle ;
\draw [line width=1.5]    (126,82) -- (155,125) ;
\draw [line width=1.5]    (215,82) -- (187,124) ;
\draw [line width=1.5]    (160,166) -- (135,213) ;
\draw [line width=1.5]    (106,150.5) -- (147,144) ;
\draw [line width=1.5]    (237,149.5) -- (197,144) ;
\draw [line width=1.5]    (182,166) -- (205,215) ;
\draw  [fill=myblue  ,fill opacity=1 ][line width=1.5]  (441,143) .. controls (441,129.19) and (452.19,118) .. (466,118) .. controls (479.81,118) and (491,129.19) .. (491,143) .. controls (491,156.81) and (479.81,168) .. (466,168) .. controls (452.19,168) and (441,156.81) .. (441,143) -- cycle ;
\draw  [fill=myblue  ,fill opacity=1 ][line width=1.5]  (396,68.5) .. controls (396,59.94) and (402.94,53) .. (411.5,53) .. controls (420.06,53) and (427,59.94) .. (427,68.5) .. controls (427,77.06) and (420.06,84) .. (411.5,84) .. controls (402.94,84) and (396,77.06) .. (396,68.5) -- cycle ;
\draw  [fill=myblue  ,fill opacity=1 ][line width=1.5]  (501,68.5) .. controls (501,59.94) and (507.94,53) .. (516.5,53) .. controls (525.06,53) and (532,59.94) .. (532,68.5) .. controls (532,77.06) and (525.06,84) .. (516.5,84) .. controls (507.94,84) and (501,77.06) .. (501,68.5) -- cycle ;
\draw  [fill=myblue  ,fill opacity=1 ][line width=1.5]  (407,225.5) .. controls (407,216.94) and (413.94,210) .. (422.5,210) .. controls (431.06,210) and (438,216.94) .. (438,225.5) .. controls (438,234.06) and (431.06,241) .. (422.5,241) .. controls (413.94,241) and (407,234.06) .. (407,225.5) -- cycle ;
\draw  [fill=myblue  ,fill opacity=1 ][line width=1.5]  (369,149.5) .. controls (369,140.94) and (375.94,134) .. (384.5,134) .. controls (393.06,134) and (400,140.94) .. (400,149.5) .. controls (400,158.06) and (393.06,165) .. (384.5,165) .. controls (375.94,165) and (369,158.06) .. (369,149.5) -- cycle ;
\draw  [fill=myblue  ,fill opacity=1 ][line width=1.5]  (531,148.5) .. controls (531,139.94) and (537.94,133) .. (546.5,133) .. controls (555.06,133) and (562,139.94) .. (562,148.5) .. controls (562,157.06) and (555.06,164) .. (546.5,164) .. controls (537.94,164) and (531,157.06) .. (531,148.5) -- cycle ;
\draw  [fill=myblue  ,fill opacity=1 ][line width=1.5]  (493,225.5) .. controls (493,216.94) and (499.94,210) .. (508.5,210) .. controls (517.06,210) and (524,216.94) .. (524,225.5) .. controls (524,234.06) and (517.06,241) .. (508.5,241) .. controls (499.94,241) and (493,234.06) .. (493,225.5) -- cycle ;
\draw [line width=1.5]    (420,81) -- (449,124) ;
\draw [line width=1.5]    (509,81) -- (481,123) ;
\draw [line width=1.5]    (454,165) -- (429,212) ;
\draw [line width=1.5]    (400,149.5) -- (441,143) ;
\draw [line width=1.5]    (531,148.5) -- (491,143) ;
\draw [line width=1.5]    (476,165) -- (499,214) ;

\draw (166,129) node [anchor=north west][inner sep=0.75pt]   [align=left] {{\huge $\displaystyle i$}};
\draw (460,128) node [anchor=north west][inner sep=0.75pt]   [align=left] {{\huge $\displaystyle i$}};
\draw (151,245) node [anchor=north west][inner sep=0.75pt]  [font=\huge] [align=left] {$\displaystyle e_{i,1}$};
\draw (451,245) node [anchor=north west][inner sep=0.75pt]  [font=\huge] [align=left] {$\displaystyle e_{i,0}$};

\end{tikzpicture}%
        }
		\caption{GATE Exposures}
	\end{subfigure}%
	\hfill
	\begin{subfigure}[t]{0.45\textwidth}
		\centering
		\resizebox{0.8\linewidth}{!}{%
            \tikzset{every picture/.style={line width=0.75pt}} 

\begin{tikzpicture}[x=0.75pt,y=0.75pt,yscale=-1,xscale=1]

\draw  [fill=myred   ,fill opacity=1 ][line width=1.5]  (151,144) .. controls (151,130.19) and (162.19,119) .. (176,119) .. controls (189.81,119) and (201,130.19) .. (201,144) .. controls (201,157.81) and (189.81,169) .. (176,169) .. controls (162.19,169) and (151,157.81) .. (151,144) -- cycle ;
\draw  [fill=myblue   ,fill opacity=1 ][line width=1.5]  (106,69.5) .. controls (106,60.94) and (112.94,54) .. (121.5,54) .. controls (130.06,54) and (137,60.94) .. (137,69.5) .. controls (137,78.06) and (130.06,85) .. (121.5,85) .. controls (112.94,85) and (106,78.06) .. (106,69.5) -- cycle ;
\draw  [fill=myblue   ,fill opacity=1 ][line width=1.5]  (211,69.5) .. controls (211,60.94) and (217.94,54) .. (226.5,54) .. controls (235.06,54) and (242,60.94) .. (242,69.5) .. controls (242,78.06) and (235.06,85) .. (226.5,85) .. controls (217.94,85) and (211,78.06) .. (211,69.5) -- cycle ;
\draw  [fill=myblue   ,fill opacity=1 ][line width=1.5]  (117,226.5) .. controls (117,217.94) and (123.94,211) .. (132.5,211) .. controls (141.06,211) and (148,217.94) .. (148,226.5) .. controls (148,235.06) and (141.06,242) .. (132.5,242) .. controls (123.94,242) and (117,235.06) .. (117,226.5) -- cycle ;
\draw  [fill=myblue   ,fill opacity=1 ][line width=1.5]  (79,150.5) .. controls (79,141.94) and (85.94,135) .. (94.5,135) .. controls (103.06,135) and (110,141.94) .. (110,150.5) .. controls (110,159.06) and (103.06,166) .. (94.5,166) .. controls (85.94,166) and (79,159.06) .. (79,150.5) -- cycle ;
\draw  [fill=myblue   ,fill opacity=1 ][line width=1.5]  (241,149.5) .. controls (241,140.94) and (247.94,134) .. (256.5,134) .. controls (265.06,134) and (272,140.94) .. (272,149.5) .. controls (272,158.06) and (265.06,165) .. (256.5,165) .. controls (247.94,165) and (241,158.06) .. (241,149.5) -- cycle ;
\draw  [fill=myblue   ,fill opacity=1 ][line width=1.5]  (203,226.5) .. controls (203,217.94) and (209.94,211) .. (218.5,211) .. controls (227.06,211) and (234,217.94) .. (234,226.5) .. controls (234,235.06) and (227.06,242) .. (218.5,242) .. controls (209.94,242) and (203,235.06) .. (203,226.5) -- cycle ;
\draw [line width=1.5]    (130,82) -- (159,125) ;
\draw [line width=1.5]    (219,82) -- (191,124) ;
\draw [line width=1.5]    (164,166) -- (139,213) ;
\draw [line width=1.5]    (110,150.5) -- (151,144) ;
\draw [line width=1.5]    (241,149.5) -- (201,144) ;
\draw [line width=1.5]    (186,166) -- (209,215) ;
\draw  [fill=myblue   ,fill opacity=1 ][line width=1.5]  (445,143) .. controls (445,129.19) and (456.19,118) .. (470,118) .. controls (483.81,118) and (495,129.19) .. (495,143) .. controls (495,156.81) and (483.81,168) .. (470,168) .. controls (456.19,168) and (445,156.81) .. (445,143) -- cycle ;
\draw  [fill=myblue   ,fill opacity=1 ][line width=1.5]  (400,68.5) .. controls (400,59.94) and (406.94,53) .. (415.5,53) .. controls (424.06,53) and (431,59.94) .. (431,68.5) .. controls (431,77.06) and (424.06,84) .. (415.5,84) .. controls (406.94,84) and (400,77.06) .. (400,68.5) -- cycle ;
\draw  [fill=myblue   ,fill opacity=1 ][line width=1.5]  (505,68.5) .. controls (505,59.94) and (511.94,53) .. (520.5,53) .. controls (529.06,53) and (536,59.94) .. (536,68.5) .. controls (536,77.06) and (529.06,84) .. (520.5,84) .. controls (511.94,84) and (505,77.06) .. (505,68.5) -- cycle ;
\draw  [fill=myblue   ,fill opacity=1 ][line width=1.5]  (411,225.5) .. controls (411,216.94) and (417.94,210) .. (426.5,210) .. controls (435.06,210) and (442,216.94) .. (442,225.5) .. controls (442,234.06) and (435.06,241) .. (426.5,241) .. controls (417.94,241) and (411,234.06) .. (411,225.5) -- cycle ;
\draw  [fill=myblue   ,fill opacity=1 ][line width=1.5]  (373,149.5) .. controls (373,140.94) and (379.94,134) .. (388.5,134) .. controls (397.06,134) and (404,140.94) .. (404,149.5) .. controls (404,158.06) and (397.06,165) .. (388.5,165) .. controls (379.94,165) and (373,158.06) .. (373,149.5) -- cycle ;
\draw  [fill=myblue   ,fill opacity=1 ][line width=1.5]  (535,148.5) .. controls (535,139.94) and (541.94,133) .. (550.5,133) .. controls (559.06,133) and (566,139.94) .. (566,148.5) .. controls (566,157.06) and (559.06,164) .. (550.5,164) .. controls (541.94,164) and (535,157.06) .. (535,148.5) -- cycle ;
\draw  [fill=myblue   ,fill opacity=1 ][line width=1.5]  (497,225.5) .. controls (497,216.94) and (503.94,210) .. (512.5,210) .. controls (521.06,210) and (528,216.94) .. (528,225.5) .. controls (528,234.06) and (521.06,241) .. (512.5,241) .. controls (503.94,241) and (497,234.06) .. (497,225.5) -- cycle ;
\draw [line width=1.5]    (424,81) -- (453,124) ;
\draw [line width=1.5]    (513,81) -- (485,123) ;
\draw [line width=1.5]    (458,165) -- (433,212) ;
\draw [line width=1.5]    (404,149.5) -- (445,143) ;
\draw [line width=1.5]    (535,148.5) -- (495,143) ;
\draw [line width=1.5]    (480,165) -- (503,214) ;

\draw (170,129) node [anchor=north west][inner sep=0.75pt]   [align=left] {{\huge $\displaystyle i$}};
\draw (464,128) node [anchor=north west][inner sep=0.75pt]   [align=left] {{\huge $\displaystyle i$}};
\draw (155,245) node [anchor=north west][inner sep=0.75pt]  [font=\huge] [align=left] {$\displaystyle e_{i,1}$};
\draw (455,245) node [anchor=north west][inner sep=0.75pt]  [font=\huge] [align=left] {$\displaystyle e_{i,0}$};

\end{tikzpicture}%
        }
		\caption{DTE Exposures}
	\end{subfigure}
	\caption{An illustration of the relevant exposures that define the global and direct effects. Red and blue correspond to treated and untreated units.
	}
	\label{fig:exposures}
\end{figure}
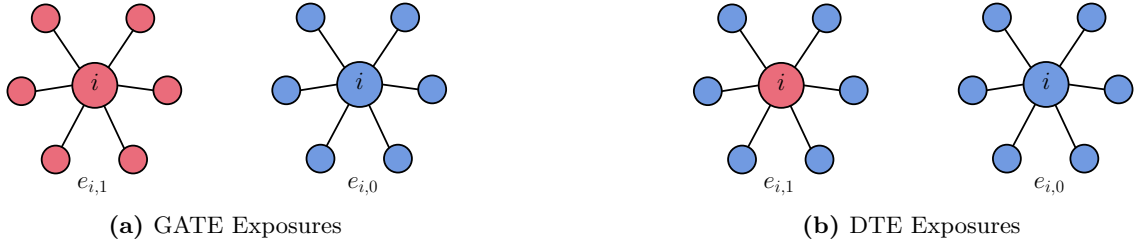

\begin{itemize}
	\item \textbf{Global Average Treatment Effect:}
	The \emph{global average treatment effect} (GATE) is defined by the following two exposures:
	$e_{i,1} = \extneighi$, which corresponds to all subjects in the extended neighborhood receiving treatment, and $e_{i,0} = \emptyset$, which corresponds to all subjects in the extended neighborhood receiving control. 
	The GATE is the effect of treating everyone versus treating no one.
	This may be relevant for policy makers which intend to roll out treatment to everyone.
	
	\item \textbf{Direct Treatment Effect:} 
	The \emph{direct treatment effect} (DTE) is defined by the following two exposures:
	$e_{i,1} = \setb{i}$, which corresponds to the subject $i$ receiving treatment while all of its neighbors remain untreated, and $e_{i,0} = \emptyset$, which corresponds to both the subject and its neighbors receiving control.
	The DTE is the effect of treatment when a subject is the only one to get it.
	This may be more relevant for an individual's decision if they believe that the crowd will not adopt treatment.
\end{itemize}

See Figure~\ref{fig:exposures} for visual depictions of these two effects and their constituent exposures.
Without loss of generality, we only consider estimands for which $e_{i,1} \neq e_{i,0}$.
To see this is without loss of generality, observe that if $e_{i,1} = e_{i,0}$ then $y_i(e_{i,1}) = y_i(e_{i,0})$ so that subject $i$ does not contribute to the causal effect and may be removed from consideration.

The causal estimand depends on the entire potential outcome function which is unknown to the experimenter even after the experiment is run, a challenge often referred to as the ``fundamental problem of causal inference'' \citep{holland1986statistics}.
For this reason, the experimenter must estimate the causal effect from the observed data.
Recall that the experimenter draws a random intervention $Z \in \Omega$ according to the design $\design$ and observes the $n$ outcomes $y_1(Z), \dots , y_n(Z)$.
Formally, we define an \emph{estimator} as a measurable function $\eate : \Omega \times \R^n \to \R$ which takes as input the realized intervention and observed outcomes and produces an estimate of the causal effect.
Because the researcher can select both the design $\design$ and estimator $\eate$, we refer to the pair $(\design, \eate)$ jointly as a \emph{statistical procedure}.

When analyzing the precision of a statistical procedure, it is necessary to place restrictions on the magnitude of the potential outcomes, because the precision can be arbitrarily poor if the magnitude of the potential outcomes is sufficiently large.
To this end, we impose further moment restrictions on the ANI model, which may be of arbitrary order.
For any $q \in [2, \infty]$, define the \emph{$q$th moment restricted ANI model} as
\[
\momodelq = \setb[\Big]{
	y \in \model: 
	\paren[\Big]{ \frac{1}{n} \sum_{i=1}^n \abs{ y_i(e_{i,k}) }^q  }^{1/q} \leq 1
	\text{ for } k \in \setb{0,1}
}
\enspace.
\]
The $q$th moment restricted model $\momodelq$ only contains potential outcome functions in the ANI model $\model$ for which the $q$th moments of the estimand-relevant potential outcomes are bounded.
In defining this restricted model, we have made two particular choices: we have assumed that the $q$th moments are bounded by $1$, and we have placed no restrictions on the potential outcomes which are not relevant for the estimand.
In Appendix~\ref{sec:supp_outcome_restrictions}, we show that both of these choices are without loss of generality when studying the minimax risk.
First, we show that replacing the moment restriction $1$ with any fixed constant $C > 0$ scales the minimax risk by a multiplicative factor of $C^2$, so that we may consider $C=1$ without loss of generality.
Second, we show that the minimax risk is unaffected if we were to place extreme restrictions on the estimand-irrelevant outcomes, so that they were effectively presumed to be known.
Because the outcomes that are irrelevant for the estimand provide no information about the causal effect under an ANI model, the extent to which they are restricted is irrelevant for the minimax risk.

\subsection{Minimax Risk for Estimating Causal Effects} \label{sec:def-minimax-risk}

In this paper, we use the mean squared error to measure the precision, or risk, of a statistical procedure.
Formally, the \emph{mean squared error} (MSE) of a design-estimator pair $(\design,\eate)$ for a causal effect $\ate$ under potential outcome function $y$ is defined as
\[
\textrm{MSE}(\design,\eate, \ate, y) 
= \Esub[\Big]{Z \sim \design}{\paren[\big]{\eate(Z,y(Z)) - \ate(y)}^2} 
\enspace.
\]
The MSE depends not only on the design $\design$ and estimator $\eate$, but also the potential outcome function $y$ under consideration.
Because the underlying potential outcome function $y$ is unknown to the experimenter, the MSE of a statistical procedure cannot be evaluated before the experiment is run.
In order to evaluate different statistical procedures without exact knowledge of the potential outcome functions, we consider its worst case MSE within a set of potential outcome functions, in our case the moment-restricted interference model.
More precisely, the \emph{maximum risk} of a statistical procedure within a $q$-th moment restricted ANI model is defined as
\[
\sup_{y \in \momodelq}
\Esub[\Big]{Z \sim \design}{\paren[\big]{\eate(Z,y(Z)) - \ate(y)}^2} 
\enspace.
\]
The maximum risk evaluates a design-estimator pair by comparing its worst-case performance among all potential outcome functions in the $q$-th moment restricted ANI model, $y \in \momodelq$.
Unlike MSE, the maximum risk can, in principle, be evaluated before the experiment is run and can therefore be used as a criterion for selecting a design-estimator pair.

We define the minimax risk as the smallest achievable maximum risk when considering all possible choices of design and estimator.
Formally, the \emph{minimax risk} depends on the causal effect $\ate$, the underlying network $G$, and the moment restriction $q$, and is defined as
\begin{equation}\label{eq:minimax_risk}
	\risk(\ate, G, q) 
	= \inf_{\design,\eate} \risk(\design,\eate, \ate, G, q)
	= \inf_{\design,\eate} \sup_{y \in \cM}
	\Esub[\Big]{Z \sim \design}{\paren[\big]{\eate(Z,y(Z)) - \ate(y)}^2} 
	\enspace.
\end{equation}
The minimax risk $\risk(\ate, G, q)$ characterizes the fundamental statistical limits of the experimental problem at hand: no statistical procedure can attain a maximum risk which is smaller than $\risk(\ate, G, q)$.
In addition to its theoretical motivation, a clear understanding of the minimax MSE can also provide practical guidance.
For example, an applied researcher might be satisfied using a statistical procedure whose performance is close to minimax optimal because no other procedure can offer significant improvements.
In Section~\ref{sec:supp_minimax_general} of the appendix, we show how this notion of minimax risk for ANI models and contrastive effects can be extended to the more general design-based framework of \citet{harshaw2022design}.

The minimax risk obeys several basic monotonicity properties: it increases if the moments are made to be less restrictive or if additional edges are added to the graph.
These properties follow directly from the fact that these modifications enlarge the restricted ANI model under consideration.
For two graphs $G,G'$ on the same set of vertices, we write $G' \subseteq G$ when every edge of $G'$ is also an edge of $G$.
We present the monotonicity properties below:

\begin{restatable}{proposition}{monotonicity}\label{prop:monotonicity}
	The minimax risk $\risk(G,\ate,q)$ obeys the following monotonicity properties:
	\begin{itemize}
		\item \textbf{Monotonicity in $q$:} If $q' \leq q$, then $\risk(G,\ate,q) \leq \risk(G,\ate,q')$.
		\item \textbf{Monotonicity in $G$:} If $G' \subseteq G$, then $\risk(G',\ate,q) \leq \risk(G,\ate,q)$.
	\end{itemize}
\end{restatable}

We have defined the minimax risk in terms of deterministic estimators which use only the data from the experiment; however, it is possible to consider randomized statistical procedures which use an additional source of randomness other than the treatment assignment, e.g. a random number generator.
In Section~\ref{sec:randomness_appendix} of the appendix, we show that the minimax risk among all procedures, including randomized ones, is the same as the minimax risk among the restricted class of deterministic procedures, mirroring well-known results from classical decision theory \citep{lehmann1998theory}.
Nevertheless, additional randomness may still be beneficial for constructing simpler and computationally efficient statistical procedures.
Indeed, the statistical procedures introduced in this work are randomized, and their ``de-randomizations'' appear computationally challenging to implement.

Finally, it may be beneficial to draw comparisons between the design-based perspective and the sampling-based perspective, particularly from the viewpoint of semi-parametric inference.
In the classical superpopulation setting, one observes independent samples from some distribution $P_{\theta}$, parametrized by a high dimensional nuisance component $\theta$.
The goal is to estimate some low dimensional functional $\psi(\theta)$ of the nuisance parameter.
In design-based inference, one can view the entire potential outcome function $y$ as the high-dimensional nuisance parameter and $\tau(y)$ as the low-dimensional quantity of interest.
One major distinction between these perspectives arises in the context of experimental design: the semi-parametric theory typically takes the data generating distribution $P_{\theta}$ as given, and this includes the random assignment of treatment.
In the design-based framework, the experimental design $\design$ is decoupled from the nuisance parameter $\tau(y)$ and needs to be chosen by the experimenter.
As we will see, this selection of both design and estimator will introduce additional challenges in establishing lower bounds on the minimax risk in the design-based setting.

\subsection{Conflict Graph} \label{sec:conflict_graph}

Our analysis of the minimax risk will prominently feature the conflict graph $\cgraph$, originally introduced by \citet{kandiros2024conflict} for the purpose of constructing statistical procedures.
In this paper, we show that the conflict graph is in fact a central quantity when studying statistical complexity of these network experiment problems.

At a high level, the conflict graph $\cgraph$ encodes the fundamental unobservability between estimand-relevant exposures.
Recall the exposures $\setb{ (e_{i,1}, e_{i,0} ) }_{i=1}^n$ which define the contrastive estimand.
We say that two exposures $e_{i,k}$ and $e_{j,\ell}$ are \emph{in conflict} if there is no intervention under which they can both be observed, i.e. there does not exist $z \in \Omega$ for which $h_i(z) = e_{i,k}$ and $h_j(z) = e_{j,\ell}$.
This happens exactly when the two exposures $e_{i,k}$ and $e_{j,\ell}$ specify different treatment assignments to at least one of the units in the intersection of extended neighborhoods $\extneighi \cap \extneigh{j}$.
The conflict graph is formally defined below:

\begin{definition}
	Given a network $G$ and contrastive effect $\ate$, the corresponding \emph{conflict graph} $\cgraph = (\cvert, \cedges)$ contains the estimand-relevant exposures as vertices $\cvert = \setb{ (e_{i,1}, e_{i,0} ) }_{i=1}^n$ with an edge if two exposures are in conflict.
\end{definition}

Note that the conflict graph $\cgraph$ depends on both the underlying network $G$ and the causal effect of interest $\ate$.
In this way, any analysis of the minimax risk in terms of the conflict graph $\cgraph$ reflects both $G$ and $\ate$.
However, we suppress this dependency to avoid unnecessary notational clutter.
We typically consider a generic contrastive effect throughout the paper and make it clear from context when we focus on a specific effect.

\begin{figure}[t!]
	\centering
	\begin{subfigure}[t]{0.3\textwidth}
		\centering
		\resizebox{\linewidth}{!}{%
            \tikzset{every picture/.style={line width=0.75pt}} 

\begin{tikzpicture}[x=0.75pt,y=0.75pt,yscale=-1,xscale=1]

\draw   (256,75) .. controls (256,61.19) and (267.19,50) .. (281,50) .. controls (294.81,50) and (306,61.19) .. (306,75) .. controls (306,88.81) and (294.81,100) .. (281,100) .. controls (267.19,100) and (256,88.81) .. (256,75) -- cycle ;
\draw   (196,173) .. controls (196,159.19) and (207.19,148) .. (221,148) .. controls (234.81,148) and (246,159.19) .. (246,173) .. controls (246,186.81) and (234.81,198) .. (221,198) .. controls (207.19,198) and (196,186.81) .. (196,173) -- cycle ;
\draw   (309,173) .. controls (309,159.19) and (320.19,148) .. (334,148) .. controls (347.81,148) and (359,159.19) .. (359,173) .. controls (359,186.81) and (347.81,198) .. (334,198) .. controls (320.19,198) and (309,186.81) .. (309,173) -- cycle ;
\draw   (418,74) .. controls (418,60.19) and (429.19,49) .. (443,49) .. controls (456.81,49) and (468,60.19) .. (468,74) .. controls (468,87.81) and (456.81,99) .. (443,99) .. controls (429.19,99) and (418,87.81) .. (418,74) -- cycle ;
\draw   (465,172) .. controls (465,158.19) and (476.19,147) .. (490,147) .. controls (503.81,147) and (515,158.19) .. (515,172) .. controls (515,185.81) and (503.81,197) .. (490,197) .. controls (476.19,197) and (465,185.81) .. (465,172) -- cycle ;
\draw    (267,95.5) -- (232,150.5) ;
\draw    (296,95.5) -- (322,151.5) ;
\draw    (418,74) -- (306,75) ;
\draw    (453,96.5) -- (478,150.5) ;

\draw (287,76) node   [align=left] {\begin{minipage}[lt]{12.24pt}\setlength\topsep{0pt}
{\Large $\displaystyle 1$}
\end{minipage}};
\draw (226,175) node   [align=left] {\begin{minipage}[lt]{12.24pt}\setlength\topsep{0pt}
{\Large $\displaystyle 2$}
\end{minipage}};
\draw (338,174.25) node   [align=left] {\begin{minipage}[lt]{12.24pt}\setlength\topsep{0pt}
{\Large $\displaystyle 3$}
\end{minipage}};
\draw (449,75) node   [align=left] {\begin{minipage}[lt]{12.24pt}\setlength\topsep{0pt}
{\Large $\displaystyle 4$}
\end{minipage}};
\draw (494,174.25) node   [align=left] {\begin{minipage}[lt]{12.24pt}\setlength\topsep{0pt}
{\Large $\displaystyle 5$}
\end{minipage}};

\end{tikzpicture}%
        }
		\caption{Network $G$}
	\end{subfigure}%
	~ 
	\begin{subfigure}[t]{0.3\textwidth}
		\centering
		\resizebox{0.8\linewidth}{!}{%
            \tikzset{every picture/.style={line width=0.75pt}} 

\begin{tikzpicture}[x=0.75pt,y=0.75pt,yscale=-1,xscale=1]

\draw   (238,99.5) .. controls (238,85.69) and (249.19,74.5) .. (263,74.5) .. controls (276.81,74.5) and (288,85.69) .. (288,99.5) .. controls (288,113.31) and (276.81,124.5) .. (263,124.5) .. controls (249.19,124.5) and (238,113.31) .. (238,99.5) -- cycle ;
\draw   (159,172) .. controls (159,158.19) and (170.19,147) .. (184,147) .. controls (197.81,147) and (209,158.19) .. (209,172) .. controls (209,185.81) and (197.81,197) .. (184,197) .. controls (170.19,197) and (159,185.81) .. (159,172) -- cycle ;
\draw   (320,173) .. controls (320,159.19) and (331.19,148) .. (345,148) .. controls (358.81,148) and (370,159.19) .. (370,173) .. controls (370,186.81) and (358.81,198) .. (345,198) .. controls (331.19,198) and (320,186.81) .. (320,173) -- cycle ;
\draw   (406,102) .. controls (406,88.19) and (417.19,77) .. (431,77) .. controls (444.81,77) and (456,88.19) .. (456,102) .. controls (456,115.81) and (444.81,127) .. (431,127) .. controls (417.19,127) and (406,115.81) .. (406,102) -- cycle ;
\draw   (465,172) .. controls (465,158.19) and (476.19,147) .. (490,147) .. controls (503.81,147) and (515,158.19) .. (515,172) .. controls (515,185.81) and (503.81,197) .. (490,197) .. controls (476.19,197) and (465,185.81) .. (465,172) -- cycle ;
\draw    (263,124.5) -- (262,318) ;
\draw    (345,198) -- (345,255) ;
\draw    (431,127) -- (430,317) ;
\draw    (491,249) -- (490,197) ;
\draw   (455,342) .. controls (455,355.81) and (443.81,367) .. (430,367) .. controls (416.19,367) and (405,355.81) .. (405,342) .. controls (405,328.19) and (416.19,317) .. (430,317) .. controls (443.81,317) and (455,328.19) .. (455,342) -- cycle ;
\draw   (516,274) .. controls (516,287.81) and (504.81,299) .. (491,299) .. controls (477.19,299) and (466,287.81) .. (466,274) .. controls (466,260.19) and (477.19,249) .. (491,249) .. controls (504.81,249) and (516,260.19) .. (516,274) -- cycle ;
\draw   (369,281) .. controls (369,294.81) and (357.81,306) .. (344,306) .. controls (330.19,306) and (319,294.81) .. (319,281) .. controls (319,267.19) and (330.19,256) .. (344,256) .. controls (357.81,256) and (369,267.19) .. (369,281) -- cycle ;
\draw   (287,343) .. controls (287,356.81) and (275.81,368) .. (262,368) .. controls (248.19,368) and (237,356.81) .. (237,343) .. controls (237,329.19) and (248.19,318) .. (262,318) .. controls (275.81,318) and (287,329.19) .. (287,343) -- cycle ;
\draw   (210,274) .. controls (210,287.81) and (198.81,299) .. (185,299) .. controls (171.19,299) and (160,287.81) .. (160,274) .. controls (160,260.19) and (171.19,249) .. (185,249) .. controls (198.81,249) and (210,260.19) .. (210,274) -- cycle ;
\draw    (451,328.5) -- (476,293.5) ;
\draw    (282,328.5) -- (328,300.5) ;
\draw    (287,343) -- (405,342) ;
\draw    (203,291.5) -- (241,329.5) ;
\draw    (185,249) -- (184,197) ;
\draw    (204,186.5) .. controls (233,235.5) and (250,259.5) .. (251,321.5) ;
\draw    (275,321.5) .. controls (282,285.5) and (298,262.5) .. (327,189.5) ;
\draw    (202,256.5) .. controls (224,217.5) and (250,176.5) .. (251,120.5) ;
\draw    (326,263.5) .. controls (317,241.5) and (273,183.5) .. (277,120.5) ;
\draw    (473,189.5) -- (439,318.5) ;
\draw    (443,124.5) -- (475,254.5) ;
\draw    (285,336.5) .. controls (433,305.5) and (420,214.5) .. (424,125.5) ;
\draw    (288,99.5) .. controls (419,123.5) and (420,275.5) .. (421,318.5) ;

\draw (246,84) node [anchor=north west][inner sep=0.75pt]  [font=\huge] [align=left] {$\displaystyle e_{1,0}$};
\draw (167,158) node [anchor=north west][inner sep=0.75pt]  [font=\huge] [align=left] {$\displaystyle e_{2,0}$};
\draw (329,158) node [anchor=north west][inner sep=0.75pt]  [font=\huge] [align=left] {$\displaystyle e_{3,0}$};
\draw (415,88) node [anchor=north west][inner sep=0.75pt]  [font=\huge] [align=left] {$\displaystyle e_{4,0}$};
\draw (472,158.5) node [anchor=north west][inner sep=0.75pt]  [font=\huge] [align=left] {$\displaystyle e_{5,0}$};
\draw (167,261) node [anchor=north west][inner sep=0.75pt]  [font=\huge] [align=left] {$\displaystyle e_{2,1}$};
\draw (327,268) node [anchor=north west][inner sep=0.75pt]  [font=\huge] [align=left] {$\displaystyle e_{3,1}$};
\draw (244,330) node [anchor=north west][inner sep=0.75pt]  [font=\huge] [align=left] {$\displaystyle e_{1,1}$};
\draw (473,261) node [anchor=north west][inner sep=0.75pt]  [font=\huge] [align=left] {$\displaystyle e_{5,1}$};
\draw (412,329) node [anchor=north west][inner sep=0.75pt]  [font=\huge] [align=left] {$\displaystyle e_{4,1}$};

\end{tikzpicture}%
        }
		\caption{DTE Conflict Graph}
	\end{subfigure}
	~
	\begin{subfigure}[t]{0.3\textwidth}
		\centering
		\resizebox{0.8\linewidth}{!}{%
            \tikzset{every picture/.style={line width=0.75pt}} 

\begin{tikzpicture}[x=0.75pt,y=0.75pt,yscale=-1,xscale=1]

\draw   (238,99.5) .. controls (238,85.69) and (249.19,74.5) .. (263,74.5) .. controls (276.81,74.5) and (288,85.69) .. (288,99.5) .. controls (288,113.31) and (276.81,124.5) .. (263,124.5) .. controls (249.19,124.5) and (238,113.31) .. (238,99.5) -- cycle ;
\draw   (159,172) .. controls (159,158.19) and (170.19,147) .. (184,147) .. controls (197.81,147) and (209,158.19) .. (209,172) .. controls (209,185.81) and (197.81,197) .. (184,197) .. controls (170.19,197) and (159,185.81) .. (159,172) -- cycle ;
\draw   (320,173) .. controls (320,159.19) and (331.19,148) .. (345,148) .. controls (358.81,148) and (370,159.19) .. (370,173) .. controls (370,186.81) and (358.81,198) .. (345,198) .. controls (331.19,198) and (320,186.81) .. (320,173) -- cycle ;
\draw   (406,102) .. controls (406,88.19) and (417.19,77) .. (431,77) .. controls (444.81,77) and (456,88.19) .. (456,102) .. controls (456,115.81) and (444.81,127) .. (431,127) .. controls (417.19,127) and (406,115.81) .. (406,102) -- cycle ;
\draw   (465,172) .. controls (465,158.19) and (476.19,147) .. (490,147) .. controls (503.81,147) and (515,158.19) .. (515,172) .. controls (515,185.81) and (503.81,197) .. (490,197) .. controls (476.19,197) and (465,185.81) .. (465,172) -- cycle ;
\draw    (263,124.5) -- (262,318) ;
\draw    (345,198) -- (344,256) ;
\draw    (431,127) -- (430,317) ;
\draw    (491,249) -- (490,197) ;
\draw   (455,342) .. controls (455,355.81) and (443.81,367) .. (430,367) .. controls (416.19,367) and (405,355.81) .. (405,342) .. controls (405,328.19) and (416.19,317) .. (430,317) .. controls (443.81,317) and (455,328.19) .. (455,342) -- cycle ;
\draw   (516,274) .. controls (516,287.81) and (504.81,299) .. (491,299) .. controls (477.19,299) and (466,287.81) .. (466,274) .. controls (466,260.19) and (477.19,249) .. (491,249) .. controls (504.81,249) and (516,260.19) .. (516,274) -- cycle ;
\draw   (369,281) .. controls (369,294.81) and (357.81,306) .. (344,306) .. controls (330.19,306) and (319,294.81) .. (319,281) .. controls (319,267.19) and (330.19,256) .. (344,256) .. controls (357.81,256) and (369,267.19) .. (369,281) -- cycle ;
\draw   (287,343) .. controls (287,356.81) and (275.81,368) .. (262,368) .. controls (248.19,368) and (237,356.81) .. (237,343) .. controls (237,329.19) and (248.19,318) .. (262,318) .. controls (275.81,318) and (287,329.19) .. (287,343) -- cycle ;
\draw   (210,274) .. controls (210,287.81) and (198.81,299) .. (185,299) .. controls (171.19,299) and (160,287.81) .. (160,274) .. controls (160,260.19) and (171.19,249) .. (185,249) .. controls (198.81,249) and (210,260.19) .. (210,274) -- cycle ;
\draw    (185,249) -- (184,197) ;
\draw    (198,193) .. controls (231,238) and (250,259.5) .. (251,321.5) ;
\draw    (275,321.5) .. controls (282,285.5) and (302,267) .. (331,194) ;
\draw    (197,252) .. controls (219,213) and (250,176.5) .. (251,120.5) ;
\draw    (331,259) .. controls (317,239.5) and (270,185) .. (274,122) ;
\draw    (480,194) -- (439,318.5) ;
\draw    (443,124.5) -- (475,254.5) ;
\draw    (285,336.5) .. controls (433,305.5) and (420,214.5) .. (424,125.5) ;
\draw    (288,99.5) .. controls (419,123.5) and (420,275.5) .. (421,318.5) ;
\draw    (208,266) -- (324,186) ;
\draw    (206,184) -- (321,271) ;
\draw    (285,88) .. controls (392,103) and (449,217) .. (470,261) ;
\draw    (286,351) .. controls (339,338) and (453,279) .. (473,191) ;
\draw    (418,124) -- (359,261) ;
\draw    (414,322) -- (361,191) ;
\draw    (202,189) .. controls (236,234) and (320,330) .. (405,342) ;
\draw    (204,258) .. controls (222,230) and (314,119) .. (406,102) ;

\draw (246,84) node [anchor=north west][inner sep=0.75pt]  [font=\huge] [align=left] {$\displaystyle e_{1,0}$};
\draw (167,158) node [anchor=north west][inner sep=0.75pt]  [font=\huge] [align=left] {$\displaystyle e_{2,0}$};
\draw (329,158) node [anchor=north west][inner sep=0.75pt]  [font=\huge] [align=left] {$\displaystyle e_{3,0}$};
\draw (415,88) node [anchor=north west][inner sep=0.75pt]  [font=\huge] [align=left] {$\displaystyle e_{4,0}$};
\draw (472,158.5) node [anchor=north west][inner sep=0.75pt]  [font=\huge] [align=left] {$\displaystyle e_{5,0}$};
\draw (167,261) node [anchor=north west][inner sep=0.75pt]  [font=\huge] [align=left] {$\displaystyle e_{2,1}$};
\draw (327,268) node [anchor=north west][inner sep=0.75pt]  [font=\huge] [align=left] {$\displaystyle e_{3,1}$};
\draw (244,330) node [anchor=north west][inner sep=0.75pt]  [font=\huge] [align=left] {$\displaystyle e_{1,1}$};
\draw (473,261) node [anchor=north west][inner sep=0.75pt]  [font=\huge] [align=left] {$\displaystyle e_{5,1}$};
\draw (412,329) node [anchor=north west][inner sep=0.75pt]  [font=\huge] [align=left] {$\displaystyle e_{4,1}$};

\end{tikzpicture}%
        }
		\caption{GATE Conflict Graph}
	\end{subfigure}
	\caption{An illustration of a network $G$ and conflict graphs for the DTE and GATE.
	}
	\label{fig:cg}
\end{figure}
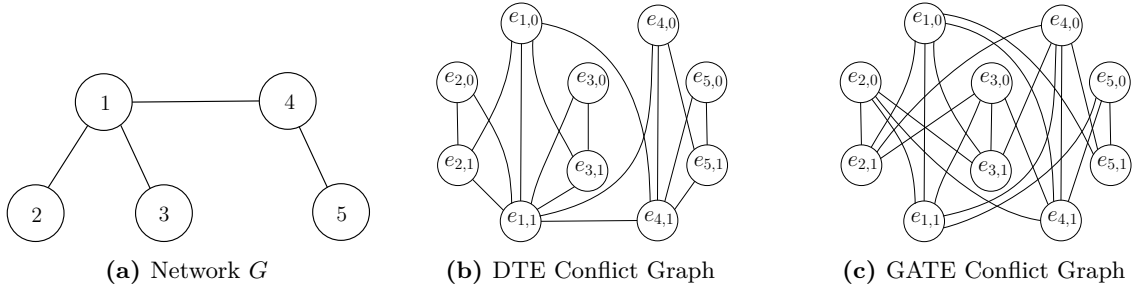

To better understand how the structure of the conflict graph $\cgraph$ depends on both the underlying network $G$ and the causal effect $\ate$, we focus on several examples below:

\begin{itemize}
	\item \textbf{No Interference:}
	Consider the setting of no-interference, in which case we may see the underlying graph $G$ as containing no edges.
	Suppose we want to estimate the average treatment effect, where $e_{i,1}$ is the exposure that subject $i$ receives treatment and $e_{i,0}$ is the exposure that subject $i$ receives the control.
	In this case, if $i \neq j$ then there are no conflicts between $e_{i,k}$ and $e_{j,\ell}$ because the individual treatments can be selected separately.
	On the other hand, there will be a conflict between $e_{i,1}$ and $e_{i,0}$ because a subject cannot be treated and untreated under any intervention.
	Thus, the conflict graph $\cgraph$ is simply a matching of the $2n$ exposures.
	
	\item \textbf{Direct Treatment Effect:} 
	Recall that $e_{i,1}$ is the exposure that subject $i$ is treated and all of its neighbors are untreated while $e_{i,0}$ is the exposure that the entire extended neighborhood of subject $i$ receives the control. 
	In this case, there is no edge in the conflict graph $\cgraph$ between all-control exposures $e_{i,0}$ and $e_{j,0}$ since we can observe both exposures by setting all units to control.
	On the other hand, if at least one of $k,\ell$ is equal to $1$ (i.e. representing the direct exposure), then there is an edge between $e_{i,k}$ and $e_{j,\ell}$ if and only if $i,j$ are connected in $G$. 
	In this way, the conflict graph $\cgraph$ has the underlying graph $G$ as a subgraph.

	\item \textbf{Global Average Treatment Effect:} 
	Recall that $e_{i,1}$ is the exposure that all extended neighbors are treated while $e_{i,0}$ is the exposure that no extended neighbors are treated.
	In this case, there are no edges in the conflict graph $\cgraph$ between all-treatment exposures $e_{i,1}$ and $e_{j,1}$ for subjects $i$ and $j$ because they can both be realized by the intervention $z = \mathbf{1}$.
	Likewise, there are no edges in the conflict graph $\cgraph$ between all-control exposures $e_{i,0}$ and $e_{i,0}$.
	Instead, the edges in the conflict graph arise between exposures $e_{i,k}$ and $e_{j,\ell}$ where $k \neq \ell$ and subjects $i$ and $j$ have common extended neighbors, i.e. $\extneighi \cap \extneigh{j} \neq \emptyset$.
	These exposures are in conflict because both $i$ and $j$ have a common neighbor who would have to be treated and not treated to observe $e_{i,k}$ and $e_{j,\ell}$, which is impossible.
	In this way, the conflict graph $\cgraph$ is a bipartite graph that resembles the two hop graph $G^2$ of the underlying graph $G$.
\end{itemize}

See Figure~\ref{fig:cg} for an illustration of how for the same underlying graph $G$, the resulting conflict graph will typically be different for these two effects.
As seen in these examples, the edges in the conflict graph encode the fundamental problem of causal inference and its extension to network interference.

The conflict graph is a summary of the experimental settings, capturing exactly which pairs of exposures are impossible to realize simultaneously under any design.
It abstracts away from other details of the experiment, such as the exact definition of the estimand and the interference structure that gave rise to these conflicts.
The following proposition demonstrates that this abstraction retains enough information about the statistical problem to exactly characterize the minimax risk of the corresponding network experiment problem.
In what follows, we let $\allgraphs$ denote the set of all labelled graphs (see Section~\ref{sec:supp-cg-characterizes-risk} in the appendix for a detailed proof).

\begin{restatable}{proposition}{cgcharacterization}\label{prop:cg-characterize-minimax}
	Let $G$ be an interference network, $\tau$ be a contrastive effect.
	For each moment parameter $q \geq 2$, there exists a function $f_q : \allgraphs \to \Reals_{\geq 0}$ such that $\risk(G, \ate, q) = f_q(\cgraph)$.
\end{restatable}

Proposition~\ref{prop:cg-characterize-minimax} demonstrates that the minimax risk can be expressed as a function of the conflict graph $\cgraph$.
A key aspect of its proof is the translation of the network experiment minimax problem into an equivalent minimax problem defined solely in terms of $\cgraph$ through its independence sets.
The translation relies on the fact that every experimental design can be identified with a probability measure over independent sets in $\cgraph$.
In fact, we show that it suffices to consider probability measures over maximal independent sets in $\cgraph$, corresponding to experimental designs which provide as much information as possible about potential outcomes.
The structure of the maximal independent sets in $\cgraph$ therefore determines the minimax risk $\risk(G, \ate, q)$, albeit in a complex way.
In particular, the function $f_q : \allgraphs \to \Reals_{\geq 0}$ in Proposition~\ref{prop:cg-characterize-minimax} is not explicitly constructed, but rather defined through a minimax problem involving $\cgraph$ alone.
Our primary analysis in this paper therefore may be understood as providing upper and lower bounds on $f_q(\cgraph)$ in terms of more accessible graph quantities.

The conflict graph presented here is slightly different than that considered in \citet{kandiros2024conflict}.
In that paper, the vertices of $\cgraph$ were subjects in the experiment whereas in this paper the vertices of $\cgraph$ are the estimand-relevant exposures.
While the two are related, we prefer to use the definition in this paper as it contains exactly the information to characterize the minimax risk.

	\section{Lower Bounds on the Minimax Risk}\label{sec:lower_bounds}

In this section, we present statistical lower bounds on the minimax risk in terms of the conflict graph.
In Section~\ref{sec:lower_bound_tools}, we begin by demonstrating how Le Cam's method can be adapted from classical minimax theory to the design-based setting.
Next, we derive two types of statistical lower bounds using global and local information from the conflict graph.
These are presented in Sections~\ref{sec:global_lower_bound} and \ref{sec:local_lower_bound}, respectively.

\subsection{Le Cam's Method for Design-Based Inference}\label{sec:lower_bound_tools}

Le Cam's method is one of the most prominent approaches for obtaining statistical lower bounds on the minimax risk \citep{yu1997assouad,duchi2023lecture,wu2017lecture}. 
The method itself proceeds in two steps: the first step is a reduction from estimation to testing and the second step is the construction of a mixture which quantifies the limits of the testing problem.
In this section, we show how the first step can be adapted to the design-based setting.

Le Cam's method requires the construction of probability measures supported on potential outcome functions in the interference model.
To facilitate this, let $\cA$ be the $\sigma$-algebra generated by all finite subsets of the restricted model $\momodelq$, i.e. the countable-cocountable $\sigma$-algebra.
A \emph{distribution} over potential outcome functions is a probability measure on $(\momodelq, \cA)$.

Let $H_0,H_1$ be two distributions over $\momodelq$, supported on submodels $\cM_0,\cM_1 \subseteq \momodelq$ respectively.
In order to quantify the difference between the evaluation of the causal effect on the support of $H_0$ and $H_1$, we define their \emph{effect separation} as
\[
d(H_0,H_1) = \inf_{y \in \cM_0, y' \in \cM_1} \abs{\tau(y) - \tau(y')} \enspace.
\]
For any distribution $H$ over potential outcome functions and design $\design$, let $Q_H^{(\design)}$ be the pushforward measure of the product measure $\design \otimes H$ under the mapping $(z,y) \mapsto (z,y(z))$. In other words, $Q_H^{(\design)}$ is the distribution of the observed data when the potential outcome function is drawn according to $H$ and the intervention is drawn according to $\design$.

\begin{restatable}{proposition}{lecam}
\label{prop:lecam}
For any network $G$, causal effect $\tau$, and moment parameter $q$, the minimax risk is lower bounded as
\[
\risk(G, \ate, q) \geq \inf_\design \sup_{H_0,H_1} \frac{d^2(H_0,H_1)}{8} \cdot \paren[\Bigg]{1 - \tv\paren[\Big]{Q_{H_0}^{(\design)}, Q_{H_1}^{(\design)}}} 
\enspace,
\]
where $\tv$ is the total variation distance and the infimum is taken over all designs $\design$.
\end{restatable}

Proposition~\ref{prop:lecam} provides a lower bound that increases with the effect separation between $H_1$ and $H_0$, but decreases with their total variation distance.
Applying the proposition requires constructing distributions $H_0$ and $H_1$ to trade-off these quantities, which is a challenge shared by all statistical lower bounds of this form.

The aspect of Proposition~\ref{prop:lecam} that is unique to the design-based setting is the outer infimum over experimental designs $\design$, reflecting the fact that the choice of $\design$ is in the researcher's control.
This does not appear in the classical setting, where the data generating process is fixed.
This means that when applying Proposition~\ref{prop:lecam}, we must consider a choice of distributions $H_0$ and $H_1$ for each possible experimental design $\design$.
If we were to pick only choice of $H_0$ and $H_1$, it would need to produce good lower bounds for all possible designs to be informative.
Otherwise, the distributions $H_0$ and $H_1$ must be chosen with respect to the design, which we refer to as a \emph{design-adaptive} construction.
Design-adaptivity plays an important role in the development of our local-information lower bounds, whereas our global-information lower bounds are non-adaptive.

Very recently, \citet{gao2026impossibility} established statistical lower bounds for misspecification testing within the design-based framework.
Their approach may be understood as carrying out the second step of Le Cam's method.
In this work, we focus on minimax optimal rates of estimation and the mixtures we construct are considerably more complex.

\subsection{Lower Bound using Global Information}\label{sec:global_lower_bound}

Before presenting our lower bound, we must first define several graph-theoretic notions.
For the moment, let us consider a generic graph which (in a slight abuse of notation) we denote as $G = (V,E)$.
A subset of vertices $S \subseteq V$ is said to be \emph{independent} if no two vertices in $S$ are adjacent in $G$ and the \emph{independence number} of $G$ is the size of its largest independent set, denoted $\indset(G)$.
Given a subset of vertices $T \subseteq V$, the \emph{$T$-induced subgraph} is $G_T = (V_T, E_T)$ where $V_T = T$ and the edges $E_T$ are those edges in the original graph $G$ between vertices in $T$.
We use $\indset(G, T)$ to denote the independence number of the induced subgraph $G_T$.

The statistical lower bound in this section depends on the induced independence numbers within the conflict graph.
These quantities represent global information about the conflict graph in the sense that they depend on its entire topology and cannot be determined from local information such as degree sequences.
To gain intuition for these results, observe that each intervention $z \in \Omega$ reveals a set of exposures $S \subseteq \cvert$ which, by the definition of conflicts, form an independent set in $\cgraph$.
Thus, if all independent sets in $\cgraph$ are small then each intervention can only provide information on a similarly small number of exposures.
The following theorem formalizes this intuition into a precise statistical lower bound:

\begin{theorem}
\label{thm:global_lower_bound}
Let $G$ be an interference network, $\tau$ be a contrastive effect.
If $n \geq 3$, then for all moment parameters $q \geq 2$
\[
\risk(G,\tau,q) \geq  1.05 \cdot 10^{-5} \cdot \sup_{T \subseteq V_\cH, |T| \geq 5}  \frac{1}{ \indset(\cgraph,T) } \cdot \paren[\Big]{\frac{|T|}{n}}^{2 - \frac{2}{q}} \enspace.
\]
\end{theorem}

For a fixed subset of exposures $T$ in the conflict graph $\cH$, the lower bound in Theorem~\ref{thm:global_lower_bound} has two parts: the induced independence number $\indset(\cgraph,T)$ and the ratio of $(T / n)^{2 - 2/q}$.
Roughly speaking, the former arises as an effective sample size while the latter reflects the amount of mass that can be ``hidden'' depending on the moment $q$ that is being restricted.
The intuition behind these quantities is that if the induced independence number is large, we can construct a mixture that hides mass among more potential outcomes.
The overall lower bound is obtained by taking the supremum over sufficiently large subsets of exposures $T$.
We have not made an attempt to optimize the constant appearing in Theorem~\ref{thm:global_lower_bound}.

In typical applications of Theorem~\ref{thm:global_lower_bound}, it suffices to consider $T$ which are of order $\bigOmega{n}$.
In general, it is always possible to set $T = \cvert$ in which case the lower bound becomes $1 / \indset(\cgraph)$.
In the specific case of the direct effect, the conflict graph $\cgraph$ contains $G$ as an induced subgraph of size $n$, so that by setting $T = \setb{ e_{i,1} : i \in [n] }$, we obtain that $\indset(\cgraph,T) = \indset(G)$.
On the other hand, there are cases where the subset $T$ that maximizes the lower bound contains a vanishingly small fraction of vertices in $\cgraph$.
For example, consider a conflict graph $\cH$ that is the union of a clique $R$ of size $k_n$ and $2n - k_n$ isolated nodes, where $k_n = \littleO{n}$.
In this case, the optimal choice is to set $T = R$ in which case $\indset(\cgraph,T)= 1$ and the resulting lower bound is $\bigOmega{k_n / n}$ for $q=2$ and $\bigOmega{(k_n/n)^2}$ for $q=\infty$, both of which can be shown to be tight up to constant factors.
In general, the lower bound increases as $q$ decreases.

As stated above, the main idea for proving Theorem~\ref{thm:global_lower_bound} comes from the relationship between independent sets in $\cgraph$ and sets of exposures which can be simultaneously observed.
In the remainder of this section, we provide a proof sketch together with the necessary ingredients.
For a simplified proof sketch, we consider the case where $T = \cvert$ which yields the weaker $\bigOmega{1/|\cI(\cH)|}$ bound.
This simplification highlights the key features of the proof while avoiding some technical complications.
For the purpose of notational simplicity (and without loss of generality), we consider the casual estimand $\ate = (1/n) \sum_{i=1}^n \braces{ y_i(e_{i,1}) + y_i(e_{i,0}) }$ where the contrast is now a sum.

To apply Le Cam's method, we will construct distributions $H_1$ and $H_0$ over potential outcome functions in the restricted ANI model.
To this end, we consider a class of distributions $H(p)$ on the potential outcome functions which are defined by a single parameter $p \in (0,1)$.
Under $H(p)$, we select a subset $S$ of the $2n$ exposures uniformly at random among all subsets of size $\lceil 2pn \rceil$ and assign outcome $1$ to exposures in $S$ and outcome $0$ to exposures not in $S$.
For a distribution $H$ over outcome functions and intervention $z \in \Omega$, let $Q_{H}(\cdot | z)$ be the conditional distribution of the observed outcomes conditioned on the intervention being equal to $z$. 
Note that this distribution does not depend on the design $\design$ at all, since we are fixing the assignment $z$. 
For any $z \in \setb{0,1}^n$, let $R_z$ be the subset of the $2n$ exposures that we get to observe when the treatment assignment is $z$. 

We will consider mixtures of the form $H_0 = H(1/2)$ and $H_1 = H(1/2 + \delta)$ where the mixture parameter $\delta >0$  controls the separation between the two distributions.
We provide an auxiliary lemma about distinguishing $H_0$ and $H_1$ from the observed outcomes.

\begin{restatable}{lemma}{tvbound}\label{lem:tv_bound}
	Suppose that $n \geq 3$ and fix an intervention $z \in \Omega$.
	If the mixture parameter satisfies $\delta \leq 0.026 / \sqrt{|R_z|}$, then the total variation distance between the conditional observed distributions is bounded as 
$
d_{\mathrm{TV}}(Q_{H_0}^{(z)}, Q_{H_1}^{(z)}) \leq \frac{7}{8}
$.
\end{restatable}

Lemma~\ref{lem:tv_bound} quantifies how large the mixture parameter $\delta$ can be so that the observed outcomes under the two mixtures $H_0$ and $H_1$, conditioned on the intervention $z$, are hard to distinguish.
Together with an application of Le Cam's method as stated in Proposition~\ref{prop:lecam}, we are ready to provide a proof sketch of Theorem~\ref{thm:global_lower_bound}.

\begin{proof}[Proof Sketch of Theorem~\ref{thm:global_lower_bound}]
Let $H_0 = H(1/2)$ and $H_1 = H(1/2 + \delta)$ where $\delta = 0.026/\sqrt{|\cI(\cH)|}$. We will show that these two distributions satisfy the conditions of Proposition~\ref{prop:lecam}, which will then imply the desired lower bound.
Note that the choice of $H_0,H_1$ is independent of the design $\design$.

First, by definition of $H_0,H_1$ we have $d(H_0,H_1) = 2\delta = 0.052/\sqrt{|\cI(\cH)|}$. 
Second, by definition of the observed outcome distributions and the law of total probability, we have that for any design $\design$ 
\begin{align}
    d_{\mathrm{TV}}(Q_{H_0}^{(\design)}, Q_{H_1}^{(\design)}) 
    &= \sum_{z \in \setb{0,1}^n} \design(z) \cdot d_{\mathrm{TV}}(Q_{H_0}(\cdot|z), Q_{H_1}(\cdot|z)) \label{eq:tv_conditional} \enspace.
\end{align}

Therefore, if we can show that $Q_{H_0}(\cdot|z), Q_{H_1}(\cdot|z)$ are close for any value of $z$, then the same would be true for $Q_{H_0}^{(\design)}, Q_{H_1}^{(\design)}$ as well, for any design $\design$.
For a fixed $z$, the set $R_z$ is an independent set in $\cH$. Indeed, if two exposures in $R_z$ were connected by an edge in $\cH$, then by definition of the conflict graph we would not be able to observe both with the same assignment $z$. Thus, we have that $|R_z| \leq |\cI(\cH)|$, which implies that 
\[
\delta = \frac{0.026}{\sqrt{|\cI(\cH)|}} \leq \frac{0.026}{\sqrt{|R_z|}} \enspace.
\]
Thus, by Lemma~\ref{lem:tv_bound}, we have that $d_{\mathrm{TV}}(Q_{H_0}(\cdot|z), Q_{H_1}(\cdot|z)) \leq 7/8$ for all $z$, which implies that $d_{\mathrm{TV}}(Q_{H_0}^{(\design)}, Q_{H_1}^{(\design)}) \leq 7/8$ for all designs $\design$ by \eqref{eq:tv_conditional}. Therefore, using Proposition~\ref{prop:lecam}, we can lower bound the minimax risk as
\[
\cR(G,\tau,q,C) \geq \frac{d(H_0,H_1)^2}{8}  \cdot \inf_\design \paren[\Big]{1 - \tv\paren[\big]{Q_{H_0}^{(\design)}, Q_{H_1}^{(\design)}}} \geq \frac{0.052^2}{8 \cdot |\cI(\cH)|} \cdot (1 - 7/8) = \bigOmega[\Bigg]{\frac{1}{|\cI(\cH)|}} \enspace.
\qedhere
\]
\end{proof}

\subsection{Lower Bound using Local Information}\label{sec:local_lower_bound}

In this section, we provide a statistical lower bound on the minimax risk in terms of the degrees of the conflict graph $\cgraph$.
In contrast to global properties like induced independent sets, the degrees of the $\cgraph$ describe local connectivity information.
The lower bound presented in this section is complimentary to the one in the previous section in the sense that there are situations where one or the other is more informative.

For a vertex $e_{i,k} \in \cvert$ in the conflict graph $\cgraph$, its \emph{degree} $d_{i,k}$ is the number of vertices to which it is connected.
For a subset of vertices $T$, we denote the minimum degree in $T$ as $\dmin(\cgraph, T) = \min_{e_{i,k} \in T} d_{i,k}$.
Finally, we define the \emph{generalized harmonic mean degree} of vertices in $T$ as $\dhar(\cgraph, T,r) = \abs{T} / \braces{ \sum_{e_{i,k} \in T} 1 / d_{i,k}^r }^{1/r}$.
Note that for $r = 1$, the generalized harmonic mean reduces to the usual harmonic mean.

The minimax rate will increase as the conflict graph becomes more densely connected and the degrees of the vertices in $\cgraph$ provide local information regarding its connectivity. 
Thus, it seems natural that the the minimax rate can be lower bounded in terms of the degrees.
This can be formalized thru the following theorem, whose proof appears in Section~\ref{sec:local_lower_bound_appendix} of the appendix.

\begin{restatable}{theorem}{localLowerBoundGeneral}\label{thm:local_lower_bound_general}
	Let $G$ be an interference network, $\tau$ be a contrastive effect.
	If $n \geq 5$ then for all moment parameters $q \geq 2$,
\[
\risk(G,\tau,q) \geq \frac{\localconst}{n^{2 - \frac{2}{q}}} \sup_{T \subseteq \cvert, |T| \geq 5} 
\min\paren[\Big]{\dmin(\cgraph, T), \sqrt{ \dhar(\cgraph,T,2-2/q) }, \abs{T} }^{2 - \frac{2}{q}}
\enspace.
\]
\end{restatable}

For a fixed subset of vertices $T$ in the conflict graph, the lower bound in Theorem~\ref{thm:local_lower_bound_general} involves the minimum of three terms: the minimum degree in $T$, the generalized harmonic mean degree in $T$, and the size of $T$.
In order for the lower bound to be large, there must exist a subset $T$ for which all three of these quantities are large.
In particular, the second term indicates that we must choose $T$ such that its (generalized harmonic) average degree is large.
At this point, the first and third terms are in tension: we need to select a large number of vertices $T$, but we do not want the minimum degree to be small.
This means the supremum will typically not be achieved by the entire subset of vertices $T = \cvert$ because the resulting lower bound will be (at most) the minimum degree in $\cgraph$, which could be quite small, i.e. constant order.
There is less tension between these three terms when many exposures have approximately the same degree in the conflict graph.

The intuition that the ``bulk'' of the degrees in the conflict graph governs the minimax risk can be formalized using the critical degree, which is a quantile of the degree distribution.
More precisely, we define the \emph{degree distribution} of $\cgraph$ as $F_{\cgraph} : \Reals \to [0,1]$ where
\[
F_{\cgraph}(t) = \frac{1}{2n} \sum_{e_{i,k} \in \cvert} \indicator{ d_{i,k} \leq t } \enspace.
\]
From the degree distribution, we can define for any $q \geq 2$ the \emph{critical degree} $d^*_q(\cH)$ of $\cH$ as the largest integer $k$ such that there are at least $k^{\frac{2q-2}{3q-2}}$ nodes in $\cH$ with degree at least $k$, i.e. 
\[
d^*_q(\cgraph) = \sup\setb[\Bigg]{k \in [n] : 1 - F_\cH(k) \geq \frac{ k^{\frac{2q-2}{3q-2}} }{n}} \enspace. 
\]
The supremum can be understood as follows: keep searching for vertices with degree at least $k$ and we stop the first time there are fewer than $k^{\frac{2q-2}{3q-2}}$ of them.
In this way, the critical degree describes a way to resolve the trade-off between finding a subset of vertices that is large in size but also has a large (generalized harmonic) mean of degrees, which is a central tension in the lower bound of Theorem~\ref{thm:local_lower_bound_general}.
Note that the critical degree is decreasing as $q$ increases.

The following corollary of Theorem~\ref{thm:local_lower_bound_general} provides a lower bound based on the critical degree of $\cH$, which is the bound mentioned in the Main Result of Section~\ref{sec:intro}.

\begin{restatable}{corollary}{localLowerBoundCor}\label{cor:local_lower_bound_general}
    Let $G$ be an interference network, $\tau$ be a contrastive effect.
    For all moment parameters $q \geq 2$, we have
\[
\cR(G,\tau,q) \geq \localconst  \cdot \frac{(d^*_q(\cH))^{\frac{4(q-1)^2}{q(3q - 2)}}}{n^{2 - \frac{2}{q}}} \enspace.
\]
\end{restatable}

The critical degree provides a resolution of the tension between the three terms of Theorem~\ref{thm:local_lower_bound_general} by subsets $T$ consisting of nodes with the largest degree.
This simplifies the lower bound, but hides the fundamental tension.
To understand the implications of Corollary~\ref{cor:local_lower_bound_general}, it is informative to consider the lower bound for $d$-regular conflict graphs, for which the critical degree is $d^*_q(\cgraph) = d$.
Under $q=2$ moment restrictions, the minimax lower bound in Corollary~\ref{cor:local_lower_bound_general} is of rate $d^{1/2} / n$.
Under more restrictive $q = \infty$ conditions, the lower bound becomes $d^{4/3} / n^2$.
As anticipated from monotonicity of the minimax risk as a function of $q$, we have that $d^{4/3} / n^2 \leq d^{1/2} / n$ because $d \leq n$.

If the average degree in the conflict graph is large, then a standard argument using Markov's inequality implies that there should be a sufficiently large subset of vertices with sufficiently large degree. Indeed, we can show that the critical degree is always lower bounded by the average degree in the conflict graph (see Proposition~\ref{prop:critical_to_avg} in Section~\ref{sec:local_lower_bound_appendix} of the appendix).

We now present a proof sketch for Theorem~\ref{thm:local_lower_bound_general}.
For simplicity, we focus on second moment restrictions ($q=2$) and make the simplifying assumption that all degrees in the conflict graph $\cH$ are lower bounded by $d$, in which case the lower bound will be $\sqrt{d}/n$, as realized by $T = [n]$. 
This case captures the essential obstacles for establishing the general claim, while avoiding some technical details of the argument.
A notable difference of this proof compared to the global lower bound of Section~\ref{sec:global_lower_bound} is that the mixture distributions $H_0$ and $H_1$ will now be selected \emph{adaptively} depending on the experimental design $\design$ under consideration. 

\begin{proof}[Proof Sketch of Theorem~\ref{thm:local_lower_bound_general}]
Assume $d_i \geq d$ for all $i \in V_\cH$. 
For any exposure $e=e_{i,k} \in V(\cH)$, associated with a unit $i$ in the graph $G$, let $N(e)$ be the set of neighboring exposures in $\cH$ and let $A_e$ be the event that we observe the outcome of $i$ for exposure $e$, i.e.
\[
A_e = \setb[\Big]{z \in \setb{0,1}^n: h_i(z) = e} \enspace.
\]
Let $R_z$ denote the set of exposures whose outcomes we observe under $z$.
By definition, we have
\[
|R_z| = \sum_{e \in V(\cH)} \indicator{A_e}(z) \enspace.
\]
Consider an arbitrary design $\design$ and let $p \in (0,1/2)$ be a parameter that will depend on $\design$, but will be chosen later.
Our construction will be design-adaptive, depending on the size of exposure probabilities.
If at least one exposure has large probability, we can exploit that exposure to hide a large causal effect in its conflicting exposures which are unlikely to be observed.
If all exposure probabilities are small, then it is unlikely that we will observe many potential outcomes at all, in which case we can hide the causal effect among these outcomes.
The parameter $p$ represents what is to be considered small and large exposure probabilities and is chosen at the end to optimize the resulting lower bound.
Let us now formally distinguish between these two cases.

\begin{itemize}
    \item \textbf{Case 1:} Suppose there exists an exposure $e \in V(\cH)$ such that $\design(E_e) \geq p$. In that case, we define $H_0$ and $H_1$ to be supported on potential outcome functions $y^{(0)},y^{(1)}$ respectively, where
    \[
    y_i^{(1)}(e_{i,k}) = \begin{cases}
        \sqrt{\frac{n}{|N(e)|}} & \text{if } e_{i,k} \in N(e) \\
        0 & \text{otherwise}
    \end{cases}
    \quad , \quad 
    y_i^{(0)}(e_{i,k}) = 0 \quad \text{for all } e_{i,k} \in V(\cH) \enspace.
    \]
    By construction, we have that 
    \[
    d(H_0,H_1) = \frac{1}{n} \sum_{e' \in N(e)} \paren[\Bigg]{\sqrt{\frac{n}{|N(e)|}} - 0} = \sqrt{\frac{|N(e)|}{n}} \geq \sqrt{\frac{d}{n}}\enspace.
    \]
    Also, if event $E_e$ occurs, then we get to observe the outcome $y_i(e)$. By definition of the conflict graph, this implies that we do not observe any of the outcomes of exposures $e' \in N(e)$. Since $H_0,H_1$ only differ in the outcomes of $N(e)$, they are impossible to distinguish under that event. Thus,
    \[
    d_{\mathrm{TV}}(Q_{H_0}^{(\design)}, Q_{H_1}^{(\design)}) \leq 1 - \design(E_e) \leq 1 - p
    \]
    Therefore, using Proposition~\ref{prop:lecam} for this particular design $\design$, we have 
    \begin{align}
    \inf_{\eate} \mathrm{Risk}(\design,\eate) & \geq \frac{d(H_0,H_1)^2}{8} \cdot \paren[\big]{1 - d_{\mathrm{TV}}(Q_{H_0}^{(\design)}, Q_{H_1}^{(\design)})} 
     \geq \frac{dp}{8n} \enspace. \label{eq:case1}
    \end{align}
    \item \textbf{Case 2:} Suppose for all exposures $e \in V(\cH)$ we have that $\design(E_e) < p$. In that case, define the event $A$ that at most $4p n $ exposures are observed in total, i.e.
    \[
    A = \setb{z \in \setb{0,1}^n: |R_z|  \leq 4pn} \enspace,
    \]
    By Markov's inequality,
    \[
    \design(A^c) = \design\paren[\Bigg]{|R_z| \geq 4pn} \leq \frac{\E[\Big]{|R_z|}}{4pn} = \frac{\sum_{e \in V(\cH)} \design(E_e)}{4pn} \leq \frac{p \cdot |V(\cH)|}{4pn} =  \frac{1}{2} \enspace.
    \]
    Define $H_0 = H(1/2)$ and $H_1 = H(1/2 + \delta)$ where $\delta = 0.026/\sqrt{4pn}$. Then, if $z \in A$, by definition  $|R_z| \leq 4pn$. By Lemma~\ref{lem:tv_bound}, for any $z \in A$, we have
    \[
    d_{\mathrm{TV}}(Q_{H_0}(\cdot|z), Q_{H_1}(\cdot|z)) \leq \frac{7}{8} \enspace.
    \]
    The law of total probability then implies that 
\[
1 - d_{\mathrm{TV}}(Q_{H_0}^{(\design)}, Q_{H_1}^{(\design)}) \geq \design(A) \cdot (1 -c) \geq \frac{1}{2}\paren[\Big]{1 - \frac{7}{8}} = \frac{1}{16}\enspace.
\]
Therefore, using Proposition~\ref{prop:lecam} for this particular design $\design$, we get
\begin{align}
    \inf_{\eate} \mathrm{Risk}(\design,\eate) & \geq \frac{d(H_0,H_1)^2}{8} \cdot
    \paren[\big]{1 - d_{\mathrm{TV}}(Q_{H_0}^{(\design)}, Q_{H_1}^{(\design)})} \geq \frac{4}{8 \cdot 2pn} \cdot \frac{1}{16} = \bigOmega[\Bigg]{\frac{1}{pn}}  \label{eq:case2}
\end{align}
\end{itemize}
 Since any design $\design$ falls in exactly one of the two cases above, the minimax risk is lower bounded by the minimum of \eqref{eq:case1} and \eqref{eq:case2}.
 By selecting $ p = \bigTheta{1/\sqrt{d}}$ to optimize this minimum, we get $\cR(G,\ate,2) = \bigOmega{\sqrt{d}/n}$, as desired.
\end{proof}

	\section{Upper Bounds on the Minimax Risk}\label{sec:upper_bounds}

In this section, we present general upper bounds on the minimax risk of network experiment problems.
One can view the analysis of previously proposed statistical procedures as implicitly deriving upper bounds on the minimax risk.
However, most analyses in the literature are appropriate for providing sufficient conditions for consistency and may not accurately describe the minimax risk itself.
Additionally, most works consider only a single moment restriction, usually $q=2$ or $q = \infty$.
In contrast, the upper bounds we present in this section are improved (i.e. smaller) and hold for arbitrary moment restrictions.

\subsection{Unbiased Statistical Procedures}\label{sec:upper_bound_cgd}

In order to upper bound the minimax risk, we must construct a statistical procedure and analyze its maximum risk.
We begin by considering a class of unbiased statistical procedures, which we then refine in the next section thru a bias-variance trade-off.
We summarize these upper bounds below in Theorem~\ref{thm:conflict_graph_design} and defer a discussion of the underlying statistical procedure until the end of this section.

\begin{restatable}{theorem}{cgd}
    \label{thm:conflict_graph_design}
    Let $G$ be an interference network, $\tau$ be a contrastive effect.
    For any moment parameter $q \geq 2$, the minimax risk is bounded as
	\[
	\risk(G, \ate, q)
	\leq  \constant \cdot \frac{\lamH^{\frac{2}{q}}\cdot d_{\mathrm{avg}}(\cH)^{1 - \frac{2}{q}}}{n} \enspace,
	\]
	where $\lamH$ is the largest eigenvalue of the adjacency matrix of $\cH$ and $d_{\mathrm{avg}}(\cH)$ is the average degree of $\cH$.
	Moreover, this bound on the minimax risk is achieved by an unbiased statistical procedure.
\end{restatable}

Theorem~\ref{thm:conflict_graph_design} provides upper bounds on the minimax risk which depend on the moment restriction under consideration.
Under second moment restrictions ($q=2$), we recover the $\lamH / n$ bound of \citet{kandiros2024conflict}.
For uniformly bounded outcomes ($q=\infty$), we obtain an improved rate of $\davgH / n$.
For intermediate moments $q \in (2,\infty)$, the resulting upper bound interpolates between these two quantities.
Like the true minimax risk, our upper bounds are decreasing as the network experiment problem becomes easier, i.e. $q$ increases so the moment condition becomes more restrictive.

The difference between $\davgH$ and $\lamH$ becomes more pronounced under degree heterogeneity.
For example, the largest adjacency eigenvalue always lies between the average and maximum degrees: $\davgH \leq \lamH \leq \dmaxH$.
When there is low degree heterogeneity, meaning that all vertices have roughly the same degree, then all three quantities are roughly equal.
On the other hand, the quantities can be of different asymptotic orders in the presence of degree heterogeneity.
An illustrative example which demonstrates this separation is when $\cgraph$ is a star graph, where one ``center'' node is connected to all remaining ``leaf'' nodes and and no leaf nodes are connected.
Then we would have $\davgH \approx 2$, $\lamH \approx \sqrt{n}$, and $\dmaxH \approx n$.
Degree heterogeneity is commonly observed in social networks where degree distributions are heavy-tailed, i.e. a small number of people are more popular than the average person.
In this case, the degree heterogeneity of the underlying network $G$ will induce degree heterogeneity in the conflict graph $\cgraph$ and, as a result, the quantities $\lamH$ and $\davgH$ will have different asymptotic orders.

The interplay between moment restrictions and degree heterogeneity in the minimax rate can be understood by considering how much of the total ``mass'' of the potential outcomes can be concentrated on a few number of subjects.
The key idea can be explained as follows: under less restrictive moment conditions, the ``mass'' of the potential outcomes is allowed to be more concentrated on a few number of subjects, which in turn means that the maximal risk will depend on the smallest exposure probabilities among all subjects.
If the conflict graph has heterogeneous degrees, then increasing the probability of observing an exposure with many conflicts is only possible by decreasing other exposure probabilities.
On the other hand, under more restrictive moment conditions, the ``mass'' of potential outcomes cannot be hidden in just a few subjects.
This means that it can be advantageous to select an experimental design under which a few exposures are observed with small probability if doing so allows most exposures to be observed with larger probability.

Underlying Theorem~\ref{thm:conflict_graph_design} is a new family of experimental designs which are extensions of the Conflict Graph Design of \citet{kandiros2024conflict} to accommodate general moment restrictions.
At a high level, the Conflict Graph Design works by sampling independent auxiliary ``desired inclusion'' variables and then resolving conflicts according to a particular importance ordering.
The distribution of the auxiliary variables and the importance ordering depend on the underlying moment restriction $q$.
For the setting of $q=2$ considered by \citet{kandiros2024conflict}, auxiliary variables have identical distributions and the importance ordering is given by the largest eigenvector of the adjacency matrix $\cgraph$.
For the case of $q = \infty$, the distribution of each auxiliary variable depends on the exposure's degree and the importance ordering is simply the degree ordering.
Detailed description and analyses of this new family of experimental designs may be found in Section~\ref{sec:cgd_appendix} of the supplement.

\subsection{Bias-Variance Tradeoff}\label{sec:bias_variance}

In this section, we consider bias-variance trade-offs which can further improve the upper bounds on the minimax rate.
The key idea is simple: if a few ``central'' or ``high degree'' vertices in $\cgraph$ cause $\lamH$ or $\davgH$ to be large, then we can choose to ignore their outcomes when constructing our experimental design and estimator so that the resulting variance is improved at the cost of some additional bias. 
Similar ideas have been explored in the context of AIPW estimation in network experiments \citep{zocklein2025regression}.
The challenge will be to identify the optimal set $S$ of vertices to ignore.

More formally, let $S \subseteq \cvert$ be a subset of the vertices in the conflict graph.
Let $\ccgraphS$ denote the subgraph of $\cgraph$ induced by the vertices in $\cvert \setminus S$.
Let $\ate_S$ be the contrastive treatment effect restricted to the exposures in $\cvert \setminus S$, i.e. 
\[
\ate_S = \frac{1}{n} \sum_{i=1}^n \paren[\Big]{y_i(e_{i,1})\indicator{e_{i,1} \notin S} - y_i(e_{i,0})\indicator{e_{i,0} \notin S}}\enspace.
\]
If we were to target the estimand $\ate_S$ by running the Conflict Graph Design from Theorem~\ref{thm:conflict_graph_design} on the subgraph $\cgraph(\cvert \setminus S)$, then our estimator $\eate$ is unbiased for the restricted estimand $\ate_S$, i.e. $\E{\eate} = \ate_S$.
Thus, its variance and bias for the actual estimand $\ate$ can be bounded as
\[
\Var{\eate} = 
\bigO[\bigg]{
	\frac{\lambda(\ccgraphS)^{\frac{2}{q}} \cdot {\davg(\ccgraphS)}^{1 - \frac{2}{q}}}{n} 
}
\quadand
\paren[\Big]{\E{\eate} - \tau}^2 = 
\bigO[\bigg]{  \paren[\bigg]{ \frac{|S|}{n} }^{2 - \frac{2}{q}} } 
\enspace.
\]
Based on these bounds, we can define the following proxy for the mean squared error of the estimator $\eate$, as it depends on the subset $S$ that is to be removed:
\[
Q_{q}(S,\cH) := 
\lambda(\ccgraphS)^{\frac{2}{q}} \cdot {\davg(\ccgraphS)}^{1 - \frac{2}{q}}
+ \frac{ \abs{S}^{2-\frac{2}{q} }}{ n^{1-\frac{2}{q}} }
\enspace.
\]
The best possible MSE using this approach would then be given by the subset $S$ that minimizes $Q_{q}(S,\cH)$.
This is the content of the main result of this section (for the proof, see Appendix~\ref{sec:exponential_bv_appendix}).

\begin{restatable}{theorem}{biasvarianceexponential}
\label{thm:bias_variance_exponential}
Let $G$ be an interference network, $\tau$ be a contrastive effect.
For any moment parameter $q \geq 2$, the corresponding minimax risk is bounded as
\[
\cR(G,\tau,q) \leq  \constant \cdot \frac{\min_{S \subseteq \cvert} Q_{q}(S,\cH)}{n}
\enspace.
\]
\end{restatable}

The upper bound on the minimax rate in Theorem~\ref{thm:bias_variance_exponential} is an improvement upon the one presented in Theorem~\ref{thm:conflict_graph_design}.
To see this, observe that selecting $S = \emptyset$ results in the same rate, and indeed the same statistical procedure.
Whether the optimal set $S^* = \argmin_{S \subset \cvert} Q_{q}(S, \cgraph)$ offers a substantially different rate depends on the heterogeneity of the conflict graph.
If the vertices of the conflict graph all have degree roughly $d$, then the variance term is $\bigO{d}$ because in this case $\lamH \approx \davgH \approx d$.
In this case, we would need to remove $S = \bigOmega{n}$ vertices before the variance term would become $\littleO{d}$, at which point the squared bias term would already be $\bigOmega{1}$.
Thus, the bias-variance trade-off does not improve the minimax rate when $\cgraph$ has homogeneous degrees.

On the other hand, the bias-variance trade-off offers a large improvement when the conflict graph has significant degree heterogeneity.
The reason is that removing a small number of vertices will substantially decrease the variance term whereas the small number of them ensures the increase in bias is minimal.
Consider the illustrative example where $\cgraph$ is the star graph and let us focus on $q = 2$.
In this case, the unbiased procedure of Theorem~\ref{thm:conflict_graph_design} yields a rate of $1 / \sqrt{n}$ whereas by removing the center node $S = \{1\}$, we have that $\lambda(\ccgraphS) = 1$, so that the resulting rate from Theorem~\ref{thm:bias_variance_exponential} is $1/n$.
Moreover, the lower bounds establish that $1/n$ is indeed the exact minimax-rate for the star graph.

Finding the optimal subset $S^*$ that minimizes $Q_q(S,\cgraph)$ is a combinatorial optimization problem.
The problem is similar to the Partial Vertex Cover \citep{hochba1997approximation} which is NP-Hard, so a computationally feasible algorithm to find $S^*$ is unlikely to exist.
However, building upon approximation algorithms for Partial Vertex Cover, we show in the following proposition that there is a polynomial-time algorithm which computes a set $S$ that is within a logarithmic factor of optimal, i.e. $Q_{q}(S, \cgraph) \leq \log(n) \cdot Q_{q}(S^*, \cgraph)$.
This yields a computationally efficient approach to the bias-variance trade-off that incurs only a logarithmic increase in the risk.
This is the content of the proposition below and we refer to Section~\ref{sec:approximation_algo_appendix} in the supplement for more details.

\begin{restatable}{proposition}{approximationalgorithm}
\label{prop:approximation_algorithm}
Let $G$ be an interference network, $\tau$ be a contrastive effect.
For any moment parameter $q \geq 2$, there exists a design-estimator pair $(\design,\eate)$ that can be implemented in polynomial time whose maximum risk is bounded as
\[
\sup_{y \in \momodelq}
\mathrm{MSE}(\design,\eate, \ate, y) 
\leq  \bigO[\Big]{(\log n)^{2 - \frac{2}{q}}} \cdot \frac{ \min_{S \subseteq V(\cH)} Q_{q}(S,\cgraph)}{n} \enspace.
\]
\end{restatable}

	\section{Tightness of Minimax Analysis}\label{sec:comparisons}

The minimax risk introduced in Section~\ref{sec:def-minimax-risk} is a function of the underlying network $G$ and the contrastive estimand $\ate$.
While we present upper and lower bounds that use different information from the conflict graph $\cgraph$, the bounds themselves do not always match in the usual sense, i.e. they are not equal up to a constant.
In this section, we show that the upper and lower bounds are both tight within the class of regular graphs.
Moreover, these bounds are nearly matching for a uniformly random regular graph, characterizing the minimax rate in an average-case sense.

To demonstrate the tightness of our upper and lower bounds, we consider the class of \emph{$d$-regular} graphs, where each vertex has exactly $d$ neighbors.
We denote the class of $d$-regular graphs on $n$ nodes by $\cG_{d,n}$, though we simply write $\cG_d$ when $n$ is clear from context.
Using triangular array asymptotics, we consider the degree parameter $d$ to be possibly growing with $n$.

Restricting our attention to the class of regular graphs is useful because it is parametrized by a single parameter $d$.
This has the benefit of simplifying the presentation of our upper and lower bounds on the minimax rate.
On the other hand, the class is rich enough to contain a variety of different graphs.
We further focus our attention on the bounds obtained for second moment restrictions, $q = 2$.
Our goal is two fold: we aim to show (1) how the upper and lower bounds depend on $d$ and that (2) they provide tight bounds on the minimax risk within this class $\cG_{d,n}$.

\subsection{Direct Treatment Effect}

We first focus on estimating the Direct Effect $\ate_{\mathrm{DTE}}$, which was defined in Section~\ref{sec:preliminaries}.
As discussed there, the conflict graph $\cgraph$ for this estimand behaves roughly like the original interference graph $G$. 
In particular, as we have seen in Corollary~\ref{corollary:direct-rate}, the following upper and lower bounds on the minimax risk hold for any $d$-regular graph $G$:
\begin{equation}\label{eq:upper_lower_d_regular}
\frac{\sqrt{d+1}}{n} \lesssim \cR(G,\ate_{\mathrm{DTE}},2) \lesssim \frac{d+1}{n} \enspace.
\end{equation}
Note that we have removed the independent set $\indset(G)$ term from the lower bound because it is not governed by $d$, though it will play a role in our later analysis.
When the degree $d$ is constant relative to the number of units $n$, then the bounds match up to a constant factor. 
A particular case of interest is the Stable Unit Treatment Value Assumption (SUTVA), where there is no interference and the graph $G$ is empty. 
In that case, we obtain that the minimax risk of estimating the Average Treatment Effect (ATE) is $\Theta(1/n)$, which is the classical parametric rate of estimation.

When $d$ is growing with $n$, there is a multiplicative gap of $\sqrt{d}$ between the upper and lower bounds.
We will demonstrate that there are families of graphs in $\cG_{d}$ for which the minimax rate coincides with the upper bound as well as families for which the minimax rate coincides with the lower bound.
The next proposition establishes tightness of the upper bound.

\begin{restatable}{proposition}{upperboundtightness}
    \label{prop:d_regular_upper_bound_tightness}
    For any sequence $d = \{d^{(n)}\}_{n=1}^\infty$ with $d = \littleO{n}$, 
    \[
    \sup_{G \in \cG_{d}} \cR(G,\ate_{\mathrm{DTE}},2) \asymp \frac{d}{n} \enspace.
    \]
\end{restatable}

For the proof, we consider the family of $d$-chorded cycles, where the maximum independent set is $\bigO{d/n}$. We then use the global lower bound of Theorem~\ref{thm:global_lower_bound} to establish the matching lower bound of $\bigOmega{d/n}$.
For further details, see Section~\ref{sec:supp_tightness_upper} of the supplementary.

Next, we show that the lower bound of $\sqrt{d}/n$ is also attainable up to constant factors for a family of graphs in $\cG_{d}$ with $d = \bigTheta{n^{2/3}}$.

\begin{restatable}{proposition}{lowerboundtightness}\label{prop:d_regular_lower_bound_tightness}
    There exists a sequence $d = \{d^{(n)}\}_{n=1}^\infty$ such that $d = \bigTheta{n^{2/3}}$ and
    \[
    \inf_{G \in \cG_{d}} \cR(G, \ate_{\mathrm{DTE}} ,2) \asymp \frac{\sqrt{d}}{n} \enspace.
    \]
\end{restatable}

For the proof, we focus on a family of graphs called the \emph{Generalized Quadrangles}, which are well studied in finite geometry \citep{tits1959trialite,payne2009finite}. The crucial property of these graphs is the existence of large independent sets that have very well-behaved intersections with each other.
Combined with other strong connectivity properties, this allows us to construct an improved design tailored for this graph family that achieves the $\sqrt{d}/n$ estimation rate.
More details can be found in Section~\ref{sec:supp_tightness_lower} of the supplementary.
Together, these results imply that both the upper and lower bounds from~\ref{eq:upper_lower_d_regular} in terms of the degree parameter are tight.

Finally, we investigate the behavior of the minimax risk for the ``average'' graph in $\cG_{d}$, by considering a graph sampled uniformly at random from $\cG_{d}$.

\begin{restatable}{corollary}{randomdregulardte}
\label{cor:random_d_regular}
Suppose $G$ is a uniformly random $d$-regular graph with $d = o(n)$.
Then, with probability at least $1 - o(1)$ over the random choice of $G$, 
    \[
    \frac{d}{n \cdot \log d} \lesssim \cR(G, \ate_{\mathrm{DTE}} ,2) \lesssim \frac{d}{n} \enspace.
    \]
\end{restatable}

Corollary~\ref{cor:random_d_regular} shows that for the overwhelming majority of $d$-regular graphs, the minimax risk of estimating the DTE is equal to the upper bound of Corollary~\ref{corollary:direct-rate} up to a $\log d$ factor. 
To prove that, we utilize the fact that  if $d = \littleO{n}$, then with high probability the largest independent set in $G$ has size at most $\bigO{n \log d / d}$ \citep[see e.g.][]{cooper2002random}, together with our global lower bound from Theorem~\ref{thm:global_lower_bound}.
See Section~\ref{sec:random_d_regular_appendix} of the supplementary for more details.

\subsection{Global Average Treatment Effect}

When the estimand of interest is the GATE $\ate_{\mathrm{GATE}}$, the conflict graph $\cgraph$ is similar in structure to the two hop graph $G^2$ of the underlying network $G$, as demonstrated in Section~\ref{sec:conflict_graph}.
In particular, Corollary~\ref{corollary:gate-rate} provides the following upper and lower bounds on the minimax risk that hold for any graph $G$:

\begin{equation}
    \label{eq:upper_lower_gate}
    \frac{\sqrt{d_{\mathrm{avg}}(G^2)}}{n} \lesssim \cR(G, \ate_{\mathrm{GATE}} ,2) \lesssim \frac{\lambda(G^2)+1}{n} \enspace.
\end{equation}

Even when the graph $G$ is $d$-regular, the quantities $\sqrt{d_{\mathrm{avg}}(G^2)}$ and $\lambda(G^2)$ are unfortunately not determined by the degree parameter $d$.
The reason is that the size of a two-hop neighborhood can range from $d$ to $d^2$.
However, the following theorem shows that when the network $G$ is random $d$-regular, $\sqrt{d_{\mathrm{avg}}(G^2)}$ and $\lambda(G^2)$ are well-approximated by $d$ and $d^2$ so that there is a multiplicative gap of order $d$ between them:

\begin{restatable}{corollary}{randomdregulargate}
\label{cor:random_d_regular_gate}
Suppose $G$ is a random $d$-regular graph with $d = o(\sqrt{n})$.
Then, with probability at least $1 - o(1)$ over the random choice of $G$, 
    \[
    \frac{d}{n } \lesssim \cR(G, \ate_{\mathrm{GATE}}, 2) \lesssim \frac{d^2}{n} \enspace.
    \]
\end{restatable}

Corollary~\ref{cor:random_d_regular_gate} together with Corollary~\ref{cor:random_d_regular} provides evidence for the widely held belief that the DTE is easier to estimate than the GATE.
To see this, observe that when $G$ is a random $d$-regular graph, then these two corollaries together yield that $\risk(G, \ate_{\mathrm{DTE}}) \lesssim d/n$ whereas $\risk(G, \ate_{\mathrm{GATE}}) \gtrsim d/n$.
To the best of our knowledge, this provides the first theoretical evidence of the difference between the rates at which different network effects can be estimated.
	
	\section{Conjectured Computational Hardness}\label{sec:conjecture}

In the previous section, we have shown that while the derived upper and lower bounds on the minimax risk do not match in all cases, they are both tight for the class of regular graphs and are in fact matching for random regular graphs.
However, this still leaves open the question of whether a simple characterization of the minimax risk is possible.
In this section, we present two conjectures against the existence of an efficient algorithm which can approximately compute the minimax risk and discuss their consequences.
For the purposes of conciseness, we gloss over certain technical aspects of computational complexity theory.

At the heart of our conjectures is an intimate connection between the experimental design problem and computational problems regarding independent sets in graphs.
As noted in Section~\ref{sec:conflict_graph}, the fact that $\risk(G, \ate, q)$ can be expressed as $f_q(\cgraph)$ for some function $f_q$ relies on the translation of the experimental design problem into a problem regarding distributions over maximal independent sets of $\cgraph$.
The central question considered in this section is whether this function $f_q : \allgraphs \to \Reals_{\geq 0}$ is efficiently computable in any meaningful sense.

The computer science literature has established that computational problems involving the independent sets of a graph are provably hard.
For example, it is NP-Hard to calculate the size of the largest independent set $\indset(G)$ of a graph $G$ \citep{karp1972reducibility}.
In fact, even approximating $\indset(G)$ to within a multiplicative $n^{1-\epsilon}$ factor is NP-Hard for all $\epsilon > 0$ \citep{hastad1996clique,zuckerman2006linear}.
This $n^{1-\epsilon}$ inapproximability result presumes that the input graph is a priori unrestricted so that the size of the maximal independent set could range from $1$ to $n$.
By focusing specifically on $d$-regular graphs where the size of the largest independent set is known to lie between $n/d$ and $n$, the inapproximability can be strengthened further still: approximating $\indset(G)$ to within a $\bigO{d/\log^2 d}$ multiplicative factor for a $d$-regular graph $G$ is hard assuming the Unique Games Conjecture \citep{austrin2011inapproximability}.
Based on these results and the connection between the minimax rate of network experiment problems and the structure of independent sets in the conflict graph, we present our two conjectures:

\begin{conjecture} \label{conjecture:approximation}
	Given a network $G$, contrastive effect $\ate$, and moment parameter $q$, the minimax risk $\risk(\ate, G, q) = f_q(\cgraph)$ is NP-Hard to approximate within a multiplicative factor $n^c$ for some $c \in (0,1)$.
\end{conjecture}

\begin{conjecture} \label{conjecture:d-regular-approx}
	Given a $d$-regular graph $G$ and the direct effect $\ate_{\mathrm{DTE}}$, the minimax risk $\risk(G, \ate_{\mathrm{DTE}}, q) = f_q(\cgraph)$ is NP-Hard to approximate within a multiplicative factor of $d^{1/2 - \epsilon}$ for all $\epsilon > 0$.
\end{conjecture}

Conjecture~\ref{conjecture:approximation} posits that even an approximation of the minimax rate is computationally challenging for general network experiment problems.
Conjecture~\ref{conjecture:d-regular-approx} is a strengthening of Conjecture~\ref{conjecture:approximation} specifically for the direct effect, which states that even determining whether the true minimax rate is close to our derived upper bound $d/n$ or lower bound $\sqrt{d} / n$ is computationally infeasible.
As discussed above, our primary evidence for these conjectures is the close relationship between the minimax rate $\risk(G, \ate, q)$ and the structure of maximal independent sets in the conflict $\cgraph$ as well as the general inapproximability results for maximum independent set.
Note that the computational hardness of approximating the minimax rate would not preclude an explicit construction of a minimax optimal statistical procedure; however, implementing such a procedure would likely be computationally inefficient, otherwise one could use Monte Carlo simulations of this procedure to approximate the minimax rate.


These conjectures posit that the minimax rate does not admit a simple description.
If true, these conjectures would have significant consequences for the practical relevance of minimax theory in network experiments.
Perhaps the most valuable aspect of minimax theory is to provide a theoretical benchmark that certain statistical procedures should pass, e.g. being minimax rate-optimal.
If the minimax rate of a network experiment is not efficiently computable, then the rate optimality of a particular statistical procedure cannot be assessed on an arbitrary network and effect.
In this case, we see two possible alternatives.
The first is to show rate optimality of statistical procedures within a certain family of graphs in a worst-case or average-case sense.
For example, Corollary~\ref{cor:random_d_regular} on the rate optimality of Conflict Graph Design for random $d$-regular graphs may be seen as an example of this type of analysis.
The second alternative is to formulate a notion of computationally efficient minimax risk, which restricts the statistical procedure to be implementable in polynomial time.
From this perspective, understanding both the optimal computationally efficient statistical procedure and the rate of its corresponding risk would be natural next steps. 
	
	\section{Conclusion} \label{sec:conclusion}

In this paper, we have studied design-based minimax optimal rates for network experiments.
Our upper and lower bounds provide further understanding of how the minimax rate depends on the conflict graph $\cgraph$.
Generally speaking, our results show that the minimax rates are governed by its ``bulk connectivity'', in the sense that a few well-connected vertices will not affect the rates.
Moreover, the analysis suggests an interplay between the strength of moment restrictions and the extent to which small but well-connected subgraphs can affect the minimax rates.

We believe that the most interesting open directions posed by this work are further investigation into the computational complexity of network experiments.
Is approximately computing the minimax risk $\risk(G, \ate, q)$ computationally hard?
In Section~\ref{sec:conjecture}, we have provided several well-motivated conjectures that suggest the answer is yes.
Progress towards resolving these conjectures in any direction would provide deeper understanding into the fundamental obstructions in constructing optimal statistical procedures.
If statistical procedures which attain the minimax risk require exponential computational resources, then what is the best rates achievable by a computationally efficient experimental design and estimator?
Questions like these lie in the intersection of statistics, computational complexity, and graph theory.
Computational considerations notwithstanding, an exact characterization of the minimax rates in terms of the conflict graph (i.e. a description of $f_q(\cgraph)$) would be valuable in further understanding the essence of these statistical limits.
Doing this would ostensibly require both improved statistical lower bounds and experimental designs, which would be interesting in their own right regardless of computational tractability.
	
	\printbibliography
	
	\newpage
		
	\appendix
	
	\addcontentsline{toc}{section}{Appendix} 
	\part{Appendix} 
	\parttoc 
	\newpage

	\section{Minimax Theory for Design-Based Inference}\label{sec:minimax_theory_appendix}

In this section, we present the proofs of our general results involving the minimax risk for estimating arbitrary causal effects. 
The subsections are organized as follows:
\begin{itemize}
    \item In Section~\ref{sec:supp_minimax_general}, we present a general definition of minimax risk for arbitrary restrictions on the potential outcome functions.
    \item In Section~\ref{sec:monotonicity_appendix}, we establish various monotonicity properties of the minimax risk. 
    \item In Section~\ref{sec:randomness_appendix}, we show that additional randomness does not improve estimation.
    \item In Section~\ref{sec:supp_outcome_restrictions}, we prove that the minimax risk is invariant to some natural restrictions of the potential outcome functions.
    \item Finally, in Section~\ref{sec:supp-cg-characterizes-risk}, we prove that the minimax risk is characterized by the conflict graph.
\end{itemize}

\subsection{A General Definition of Minimax Risk}\label{sec:supp_minimax_general}

In this section, we briefly describe how to extend the definition of minimax optimality to more general design-based causal inference.
We focus on the design-based framework proposed by setting introduced by \citet{harshaw2022design}.

In the framework of \citet{harshaw2022design}, the researcher will posit, for each subject $i \in [n]$, a \emph{subject-level model space} $\cM_i$ which is a linear subspace of measurable functions.
The model spaces serve to encode the structure of the potential outcome functions.
The \emph{entire model} is a subset $\cM \subseteq \cM_1 \times \dots \times \cM_n$, which may incorporate additional moment restrictions.
The \emph{causal estimand} is a linear functional $\tau : \cM \to \R$.

Generalizing the definition in Section~\ref{sec:def-minimax-risk}, we define the \emph{minimax risk} of estimating a causal estimand $\tau$ over a model $\cM$ on the potential outcomes as follows:
\begin{equation}\label{eq:minimax_risk_general}
    \cR(\cM,\tau) 
    = \inf_{\design,\eate} \sup_{y \in \cM} \Esub[\big]{Z \sim \design}{\paren{\eate(Z,y(Z)) - \tau(y)}^2} \enspace.
\end{equation}
When $\ate$ is a contrastive effect and $\model$ is an ANI model, then the definition of minimax risk above reduces to the one presented in Section~\ref{sec:def-minimax-risk}.

\subsection{Monotonicity of the Minimax Risk (Proposition~\ref{prop:monotonicity})}\label{sec:monotonicity_appendix}

We first establish a monotonicity result which holds for arbitrary restrictions on the potential outcome functions and for arbitrary estimands. We then use it to prove each of the monotonicity properties in Proposition~\ref{prop:monotonicity}.

\begin{restatable}{lemma}{monotonicitygeneral}
\label{prop:monotonicity_model}
If $\cM \subseteq \cM'$ then $\cR(\cM,\tau) \leq \cR(\cM',\tau)$.
\end{restatable}
\begin{proof}
    Let us fix any estimator $\eate$ and a design $P$. Then, since $\cM$ is contained in $\cM'$, we have that
\[
\mathrm{Risk}(P,\eate) = \sup_{y \in \cM} \Esub{Z \sim P}{\paren{\eate(Z,Y(Z)) - \tau(y)}^2} \leq \sup_{y \in \cM'} \Esub{Z \sim P}{\paren{\eate(Z,Y(Z)) - \tau(y)}^2} = \mathrm{Risk}(P,\eate) \enspace.
\]
Taking the infimum over all estimators $\eate$ and designs $P$ on both sides gives the desired result.
\end{proof}

Now let us prove Proposition~\ref{prop:monotonicity} using the above general monotonicity result. For convenience, we restate it below.

\monotonicity*

\begin{proof}
    The proof of each of the monotonicity properties relies on applying Lemma~\ref{prop:monotonicity_model} to the appropriate sets of potential outcome functions. We now show how to do that for each statement separately.
    \begin{enumerate}
        \item \textbf{Monotonicity in q.} This will follow immediately from Lemma~\ref{prop:monotonicity_model} if we could show that $\cM(G,q,C) \subseteq \cM(G,q')$ for any $q' \leq q$. Consider any tuple of potential outcome functions $y \in \cM(G,q)$. For any $k \in \{0,1\}$, we have that
    \[
    \frac{1}{n} \sum_{i=1}^n \abs{y_i(\contrastz{i}{k})}^{q'} 
    = \frac{1}{n} \sum_{i=1}^n \paren{\abs{y_i(\contrastz{i}{k})}^{q}}^{\frac{q'}{q}}
    \leq \paren[\Big]{\frac{1}{n} \sum_{i=1}^n \abs{y_i(\contrastz{i}{k})}^{q}}^{\frac{q'}{q}}
    \leq 1 \enspace,
    \]
    where the first inequality follows from concavity of $x \mapsto x^{q'/q}$ and Jensen's inequality and the final inequality follows from the definition of $y \in \cM(G,q)$
    Therefore, $y \in \cM(G,q')$, which implies that $\cM(G,q) \subseteq \cM(G,q')$.

    \item \textbf{Monotonicity in $G$.} Finally, for the monotonicity in $G$, it suffices to establish that if $G' \subseteq G$, then $\cM(G',q) \subseteq \cM(G,q)$, since then Lemma~\ref{prop:monotonicity_model} applies. To prove this, it suffices to show that $\cM(G') \subseteq \cM(G)$, i.e. the potential outcome functions that satisfy ANI with respect to $G'$ also satisfy ANI with respect to $G$. But this also follows from the definition of graph inclusion. To distinguish between the two different graphs, we use $\tilde{N}_H(i)$ to denote the extended neighborhood of a node $i$ in graph $H$. In particular, suppose $y \in \cM(G')$. Then, let $i \in [n]$ and $z,z' \in \setb{0,1}^n$ be two treatment assignments such that for all $i$ the exposure of $i$ under $z$ and $z'$ is the same with respect to $G$, i.e.
    \[
    \setb{j \in \tilde{N}_{G}(i) : z_j = 1} = \setb{j \in \tilde{N}_{G}(i) : z_j' = 1} \enspace.
    \]
    Since $G' \subseteq G$, we have that $\tilde{N}_{G'}(i) \subseteq \tilde{N}_{G}(i)$, which implies that
    \[
    \setb{j \in \tilde{N}_{G'}(i) : z_j = 1} = \setb{j \in \tilde{N}_{G'}(i) : z_j' = 1} \enspace.
    \]
    Since $y \in \cM(G')$, the last equality implies that $y_i(z) = y_i(z')$. Since this holds for all such $i$ and all $z,z'$, we have that $y \in \cM(G)$, which completes the proof.
    \end{enumerate}
\end{proof}

\subsection{Randomized Estimators }\label{sec:randomness_appendix}

In Section~\ref{sec:preliminaries}, we defined our setup so that the randomness only came from the treatment assignment. 
However, since the experimenter controls the data collection process, they could also utilize additional randomness in choosing the treatment and in constructing the estimator for the causal estimand of interest.
In this section, we show that allowing a statistical procedure to use additional randomness does not improve the minimax risk.

In particular, let $S$ be the sample space of the additional randomness and let $\mathcal{S}$ be a $\sigma$-algebra on $S$. Since the arguments below work for any choice of $S$ and $\mathcal{S}$, we supress the dependence on them in the following definitions.

We define a \emph{design with randomness} as a probability measure $\design_{\mathrm{rand}} : \mathcal{F} \otimes \mathcal{S} \to [0,1]$ over the product space of treatment assignments and the additional randomness. 
Drawing a sample from $\design_{\mathrm{rand}}$ produces a pair $(Z,X)$, where $Z \in \Omega$ is the treatment assignment that is used by the experimenter and $X \in S$ is the additional randomness, which could be used by the estimator. 
We emphasize again that $\design_{\mathrm{rand}}$ is the only source of randomness in this problem and it is controlled by the experimenter.
Also, the experimenter could correlate the additional randomness with the treatment assignment $Z$, which implies that $\design_{\mathrm{rand}}$ is not necessarily a product measure over $\Omega$ and $S$.
A \emph{randomized estimator} $\eate_{\mathrm{rand}} : \Omega \times S \times \R^n  \to \R$ is a measurable function with respect to the product $\sigma$-algebra $\mathcal{F} \otimes \mathcal{S} \otimes \cB(\R^n) $, where $\cB(\R^n)$ is the Borel $\sigma$-algebra on $\R^n$. 
The pair $(\design_{\mathrm{rand}}, \eate_{\mathrm{rand}})$ is referred to as a \emph{randomized statistical procedure}.
In this setting, the experimenter first draws a random pair $(Z,X)$ according to the design with randomness $\design_{\mathrm{rand}}$, where $Z \in \Omega$ and $X \in S$. After applying the treatment $Z$, they observe outcomes $Y(Z)$ and output the estimate $\eate_{\mathrm{rand}}(Z,X,Y(Z))$ of the effect $\tau$.

In analogy with the definitions in Section~\ref{sec:supp_minimax_general}, we can define the \emph{Randomized Minimax Risk} for a given causal effect $\tau$ and model $\cM$ as the minimum risk over all designs with randomness and all randomized estimators, i.e.

\[
\mathcal{R}_{\mathrm{rand}}(\cM,\tau) = \inf_{\design_{\mathrm{rand}},\eate_{\mathrm{rand}}} \sup_{y \in \cM} \Esub{(Z,X) \sim \design_{\mathrm{rand}}}{\paren{\eate_{\mathrm{rand}}(Z,X,Y(Z)) - \tau(y)}^2} \enspace. 
\]

We now show that the minimax risk of randomized statistical procedures is the same as that of deterministic procedures. In particular, we establish that for any pair of design with randomness and randomized estimator, there exists a deterministic estimator that uses only the treatment part of the randomness and has the same or smaller MSE for all potential outcomes functions. 

\begin{restatable}{proposition}{randomness}
\label{prop:randomness_model}
For any model $\cM$ and any effect $\tau$, $\cR_{\mathrm{rand}}(\cM,\tau) = \cR(\cM,\tau)$.
\end{restatable}

\begin{proof}
    Since deterministic estimators are a special case of randomized estimators, we have by definition that
    \[
    \mathcal{R}_{\mathrm{rand}}(\cM,\tau) \leq \mathcal{R}(\cM,\tau) \enspace.
    \]
    To establish the reverse inequality, let's fix any randmized experimental design $\design_{\mathrm{rand}}$ and any randomized estimator $\eate_{\mathrm{rand}}$. Since $\design_{\mathrm{rand}}$ is a probability measure over the product space $\Omega \times S$, let's denote the marginal distribution of $\design_{\mathrm{rand}}$ over $\Omega$ by $\design$. 
    Now, define the estimator $\eate : \Omega \times \R^n \to \R$ as
    \[
    \eate(Z,y(Z)) = \Esub[\Big]{X \sim \design_{\mathrm{rand}}(\cdot \mid Z)}{\eate_{\mathrm{rand}}(Z,X,y(Z))} \enspace.
    \]
    Clearly, $\eate$ is a deterministic estimator that only depends on the treatment assignment $Z$ and the observed outcomes $y(Z)$.

    Let $y \in \cM$ be a potential outcome function in model $\cM$.
    Then, we have that
    \begin{align*}
    \Esub[\Bigg]{(Z,X) \sim \design_{\mathrm{rand}}}{\paren[\Big]{\eate_{\mathrm{rand}}(Z,X,y(Z)) - \tau(y)}^2} & = \Esub[\Bigg]{Z \sim \design}{\Esub[\Big]{X \sim \design_{\mathrm{rand}}(\cdot \mid Z)}{\paren{\eate_{\mathrm{rand}}(Z,X,y(Z)) - \tau(y)}^2}} \\
    & \geq \Esub[\Bigg]{Z \sim \design}{\paren[\Big]{\Esub[\Big]{X \sim \design_{\mathrm{rand}}(\cdot \mid Z)}{\eate_{\mathrm{rand}}(Z,X,y(Z))} - \tau(y)}^2} \\
    \intertext{by using Jensen's inequality for the convex function $x \mapsto x^2$}\\
    &= \Esub[\Bigg]{Z \sim \design}{\paren[\Big]{\eate(Z,y(Z)) - \tau(y)}^2} \enspace.
    \end{align*}
    Therefore, the statistical procedure $(\design, \eate)$ has smaller or equal MSE than the procedure $(\design_{\mathrm{rand}}, \eate_{\mathrm{rand}})$ that operates with extra randomness, for all potential outcome functions $y \in \cM$. Therefore, taking supremum over all $y \in \cM$ and then infimum over all pairs $(\design_{\mathrm{rand}}, \eate_{\mathrm{rand}})$, we have
    \[
    \mathcal{R}_{\mathrm{rand}}(\cM,\tau) \geq \mathcal{R}(\cM,\tau) \enspace,
    \]
    which completes the proof.
\end{proof}

\subsection{Variations of Outcome Restrictions}\label{sec:supp_outcome_restrictions}

In this section, we show that the minimax risk remains unaffected under some natural variations of the outcome restrictions introduced in Section~\ref{sec:preliminaries}. 

We start by showing that if we change the scaling of the constraint on the outcomes, the minimax risk scales accordingly.
In particular, for an interference network $G$, effect $\ate$, moment parameter $q \geq 2$ and positive real number $C>0$, we define the following set of potential outcome functions:
\[
\cM(G,\tau,q,C)=
\setb[\Big]{
	y \in \model: 
	\paren[\Big]{ \frac{1}{n} \sum_{i=1}^n \abs{ y_i(e_{i,k}) }^q  }^{1/q} \leq C
	\text{ for } k \in \setb{0,1}
} \enspace.
\]
This generalizes the $q$th moment restricted model $\momodelq$ defined in Section~\ref{sec:preliminaries}, which is defined with $C = 1$.
Likewise, we can define the minimax risk of estimating the effect $\ate$ over this set of potential outcome functions as
\[
\cR(G,\tau,q,C) = \inf_{\design,\eate} \sup_{y \in \cM(G,\tau,q,C)} \Esub{Z \sim \design}{\paren{\eate(Z,y(Z)) - \tau(y)}^2} \enspace.
\]
We now show that changing the bound on the outcomes results in a simple scaling of the minimax risk. Therefore, it is sufficient to consider the case $C=1$ without loss of generality.

\begin{proposition}
\label{prop:scaling}
Let $G$ be an interference network, $\ate$ be a contrastive effect and $q \geq 2$ be a moment parameter. Then, for any positive real number $C>0$, we have that
\[
\cR(G,\tau,q,C) = C^2 \cdot \cR(G,\tau,q,1) \enspace.
\]
\end{proposition}

\begin{proof}
    For any $L>0$ and potential outcome function $y$, we denote by $Ly$ the potential outcome function defined as $(Ly)_i(z) = L \cdot y_i(z)$ for all $i \in [n]$ and $z \in \setb{0,1}^n$.
    Let $U,V>0$ be any positive real numbers.
    Let $\design$ be any design and $\eate$ be any estimator. We define the following estimator $\eate_{U,V}$ as
    \[
    \eate_{U,V}(Z,y(Z)) = \frac{U}{V} \cdot \eate(Z, (V/U) \cdot y(Z)) \enspace.
    \]
    In particular, $\eate_{U,V}$ scales the observed outcomes by $V/U$, applies the estimator $\eate$, and then scales the output by $U/V$.

    Now consider an arbitrary potential outcome function $y \in \cM(G,\tau,q,U)$. Then, we have that $(V/U) \cdot y \in \cM(G,\tau,q,V)$, since for any $k \in \setb{0,1}$, we have that
    \[
    \paren[\Big]{ \frac{1}{n} \sum_{i=1}^n \abs{ ((V/U) \cdot y)_i(e_{i,k}) }^q  }^{1/q} = \frac{V}{U} \paren[\Big]{ \frac{1}{n} \sum_{i=1}^n \abs{ y_i(e_{i,k}) }^q  }^{1/q} \leq \frac{V}{U} \cdot U = V \enspace.
    \]
    Also, by linearity we have 
    \[
    \tau((V/U) \cdot y) = (V/U) \cdot \tau(y) \enspace.
    \]
    Therefore, for any $y \in \cM(G,\tau,q,U)$, we have that
    \begin{align*}
    \Esub[\Big]{\design}{\paren[\big]{\eate_{U,V}(Z,y(Z)) - \ate(y)}^2} &= \Esub[\Bigg]{\design}{\paren[\Big]{\frac{U}{V}\eate(Z,(V/U)\cdot y(Z)) - \frac{U}{V}\ate((V/U)\cdot y)}^2}\\
    &= \frac{U^2}{V^2} \cdot \Esub[\Bigg]{\design}{\paren[\Big]{\eate(Z,(V/U)\cdot y(Z)) - \ate((V/U)\cdot y)}^2}\\
    &\leq \frac{U^2}{V^2} \cdot \sup_{y \in \cM(G,\tau,q,V)} \Esub[\Bigg]{\design}{\paren[\Big]{\eate(Z, y(Z)) - \ate(y)}^2}\enspace,
    \end{align*}
    where in the last step we used the fact that $(V/U) y \in \cM(G,\tau,q,V)$. 
    Therefore, for any statistical procedure $(\design,\eate)$, we can find another statistical procedure $(\design, \eate_{U,V})$ whose maximal error under $\cM(G,\tau,q,V)$ is upper bounded by the scaled maximal error of $(\design,\eate)$ under $\cM(G,\tau,q,U)$. 
    It follows that
    \begin{equation}\label{eq:risk_inequality}
        \cR(G,\tau,q,U) \leq \frac{U^2}{V^2}\cdot \cR(G,\tau,q,V)
    \end{equation}
    We can now apply inequality~\ref{eq:risk_inequality} twice, one with $U = C, V = 1$ and one with $U = 1, V = C$ to obtain the result.
\end{proof}

We now show that the outcomes of units corresponding to exposures that do not appear in the estimand of interest do not affect the minimax risk. In particular, we show that even if we restrict these outcomes to have value $0$, the minimax risk remains the same.

Specifically, for any graph $G$ and estimand $\ate$ with relevant exposures $e_{i,1},e_{i,0}$ for $i \in [n]$, let $\cM^0(G,\tau)$ be the model of outcomes that satisfy ANI with respect to $G$ and have zero outcome in the non-relevant exposures, i.e.
\[
\cM^0(G,\tau) = 
\setb[\Big]{
    y \in \cM(G): 
	y_i(e) = 0 \text{ , if } e \neq e_{i,0},e_{i,1} ,\forall i \in [n]
}
\]
Then, for any graph $G$, estimand $\ate$ and $q \geq 2$, we can define the model of potential outcome functions with bounded $q$ moments and zero non-relevant outcomes as follows
\[
\cM^0(G,\tau,q)=
\cM^0(G,\tau) \cap \cM(G,\tau ,q)
 \enspace.
\]
Analogously to the previous definitions, we can define the minimax risk with respect to outcome functions in $\cM^0(G,\tau,q)$ as follows
\[
\cR^0(G,\tau,q) = \inf_{\design,\eate} \sup_{y \in \cM^0(G,\tau,q)} \Esub{Z \sim \design}{\paren{\eate(Z,y(Z)) - \tau(y)}^2} \enspace.
\]
In the next proposition, we show that the minimax risk is unaffected by the restrictions in the non-relevant exposures.

\begin{proposition}
    For any graph $G$, estimand $\ate$ and $q \geq 2$, 
    \[
    \cR^0(G,\tau,q) = \cR(G,\tau,q) \enspace.
    \]
\end{proposition}
\begin{proof}
    Since $\cM^0(G,\tau,q) \subseteq \cM(G,\tau,q)$, we trivially have
    \[
    \cR^0(G,\tau,q) \leq \cR(G,\tau,q) \enspace.
    \]
    We thus focus on showing the reverse inequality. 
    For any outcome function $y \in \cM(G)$, we denote by $y^0$ the corresponding outcome function that agrees with $y$ in the relevant exposures and has $0$ outcome everywhere else, i.e.
    \[
    y_i^0(z) = \begin{cases}
        y_i(z) &\text{ if } h_i(z) \in \setb{e_{i,0},e_{i,1}}\\
        0 &\text{ otherwise}
    \end{cases}\enspace.
    \]
    Clearly, if $y \in \cM(G,\tau,q)$, then $y^0 \in \cM(G,\tau,q)$, since the moment constraint is not affected by outcomes of exposures that do not appear in $\tau$. 
    Also, we always have $\tau(y) = \tau(y^0)$, since the estimand doesn't depend on the outcomes we are zeroing out.

    Now, let $(\design,\eate_0)$ be an arbitrary statistical procedure. We will construct an estimator $\eate$, such that the worst case performance of $(\design,\eate)$ under $\cM(G,\tau,q)$ is no worse than than of $(\design,\eate_0)$ under $\cM^0(G,\tau,q)$. 
    Indeed, we define $\eate$ in terms of $\eate_0$ as follows
    \[
    \eate(z,y(z)) = \eate_0(z,y^0(z))
    \]
    In words, $\eate$ works by zeroing out the non-relevant outcomes and then applying $\eate_0$ to the result. 
    Then, for any $y \in \cM(G,\tau,q)$ 
    \begin{align*}
         \Esub[\Big]{\design}{\paren[\big]{\eate(Z,y(Z)) - \ate(y)}^2} &= 
         \Esub[\Big]{\design}{\paren[\big]{\eate_0(Z,y^0(Z)) - \ate(y)}^2}\\
        &= \Esub[\Big]{\design}{\paren[\big]{\eate_0(Z,y^0(Z)) - \ate(y^0)}^2}
    \end{align*}
    Taking the supremum over $y \in \cM(G,\tau,q)$ on both sides yields
    \begin{align*}
        \sup_{y \in \cM(G,\tau,q)} \Esub[\Big]{\design}{\paren[\big]{\eate(Z,y(Z)) - \ate(y)}^2} &= 
        \sup_{y \in \cM(G,\tau,q)} \Esub[\Big]{\design}{\paren[\big]{\eate_0(Z,y^0(Z)) - \ate(y^0)}^2}\\
        &= \sup_{y \in \cM^0(G,\tau,q)} \Esub[\Big]{\design}{\paren[\big]{\eate_0(Z,y(Z)) - \ate(y)}^2}\enspace.
    \end{align*}
    The last step follows by the following observation, which follows by the definitions.
    \[
    \cM^0(G,\tau,q) = \setb[\Big]{y^0: y \in \cM(G,\tau,q)} \enspace.
    \]
    Since the above holds for an arbitrary statistical procedure $(\design,\eate)$, we conclude that
    \[
    \cR(G,\tau,q) \leq \cR^0(G,\tau,q) \enspace.
    \]
    The result follows.
\end{proof}

\subsection{Conflict Graph Characterizes Minimax Risk (Proposition~\mainref{prop:cg-characterize-minimax})} \label{sec:supp-cg-characterizes-risk}

In this section, we show that while the conflict graph abstracts away many features of the experimental design problem, it contains the essential information to characterize the minimax risk.
More precisely, we will show that the minimax risk is a function of the conflict graph.
In the main body, this is formalized as Proposition~\mainref{prop:cg-characterize-minimax}, which we reproduce below for completeness.
Recall that $\allgraphs$ is the set of all labeled graphs.

\cgcharacterization*

Our proof strategy will proceed by defining two minimax problems on an arbitrary labeled graph $\cgraph$.
In these minimax problems, each of the vertices in $\cgraph$ has an associated outcome and the target estimand is the average of the outcomes.
A random independent set in $\cgraph$ is selected and the outcomes of its vertices is revealed. 
The minimax problem is to construct both the distribution over independent sets and the estimator.
We first consider distributions over all independent sets, and then the restriction to distributions supported on maximal independent sets.
Through a series of reductions, we will show that the minimax risk of these graph theoretic problems exactly corresponds to the minimax risk of the original network experiment problem.

We begin by formally defining these minimax problems.
Consider a labeled graph $\cgraph = (\cvert, \cedges)$ with $2n$ vertices.
Note that we have assumed the number of vertices is even, for simplicity.
Consider a function $u : \cvert \to \Reals$ which maps each vertex to a real number and may be identified with a vector $u \in \Reals^{|\cvert|}$.
For a moment parameter $q \geq 2$, define the \emph{$q$th moment restricted model}
\[
\cM(\cgraph, q) = \setb{ u \in \Reals^{\cvert} : 
	\paren[\Big]{ \frac{1}{n} \sum_{i=1}^{n} \abs{u(i)}^q }^{1/q} \leq 1 
	\text{ and }
	\paren[\Big]{ \frac{1}{n} \sum_{i=n+1}^{2n} \abs{u(i)}^q }^{1/q} \leq 1 
}
\enspace,
\]
which restricts the $q$th moment for the first half and second half of the outcomes.
Note that the labelling of the graph $\cgraph$ is required to give meaning to ``first half'' and ``second half''.
The labelling will also be required to define the contrast between two groups of vertices.
Formally, we define the \emph{contrastive functional} $\ate : \cM(\cgraph, q) \to \Reals$ as
\[
\ate(u) = \frac{1}{n} \sum_{i = n+1}^{2n} u(i) - \frac{1}{n} \sum_{i=1}^{n} u(i) \enspace.
\]

In these minimax problems, we must select a probability measure over independent sets in $\cgraph$, which we formalize below.
Recall that a subset of vertices $S \subseteq \cvert$ is \emph{independent} if no pair of vertices $i, j \in S$ is adjacent in $\cgraph$.
We denote the independent sets of $\cgraph$ as $\indset(\cgraph)$.
We have that the independent sets form a measureable space $(\indset(\cgraph), \mathcal{P}(\indset(\cgraph)))$, where $\mathcal{P}(\indset(\cgraph))$ is the power set.
Let $\mathbb{D}(\indset(\cgraph))$ be the set of all probability measures on independent sets $\indset(\cgraph)$.

What is observed in these minimax problems is the outcomes for each of the vertices in the (randomly) selected independent set and this is the data which must be used to estimate the average outcome.
In order to simplify the formal notation, we introduce an auxiliary variable $*$ which will denote that no outcome has been observed.
Given an independent set $S \in \indset(\cgraph)$, define the vector $u_S \in \Reals^{\cvert}$ where $u_S(i) = u(i)$ if $i \in S$ and $u_S(i) = *$ if $i \notin S$.
Thus, the experimenter observes $(S, u_S) \in \indset(\cgraph) \times (\Reals_*)^{\cvert}$, where $\Reals_* = \Reals \cup \setb{*}$.
An \emph{estimator} is a measureable function $\eate : \indset(\cgraph) \times (\Reals_*)^{\cvert} \to \Reals$.
We denote the set of all real-valued measurable functions on $\indset(\cgraph) \times (\Reals_*)^{\cvert}$ as $\Gamma(\indset(\cgraph))$.

With this notation in place, we can now define the first minimax problem for labelled graphs.
For a graph $\cgraph$ and a moment parameter $q \geq 2$, define the \emph{independent set minimax risk} $\risk(\cgraph, q)$ as 
\[
\risk(\cgraph, q) = \inf_{\substack{ 
	\design \in \mathbb{D}(\indset(\cgraph)) 
	\\ \eate \in \Gamma(\indset(\cgraph))} 
	} 
	\sup_{u \in \cM(\cgraph, q)}
	\Esub[\Big]{S \sim \design}{ \paren{ \ate(u) - \eate(S, u_S) }^2 }
	\enspace.
\]
The independent minimax risk describes the best achievable mean squared error in this graph estimation problem, defined solely in terms of the labelled graph $\cgraph$.

The second minimax problem is similar, but considers the restriction to observe only maximal independent sets in $\cgraph$.
This restriction is natural{\textemdash}given that we get to observe all outcomes in an independent set, restricting the choice of independent set to be maximal ensures that we observe as much information as possible.
Formally, an independent set $S \in \indset(\cgraph)$ is \emph{maximal} if for any $A \in \indset(\cgraph)$ with $S \subseteq A$, we have that $A = S$.
In other words, $S$ is not contained in a strictly larger independent set.
We denote the set of maximal independent sets in $\cgraph$ by $\maxindset(\cgraph)$.
In an analogous fashion, we denote the set of probability measures on $\maxindset(\cgraph)$ by $\mathbb{D}(\maxindset(\cgraph))$ and the set of real-valued measurable functions on $\maxindset(\cgraph) \times (\Reals_*)^{\cvert}$ as $\Gamma(\maxindset(\cgraph))$.
For a moment parameter $q \geq 2$, we define the \emph{maximal independent set minimax risk} as 
\[
\mirisk(\cgraph, q) = \inf_{\substack{ 
	\design \in \mathbb{D}(\maxindset(\cgraph)) 
	\\ \eate \in \Gamma(\maxindset(\cgraph))} 
	} 
	\sup_{u \in \cM(\cgraph, q)}
	\Esub[\Big]{S \sim \design}{ \paren{ \ate(u) - \eate(S, u_S) }^2 }
	\enspace.
\]
Because this minimax risk considers a restricted class of probability distributions and estimators, it follows directly from the definition that $\risk(\cgraph, q) \leq \mirisk(\cgraph,q)$.

In order to show that the conflict graph $\cgraph$ characterizes the minimax risk of the original network experiment problem $\risk(G, \ate, q)$, we will establish a series of reductions between these three minimax problems:
\[
\risk(G, \ate, q) \geq \risk(\cgraph, q) \geq \mirisk(\cgraph, q) \geq \risk(G, \ate, q)
\enspace,
\]
which will in turn establish equality amongst all three minimax risks.
Each inequality will be established in its own lemma.

In order to aid in these reductions, we establish some basic propositions which draw the necessary connections between outcomes which are observed under network interventions $z \in \Omega$ and independent sets in the corresponding conflict graph $\cgraph$.
For an intervention $z \in \Omega$, define $R_z \subseteq \cvert$ to be the subset of exposures which are observed under $z$, i.e. 
\[
R_z = \setb[\Big]{ e_{i,k} : \expmapi(z) = e_{i,k} } \enspace.
\]
First, we show that each intervention $z \in \Omega$ corresponds to an independent set in the conflict graph $\cgraph$.

\begin{proposition} \label{prop:intervention-to-ind-set}
	Fix an interference network $G$ and contrastive estimand $\ate$.
	For each intervention $z \in \Omega$, the observed exposures $R_z$ form an independent set in the conflict graph $\cgraph$, i.e. $R_z \in \indset(\cgraph)$.
\end{proposition}

\begin{proof}
	Consider two exposures $ e_{i,k}, e_{j,\ell}$ which are both observed under intervention $z \in \Omega$, i.e. $e_{i,k}, e_{j,\ell} \in R_z$.
	Because they are both observed under $z$, they cannot be in conflict, by definition.
	Thus, they are not adjacent in the conflict graph $\cgraph$.
	This establishes that none of the exposures in $R_z$ are adjacent, therefore it is in an independent set.
\end{proof}

It is worth noting that the reverse direction is not necessarily true, in the following sense: there are independent sets $A \in \indset(\cgraph)$ for which there is not intervention $z \in \Omega$ such that $R_z = A$.
However, the following proposition shows that this reverse direction is true when we restrict $A$ to be a maximal independent set.

\begin{proposition} \label{prop:maxindset-to-intervention}
	Fix an interference network $G$ and contrastive estimand $\ate$.
	There exists an injective function $f: \maxindset(\cgraph) \to \Omega$ with the following property: each maximal independent set $A \in \maxindset(\cgraph)$ has a corresponding intervention $z = f(A)$ satisfying $A = R_z$.
\end{proposition}

\begin{proof}
	Fix a maximal independent set $A$.
	We will construct an intervention $z \in \Omega$ such that $R_z = A$.
	Our construction will proceed in an iterative process.
	Suppose that $A$ has $r$ vertices, which we arbitrarily order as $e_{i_1,k_1}, e_{i_2, k_2} \dots e_{i_r, k_r}$.
	Initialize an intervention $z^{(0)} = \paren{0, \dots 0}$, which corresponds to all subject receiving control.
	For $t = 1 \dots r$, define the intervention $z^{(t)}$ from $z^{(t-1)}$ by the update: for all $j \in [n]$,
	\[
	z^{(t)}(j) \gets 
	\left\{
	\begin{array}{lr}
	z^{(t-1)(j)} &\text{if } j \notin \extneigh{i_t} \\
	1 &\text{if } j \in \extneigh{i_t} \text{ and } j \in e_{i_t,k_t}
	\end{array}
	\right.
	\]
	In other words, we iterate over the exposures in $A$, updating the intervention so that all exposures $e_{i_1, k_1} \dots e_{i_t, k_t}$ are realized by intervention $z^{(t)}$.
	An important aspect in this construction is that if $z^{(t)}(j)$ is set to $1$ at some iteration $t$, then it remains set to $1$ throughout the process.
	This is due to the fact that $A$ forms an independent set, and thus all exposures $e_{i_t, k_t} \in A$ agree on the treatment assignments of their extended neighborhoods.
	Once the iterative procedure finishes, we have that all exposures in $A$ are observed simultaneously under intervention $z = z^{(r)}$.

	The argument above establishes that $A \subseteq R_z$ and requires only that $A$ is an independent set in $\cgraph$.
	In order to establish the reverse containments, i.e. that $R_z \subseteq A$, we need to invoke maximality of $A$.
	But this is now simple: because $A \subseteq R_z$, the definition of maximality implies that $A = R_z$.

	The argument above shows that the function $f: \maxindset(\cgraph) \to \Omega$ with the desired property exists, but it remains to be shown that it is injective.
	Consider two unique maximal independent set $A$ and $A'$.
	By maximality of $A$ and $A'$, there must exist vertices $e_{i,k} \in A \setminus A'$ and $e_{j,\ell} \in A' \setminus A$ which are adjacent and thus in conflict.
	In this case, the construction above would assign different individual treatments to at least one common neighbor of $i$ and $j$.
	It follows that $f(A) \neq f(A')$ and thus $f$ is injective.
\end{proof}

Finally, we need to endow the conflict graph with a labelling scheme so that the moment restrictions and contrastive estimand agree between the network experiment problem and the independent set problems.
Given an interference network $G$ and a contrastive effect $\ate$, we will always assign the following labelling to its vertices: exposure $e_{i,k}$ gets the label $i + n*k$.
The following proposition shows that this corresponds to a suitable bijection between the classes $\momodelq$ and $\cM(\cgraph,q)$, under which the estimand of interest is also preserved.
In a slight abuse of notation, we will always be using the more restricted potential outcome model $\cM^0(G,\tau,q)$ considered in Section~\ref{sec:supp_outcome_restrictions}, though we suppress this notation for simplicity since the minimax risks are equal.

\begin{proposition} \label{prop:outcome-equivalence}
	Fix an interference network $G$, a contrastive estimand $\ate$, and let $\cgraph$ be the correspondingly labelled conflict graph.
	The function $\phi : \momodelq \to \cM(\cgraph, q)$ given by
	\[
	\phi( y )(i,k) = y_i(e_{i,k}) 
	\]
	is a well-defined bijection satisfying
	$\ate(y) = \ate(\phi(y))$.
\end{proposition}

\begin{proof}
	The proof is essentially a bookkeeping exercise.
	Both claims follow because the bijection maps the potential outcomes $y_i(e_{i,0})$ to the first half of vertices labelled $1 \dots n$ and the potential outcomes $y_i(e_{i,1})$ to the second half of vertices labelled $n+1 \dots 2n$.
\end{proof}

We are now ready to proceed with the three reductions. 
Each of these reductions relies on a similar proof strategy, the key structure of which we provide now.
Abstractly, we seek to show that the minimax risk for Problem A is greater than that of Problem B.
Given a statistical procedure $(\design, \eate)$ for Problem A, we construct a statistical procedure $(\design^*, \eate^*)$ for Problem B so that the maximum risk of the second procedure is upper bounded by that of the first, establishing the minimax reduction.
Each construction has essentially the following components:
\begin{itemize}
	\item \textbf{Natural Mapping}:
	We show the existence of a natural mapping $f$ from the sample space in Problem A to the sample space in Problem B.
	Importantly, the observed information in Problem B completely determines the observed information in the pre-image of $f$ for Problem A.

	\item \textbf{New Probability Measure}:
	The new probability measure $\design^*$ is defined as the pushforward measure of the original probability measure $\design$ under this mapping, i.e. roughly speaking $\design^*(E) = \design(f^{-1}(E))$.

	\item \textbf{New Estimator}:
	The new estimator $\eate^*$ is defined as the conditional expectation of the original estimator $\eate$ under the pre-image of $f$, i.e. roughly speaking, $\eate^*(E) = \E{\eate(x) \mid x \in f^{-1}(E)}$.
	Because $f$ is such that observed information in Problem B completely determines the observed information in the pre-image of $f$ for Problem A, this is a well-defined estimator for Problem B.

	\item \textbf{Jensen's Inequality}:
	An application of Jensen's inequality shows that the maximum risk of the original procedure  $(\design, \eate)$ for Problem A provides an upper bound on the maximum risk of the new procedure $(\design^*, \eate^*)$ for Problem B.
	Thus, the minimax risk is bounded as desired.
\end{itemize}

We will start with the first reduction, which shows that the minimax risk of the network experiment problem is at most the minimax risk of the independent set problem.

\begin{lemma} \label{lemma:network-to-graphical}
	Fix an interference network $G$, a contrastive estimand $\ate$, and let $\cgraph$ be the corresponding conflict graph.
	For every moment parameter $q \geq 2$, $\risk(G, \ate, q) \geq \risk(\cgraph, q)$.
\end{lemma}
\begin{proof}
	Fix a design-estimator pair $(\design, \eate)$ for the network experiment problem, i.e. $\design \in \mathcal{D}(\Omega)$ is a probability measure on interventions and $\eate \in \Gamma(\Omega)$ is a real-valued measureable function on network outcomes.
	Our goal is to construct a procedure $(\design^*, \eate^*)$ for the independent set problem, where $\design^* \in \mathbb{D}(\indset(\cgraph))$ is a probability measure on independent sets and $\eate^* \in \Gamma(\indset(\cgraph))$ is a measureable function on the observations from an independent set.

	We consider the natural mapping $f: \Omega \to \indset(\cgraph)$ from network interventions to independent sets given by
	Proposition~\ref{prop:intervention-to-ind-set}; namely, that the intervention $z$ can be mapped to the independent set $R_z$ of observed exposures in the conflict graph $\cgraph$. 
	Crucially, all interventions $z$ which map to the same independent set $R_z$ provide the same information about the outcomes.
	
	Let $\design^*$ be the pushforward measure of $\design$ thru $f$. 
	In particular, the new probability measure assigns the following mass to each independent set:
	\[
	\design^*(S) 
	= \sum_{z \in \Omega} \indicator{ f(z) = S } \cdot \design(z)
	= \design(  \setb{ z : f(z) =S } )
	\enspace.
	\]

	The new estimator $\eate^*$ is defined via taking a conditional expectation of the original estimator $\eate$ under the original design $\design$ where the conditioning event is that $\setb{f(z) = S}$, i.e.
	\[
	\eate^*(S, u_S) = \Esub[\Big]{ Z \sim \design }{ \eate(Z, \phi^{-1}(u)(Z)) \mid f(z) = S }
	\enspace,
	\]
	where $\phi : \momodelq \to \cM(\cgraph,q)$ is the bijection described in Proposition~\ref{prop:outcome-equivalence}.
	Note that $\eate^*(S, u_S)$ is in fact a well-defined estimator because the outcomes observed in $u_S$ provide enough information to evaluate $\phi^{-1}(u)(z)$ for each intervention $z$ such that $f(z) = S$.

	Finally, let us bound the risk of the new procedure in terms of the old one.
	Fix functions $y \in \momodelq$ and $u \in \cM(\cgraph, q)$ satisfying $\phi(y) = u$.
	Using Jensen's inequality, we have that
	\begin{align*}
		\Esub[\Big]{ z \sim \design }{ \paren[\big]{ \ate(y) - \eate(z, y(z)) }^2 } 
		&\geq \Esub[\Big]{ z \sim \design }{
			\paren[\Big]{
				\ate(y) - \E{ \eate(z, y(z)) \mid f(z) = S }
			}^2
		} \\
		&= \Esub[\Big]{ S \sim \design^* }{
			\paren[\Big]{
				\ate(u) - \eate^*(S, u_S)
			}
		} \enspace,
	\end{align*}
	where the last line used the definition of the new procedures $(\design^*, \eate^*)$ together with the fact that $\ate(y) = \ate(u)$ for $u = \phi(y)$ (Proposition~\ref{prop:outcome-equivalence}).
	Because $\phi$ is a bijection, we have that the maximum risk of the first procedure $(\design, \eate)$ is an upper bound on the second procedure $(\design^*, \eate^*)$, i.e.
	\[
	\sup_{y \in \momodelq} 
	\Esub{z \sim \design}{ \paren[\big]{ \ate(y) - \eate(z, y(z)) }^2  }
	\geq 
	\sup_{u \in \cM(\cgraph, q)}
	\Esub{S \sim \design^*}{ \paren[\big]{ \ate(u) - \eate(S, u_S) }^2  }
	\enspace.
	\]
	To establish the claim, we consider the infimum over all procedures $(\design, \eate)$ and their mappings to $(\design^*, \eate^*)$. 
\end{proof}

We are now ready to proceed to the next reduction.
This time, we show that the independent set minimax risk is at least as large as the restricted independent set minimax risk.

\begin{lemma} \label{lemma:graphical-to-restricted}
	Fix a labeled graph $\cgraph$.
	For any moment parameter $q \geq 2$, we have that $\risk(\cgraph, q) \geq \mirisk(\cgraph, q)$.
\end{lemma}
\begin{proof}
	Fix a pair $(\design, \eate)$ for the independent set problem where $\design \in \mathbb{D}(\indset(\cgraph))$ is a probability measure on independent sets and $\eate \in \Gamma(\indset(\cgraph))$ is a measureable function of observed outcomes.
	Our goal is to construct a procedure $(\design^*, \eate^*)$ for maximal independent set problem where $\design^* \in \mathbb{D}(\maxindset(\cgraph))$ is a probability measure on maximal independent sets and $\eate^* \in \Gamma(\maxindset(\cgraph))$ is a measureable function on the observations from maximal independent sets.
	
	To this end, consider any mapping $f : \indset(\cgraph) \to \maxindset(\cgraph)$ that maps independent sets to maximal independent sets and also satisfies the additional inclusion principle: for all independent sets $S$, we have that $S \subseteq f(S)$.
	For our purposes, any mapping satisfying the inclusion principle will suffice.
	Crucially, any such mapping satisfies the following property: the outcomes we observe in a maximal independent set $A$ completely determine all the outcome information we would observed in a subset $S \subseteq A$. 

	Let $\design^*$ be the pushforward measure of $\design$ through $f$.
	In particular, the new probability measure assigns the following mass to each maximal independent set:
	\[
	\design*(A) 
	= \sum_{S \in \indset(\cgraph)} \indicator{f(S) = A} \cdot \design(S)
	= \design( ( \setb{ S : f(S) = A } )
	\enspace.
	\]

	The new estimator is defined via taking a conditional expectation of the original estimator $\eate$ under $\design$ where the conditioning event is that $\setb{f(S) = A}$, i.e.
	\[
	\eate^*(A, u_A)
	= \Esub[\Big]{ S \sim \design }{
		\eate(S, u_S) \mid f(S) = A
	}
	\enspace.
	\]
	Note that this is indeed a well-defined estimator because the observed outcomes in the larger independent set $A$ provide enough information to evaluate $u_S$ for each subset $S \subseteq A$.

	Finally, upper bounding the risk of the new procedure in terms of the old one is an application of Jensen's inequality together with the definition of the newly constructed procedure.
	For an arbitrary $u \in \cM(\cgraph,q)$, we have that
	\begin{align*}
		\Esub[\Big]{S \sim \design}{ (\ate(u) - \eate(S,u_S))^2 }
		&\geq \Esub[\Big]{S \sim \design}{ \paren[\Big]{ \ate(u) - \E{ \eate(S, u_S) \mid f(u_S) = A} }^2 } \\
		&= \Esub[\Big]{A \sim \design^*}{  \paren[\Big]{ \ate(u) - \eate^*(A, u_A) }^2 }
	\end{align*}
	The inequality between minimax risks follows by taking the supremum over outcomes and infimum over procedures.
\end{proof}

Finally, we arrive at the third reduction.
This time, we show that the minimax risk of the network experiment problem is at most the minimax risk of the associated maximal independent set problem.
This reduction is more direct and so we do not need to use any conditional expectation nor Jensen's inequality.

\begin{lemma} \label{lemma:restricted-to-network}
	Fix an interference network $G$, a contrastive estimand $\ate$, and let $\cgraph$ be the corresponding conflict graph.
	For every moment parameter $q \geq 2$, $\mirisk(\cgraph, q) \geq \risk(G, \ate, q)$.
\end{lemma}
\begin{proof}
	Fix a procedure $(\design, \eate)$ for the maximal independent set problem where $\design \in \mathcal{D}(\maxindset(\cgraph))$ is a probability measure over maximal independent sets and $\eate \in \Gamma(\maxindset(\cgraph))$ is a real-valued measurable function on maximal independent sets.
	Our goal is to construct a statistical procedure $(\design^*, \eate^*)$ for the network experimental design problem, where $\design^* \in \mathbb{D}(\Omega)$ is a measure on network interventions and $\eate^* \in \Gamma(\Omega)$ is a measurable function on the observations.

	Consider the injective mapping $f : \maxindset(\cgraph) \to \Omega$.
	The function has two key properties: first, it is injective so that maximal independent sets are in bijection to a subset of interventions.
	Second, given a maximal independent set $A$ with intervention $f(A) = z$, it holds that $A$ is exactly the observed outcomes under intervention $z$.
	Thus, the potential outcomes observed under the intervention $z$ exactly correspond to the vertex outcomes that would be observed under $A$. 
	This is the crucial property that we need for this final reduction. 
	
	The new design and estimator are simple: they are essentially reconstructing exactly the procedure $(\design, \eate)$.
	Fix a network intervention $z \in \Omega$.
	If there exists a maximal independent set $A$ such that $f(A) = z$, then assign probability mass $\design^*(z) = \design(A)$.
	Otherwise, if no such $A$ exists then $\design^*(z) = 0$.
	We only need to define the estimator on interventions in the image of $f$.
	For these interventions $z \in f(\maxindset(\cgraph))$, define the new estimator as 
	\[
	\eate^*(z, y(z)) = \eate(A, \phi(y)_A) \enspace,
	\]
	where $A = f(z)$ and $\phi$ is the bijection in Proposition~\ref{prop:outcome-equivalence}.
	Note that $\phi(y)_A$ can be reconstructed exactly from the information observed in $y(z)$, so that $\eate^*$ is indeed a well-defined estimator.

	Now we show that the risks are equivalent.
	Consider potential outcome function $y \in \momodelq$ with $u = \phi(y)$.
	By the definitions above, we have the equality 
	\[
	\Esub[\Big]{A \sim \design}{ (\ate(u) - \eate(A,u_A) )^2  }
	= \Esub[\Big]{z \sim \design^*}{ (\ate(y) - \eate(z,y(z)) )^2  }
	\enspace,
	\]
	where we have used that $\ate(y) = \ate(u)$ from Proposition~\ref{prop:outcome-equivalence}.
	Because $\phi$ is a bijection, the maximum risks of the two procedures are equal:
	\[
	\sup_{u \in \cM(\cgraph,q)} \Esub[\Big]{A \sim \design}{ (\ate(u) - \eate(u,u_A) )^2  }
	= 
	\sup_{y \in \momodelq} \Esub[\Big]{z \sim \design^*}{ (\ate(y) - \eate(z,y(z)) )^2  }
	\enspace.
	\]
	The final inequality between the minimax risks follows by taking the infimum of both sides and using the mapping from $(\design, \eate) \mapsto (\design^*, \eate^*)$.
\end{proof}

After completing these three reductions, we are now ready to prove Proposition~\ref{prop:cg-characterize-minimax}.

\begin{proof}[Proof of Proposition \mainref{prop:cg-characterize-minimax}]
	Fix an interference network $G$ and a constrastive estimand $\ate$ and let $\cgraph$ be the resulting conflict graph.
	Fix a moment parameter $q \geq 2$.
	The lemmas~\ref{lemma:network-to-graphical}, \ref{lemma:graphical-to-restricted}, and \ref{lemma:restricted-to-network} demonstrate the series of inequalities of the three minimax risks
	\[
	\risk(G, \ate, q) \geq \risk(\cgraph, q) \geq \mirisk(\cgraph, q) \geq \risk(G, \ate, q)
\enspace,
	\]
	which in turn establishes that they are all equal.
	Thus, we have that the minimax risk of the network experiment $\risk(G, \ate, q)$ is equal to the minimax risk of the maximum independent set problem $\mirisk(\cgraph, q)$.
	In this sense, we can take $f_q = \mirisk(\cgraph, q)$.
\end{proof}

The proof above show that the minimax risk of the network experiment problem can be completely characterized by the conflict graph, albeit through a similarly complex minimax formulation.
The reductions themselves show that the structure of the maximal independent sets in $\cgraph$ contain the relevant information.

	\section{Lower Bounds on the Minimax Risk}\label{sec:lower_bounds_appendix}
We start by presenting the proof of LeCam's method in Section~\ref{sec:lecam_appendix}. We then continue with the proof of the global lower bound in Section~\ref{sec:global_lower_bound_appendix}, followed by the proof of the local lower bound in Section~\ref{sec:local_lower_bound_appendix}.

\subsection{Le Cam's Method (Proposition~\ref{prop:lecam})}\label{sec:lecam_appendix}

For convenience, we restate the proposition here.

\lecam*

As we discussed in Section~\ref{sec:lower_bound_tools}, the proof proceeds in two steps. In the first step, we argue that if we could solve the estimation problem, we could also test between two different hypotheses about the potential outcome functions. 
In the second step, we connect the error of the best testing procedure for this problem with the total variation distance between the observed distributions under these two hypotheses. 

We start by presenting a Lemma regarding the first step.

\begin{lemma}\label{lem:estimation_to_testing}
    Let $H_0,H_1$ be two distributions over $\cM(G,\tau,q)$ supported on finite subsets $M_0,M_1$ respectively. Let 
    \[
    \tau_0 = \inf_{y \in M_0}\ate(y) \quad \text{and} \quad \tau_1 = \sup_{y \in M_1}\ate(y) \enspace.
    \]
    Assume that $\tau_1 < \tau_0$. Let $(\design,\eate)$ be an arbitrary design-estimator pair for effect $\ate$. Define the test function $\psi : \setb{0,1}^n \times \R^n \to \setb{0,1}$ as
\[
\psi(z,y(z)) = \begin{cases}
0 & \text{if } \eate(z,y(z)) > \frac{\tau_0 + \tau_1}{2} \\
1 & \text{otherwise}
\end{cases}
\]
Then, we have that 
\begin{itemize}
    \item If $y \in M_0$,
    \[
    \design(\psi(Z,y(Z)) = 1) \leq \frac{4}{d(H_0,H_1)^2} \cdot \Esub{Z \sim \design}{\paren{\eate(Z,y(Z)) - \tau(y)}^2} \enspace.
    \]
    \item If $y \in M_1$, 
    \[
    \design(\psi(Z,y(Z)) = 0) \leq \frac{4}{d(H_0,H_1)^2} \cdot \Esub{Z \sim \design}{\paren{\eate(Z,y(Z)) - \tau(y)}^2} \enspace.
    \]
\end{itemize}

\end{lemma}

\begin{proof}
    We only prove the first statement, since the second one follows by an analogous argument.
    First of all, by definition of $\tau_0$ and $\tau_1$, we have that
\[
d(H_0,H_1) = \inf_{y \in M_0, y' \in M_1} \abs{\tau(y) - \tau(y')} = \tau_0 - \tau_1 \enspace.
\]
    By definition of $\psi$, if $\psi(z,y(z)) = 1$, then 
\[
\eate(z,y(z)) \leq \frac{\tau_0 + \tau_1}{2} \enspace.
\]
For any $y \in M_0$, we have that $\tau(y) \geq \tau_0$. Therefore, in the event that $\psi(z,y(z)) = 1$, we have that
\[
\paren[\big]{\eate(z,y(z)) - \tau(y)} \geq \paren[\Big]{\frac{\tau_0 + \tau_1}{2} - \tau_0}^2 = \frac{d(H_0,H_1)^2}{4}\enspace.
    \]
    Thus, using Markov's inequality, we have that 
    \begin{align*}
    \Esub[\Big]{Z \sim \design}{\paren{\eate(Z,y(Z)) - \tau(y)}^2} &\geq \design(\psi(Z,y(Z)) = 1) \cdot \Esub[\Big]{Z \sim \design}{\paren{\eate(Z,y(Z)) - \tau(y)}^2 \mid \psi(Z,y(Z)) = 1} \\
    &\geq \design(\psi(Z,y(Z)) = 1) \cdot \frac{d(H_0,H_1)^2}{4} \enspace.
    \end{align*}
\end{proof}

Note that in Lemma~\ref{lem:estimation_to_testing} the randomness comes only from the treatment assignment $Z$, while the outcomes are fixed.

We are now ready to present the proof of Proposition~\ref{prop:lecam}, which includes the second step of connecting the testing error with the total variation distance between the observed distributions.

\begin{proof}[Proof of Proposition~\ref{prop:lecam}]
    We have that 
    \begin{align*}
        \cR(G,\tau,q,C) & = \inf_{\design,\eate} \sup_{y \in \cM(G,\tau,q)} \Esub[\Big]{Z \sim \design}{\paren[\big]{\eate(Z,y(Z)) - \tau(y)}^2} \\
        & \geq \inf_{\design,\eate} \paren[\Bigg]{\frac{1}{2} \max_{y \in M_0} \Esub[\Big]{Z \sim \design}{\paren[\big]{\eate(Z,y(Z)) - \tau(y)}^2} + \frac{1}{2} \max_{y \in M_1} \Esub[\Big]{Z \sim \design}{\paren[\big]{\eate(Z,y(Z)) - \tau(y)}^2}} \\
        & \geq \inf_{\design,\eate} \paren[\Bigg]{ \frac{1}{2} \Esub[\Big]{y \sim H_0}{\Esub{Z \sim \design}{\paren{\eate(Z,y(Z)) - \tau(y)}^2}} + \frac{1}{2} \Esub[\Big]{y \sim H_1}{\Esub{Z \sim \design}{\paren{\eate(Z,y(Z)) - \tau(y)}^2}}} \\
        & \geq \frac{d(H_0,H_1)^2}{8} \inf_{\design,\eate} \paren[\Bigg]{\Prsub[\Big]{y \sim H_0, Z \sim \design}{\psi(Z,y(Z)) = 1} + \Prsub[\Big]{y \sim H_1, Z \sim \design}{\psi(Z,y(Z)) = 0}} \enspace,
    \end{align*}
    where in the last step we used Lemma~\ref{lem:estimation_to_testing}.
    Now, define the event $E = \setb{(z,y) : \psi(z,y) = 1}$. This event depends on the observed data. Recally that $Q_{H_0}^{(\design)}$ and $Q_{H_1}^{(\design)}$ are the distributions of the observed data under $H_0$ and $H_1$, respectively (see Section~\ref{sec:lower_bound_tools} for the definition). Then, rewriting the last inequality yields
    \begin{align*}
        \cR(G,\tau,q,C) &\geq \frac{d(H_0,H_1)^2}{8} \inf_{\design,\eate} \paren[\Bigg]{Q_{H_0}^{(\design)}(E) + Q_{H_1}^{(\design)}(E^c)} \\
        & = \frac{d(H_0,H_1)^2}{8} \inf_{\design,\eate} \paren[\Bigg]{1 - \paren[\Big]{Q_{H_1}^{(\design)}(E) - Q_{H_0}^{(\design)}(E)}} \\
        & \geq \frac{d(H_0,H_1)^2}{8} \inf_{\design} \paren[\Bigg]{1 - \sup_{A \in \cF \otimes \cB(\R^n)} \paren[\Big]{Q_{H_1}^{(\design)}(A) - Q_{H_0}^{(\design)}(A)}} \\
        \intertext{where the supremum is taken over all measurable sets $A$}\\
        & = \frac{d(H_0,H_1)^2}{8} \inf_{\design} \paren[\Bigg]{1 - \tv\paren[\Big]{Q_{H_0}^{(\design)}, Q_{H_1}^{(\design)}}} \enspace,
    \end{align*}
    where in the last step we used the definition of the total variation distance.
\end{proof}

\subsection{Lower Bound using Global Information (Theorem \ref{thm:global_lower_bound})}\label{sec:global_lower_bound_appendix}
Throughout this section, when we refer to distributions over the outcome functions, they will always be supported on a finite number of elements of $\cM(G,\ate,q)$, as noted in Section~\ref{sec:lower_bound_tools}.
For a distribution $H$ over outcome functions, we use the notation $y \sim H$ to denote that $y$ is drawn according to $H$.
We recall the notation used in Section~\ref{sec:global_lower_bound}. In particular, the following definition is useful.

\begin{definition}
Let $p \in (0,1)$ and $T \subseteq V_\cH$. We define the distribution $H(p,T)$ to be the distribution over outcome functions in $\cM(G,\ate,q)$ that is sampled in the following way.
\begin{itemize}
    \item First, we sample a uniformly random subset of $S \subseteq T$, of size $\lfloor p |T| \rfloor$.
    \item Then, we set all outcomes for all exposures that are not used to define the effect $\ate$ to be $0$, i.e.
    \[
    y_i(z) = 0 \quad \text{for all $i \in [n]$ and $z \in \setb{0,1}^n$ such that $h_i(z) \notin \setb{e_{i,0}, e_{i,1}}$} \enspace.
    \]
    \item For the remaining exposures, we set 
    \[
    y_i(e_{i,k}) = \begin{cases}
\paren[\Big]{\frac{n}{|T|}}^{\frac{1}{q}} & \text{if $e_{i,k} \in S$ and $k = 1$}  \\
-\paren[\Big]{\frac{n}{|T|}}^{\frac{1}{q}} & \text{if $e_{i,k} \in S$ and $k = 0$}  \\
0 & \text{if $e_{i,k} \notin S$}
\end{cases}
    \]
\end{itemize}
\end{definition}

In the sequel, we will always assume that $p$ is such that $p |T|$ is an integer, which can always be done by rounding off $p$.

We will denote by $\cM(p,T)$ the support of $H(p,T)$, for any $p \in (0,1), T \subseteq V_\cH$.
We start by presenting a few immediate properties of the distribution $H(p,T)$.

\begin{lemma}\label{lem:supp_Hp}
    The following are true for all $p,r \in (0,1), T \subseteq V_\cH$, causal effect $\ate$, conflict graph $\cH$ and integer $q \geq 2$:
    \begin{enumerate}[label=(\roman*)]
        \item\label{prop:support} $\cM(p,T) \subseteq \cM(G,\ate,q)$.
        \item\label{prop:ate} For all $y \in \cM(p,T)$, 
        $$\ate(y) = p \cdot \paren[\Bigg]{\frac{|T|}{n}}^{1 - \frac{1}{q}} \enspace.$$
        \item\label{prop:separation} $$d(H(p,T),H(r,T)) = |p-r| \cdot \paren[\Bigg]{\frac{|T|}{n}}^{1 - \frac{1}{q}} \enspace.$$
    \end{enumerate}
\end{lemma}

\begin{proof}
    Property~\ref{prop:support} follows by noticing that for every $y \in \cM(p)$ we have $|y_i(e_{i,k})| \leq (n/|T|)^{1/q}$ for all $i \in [n]$ and $k \in \setb{0,1}$.
Therefore, we have that
\[
\frac{1}{n} \sum_{i=1}^n \paren[\Big]{|y_i(e_{i,1})|^q + |y_i(e_{i,0})|^q} \leq \frac{|T|}{n} \cdot \frac{n}{|T|} = 1 \enspace.
\]
Property~\ref{prop:ate} follows by noticing that for every $y \in \cM(p,T)$, we have that 
\begin{align*}
\ate(y) &= \frac{1}{n} \sum_{i=1}^n \paren[\Big]{y_i(e_{i,1}) + y_i(e_{i,0})} \\
&= \frac{1}{n} \sum_{i=1}^n \paren[\Big]{\indicator{e_{i,1} \in S}\cdot \paren[\Big]{\frac{n}{|T|}}^{1/q}  - \indicator{e_{i,0} \in S } \cdot \paren[\Big]{-\paren[\Big]{\frac{n}{|T|}}^{1/q}}} \\
&= \frac{1}{n} \cdot |S| \cdot \paren[\Big]{\frac{n}{|T|}}^{1/q} \\
&= p \cdot \paren[\Bigg]{\frac{|T|}{n}}^{1 - \frac{1}{q}} \enspace.
\end{align*}
Property~\ref{prop:separation} follows immediately from the second.
\end{proof}

For an arbitrary distribution on the outcome functions $H$ and an arbitrary design $\design$, we denote by $Q_H^{(\design)}$ the distribution of the observed outcomes and treatment assignment when the potential outcomes are drawn from $H$ and the treatment assignment is drawn from $\design$ independently, i.e. the pushforward of the product measure $H \otimes \design$ under the map $(z,y)\mapsto (z,y(z))$.
Note that the distribution $Q_H^{(\design)}$ depends on the design $\design$. Also, for an intervention $z \in \setb{0,1}^n$, we denote by $Q_H^{(\design)}(\cdot | z)$ the conditional distribution of the observed outcomes conditioned on the treatment assignment being equal to $z$. 
The next proposition shows that this conditional distribution does not depend on the design $\design$, as long as $z$ has positive probability under $\design$.

\begin{proposition}\label{prop:conditional_independence}
    For any distribution $H$ over the outcome functions, any designs $\design$ and any intervention $z \in \setb{0,1}^n$ with $\design(z) > 0$, we have that for any $u \in \R^n$,
    $$
    Q_H^{(\design)}(u | z) = \Prsub{y \sim H}{y(z)= u} \enspace.
    $$
\end{proposition}
\begin{proof}
    By definition, we have that for any $u \in \R^n$,
    \begin{align*}
         Q_H^{(\design)}(u | z) &= \Pr{ y(Z) = u \mid Z = z } \\
         &= \frac{\Pr{y(Z) = u, Z = z }}{\Pr{Z = z}}\\
         &= \frac{\Prsub{y \sim H}{y(z)= u} \design(z)}{\design(z)} \\
         \intertext{using the independence of the sampled $y$ from the sampled $Z$}\\
         &= \Prsub{y \sim H}{y(z)= u} \enspace.
    \end{align*}
\end{proof}

Since the conditional distribution does not depend on the design, we will simply denote it by $Q_H(\cdot | z)$ from now on, i.e.
\[
Q_H(u | z) \triangleq \Prsub{y \sim H}{y(z)= u} \enspace.
\]

Property~\ref{prop:ate} of Lemma~\ref{lem:supp_Hp} shows that each distribution $H(p,T)$ is supported on outcomes that have a fixed value of $\ate$.
Therefore, if we could observe all counterfactual outcomes, we could easily distinguish between $H(p,T)$ and $H(q,T)$ for any $p \neq q$.
However, the fundamental problem of causal inference is that we always observe a subset of the outcomes. As defined in Section~\ref{sec:global_lower_bound}, for any $z \in \setb{0,1}^n$, we denote by $R_z$ the subset of the $2n$ exposures that we get to observe when the treatment assignment is $z$, i.e.
\[
R_z = \setb[\bigg]{h_i(z) : h_i(z) \in \setb{e_{i,0}, e_{i,1}}, i \in [n]} \enspace.
\]

We now state and prove a lemma, which quantifies the tradeoff between separating the two distributions $H(p,T), H(q,T)$ and distinguishing between them from the observed outcomes. This is a generalization of Lemma~\ref{lem:tv_bound} that appeared in the main text for the case $T = V_\cH$.

\begin{lemma}\label{lem:tv_bound_general}
    Let $z \in \setb{0,1}^n$ be an intervention, $T \subseteq V_\cH$ an arbitrary subset with $|T| \geq 5$ and suppose $|R_z \cap T| \leq |T|/2$. 
    Let $H_0 = H(1/2, T)$ and $H_1 = H(1/2 + \delta, T)$, where
\[
\delta \leq \frac{0.026}{\sqrt{|R_z \cap T|}}\enspace.
\]
Then,
\[
d_{\mathrm{TV}}(Q_{H_0}^{(z)}, Q_{H_1}^{(z)}) \leq \frac{7}{8} \enspace.
\]
\end{lemma}

\begin{proof}
    We will assume that 
    \[
    \delta \leq \frac{C}{\sqrt{|R_z \cap T|}}
    \]
    where $C>0$ is a constant that will be determined later according to the constraints.

    Any point $u \in \R^n$ that is in the support of $Q_{H_0}^{(z)}$ or $Q_{H_1}^{(z)}$ must have the property that $u_i \in \setb{-(n/|T|)^{1/q},0,(n/|T|)^{1/q}}$ for all $i \in [n]$, since these are all the possible outcomes under both distributions.
    Let $U = y(z)$ be the random vector of observed outcomes under $Q_{H_0}^{(z)}$ or $Q_{H_1}^{(z)}$. Let $S$ be the random set that is sampled under $H_0$ or $H_1$. Then, $U$ is uniquely determined by $S \cap (R_z \cap T)$. Indeed, if we know $S \cap (R_z \cap T)$, then we have
\[
y_i(z) = \begin{cases}
(n/|T|)^{1/q} & \text{if $e_{i,1} \in S \cap (R_z \cap T)$}  \\
-(n/|T|)^{1/q} & \text{if $e_{i,0} \in S \cap (R_z \cap T)$}  \\
0 & \text{otherwise}
\end{cases}
\]
    Hence, $U$ is determined by $S \cap (R_z \cap T)$. Conversely, if we know $U$, then we can determine $S \cap (R_z \cap T)$ as follows:
\[
S \cap (R_z \cap T) = \setb[\Big]{e_{i,1} : i \in [n] \text{ s.t. } y_i(z) = (n/|T|)^{1/q}} \cup \setb[\Big]{e_{i,0} : i \in [n] \text{ s.t. } y_i(z) = -(n/|T|)^{1/q}} \enspace.
\]
Therefore, the distribution of $U$ is the same as the distribution of $S \cap (R_z \cap T)$.
This allows us to compute the total variation distance between $Q_{H_0}(\cdot|z)$ and $Q_{H_1}(\cdot|z)$ as follows:
\begin{align*}
    d_{\mathrm{TV}}(Q_{H_0}(\cdot|z), Q_{H_1}(\cdot|z)) &= \frac{1}{2} \sum_{A \subseteq R_z \cap T} \abs[\Big]{\Prsub{H_1}{S \cap (R_z \cap T) = A} - \Prsub{H_0}{S \cap (R_z \cap T) = A}}\\
    &= \frac{1}{2} \sum_{A \subseteq R_z \cap T}  \Prsub{H_0}{S \cap (R_z \cap T) = A}  \abs[\Bigg]{\frac{\Prsub{H_1}{S \cap (R_z \cap T) = A}}{\Prsub{H_0}{S \cap (R_z \cap T) = A}} - 1} \enspace.
\end{align*}
Let us now compute these probabilities. For convenience we denote $s = |R_z \cap T|$ and $t = |T|$. Suppose $|A| = k$ and we will restrict our attention to $k$ such that $|k - s/2| \leq \sqrt{s/2}$. 
We will show later that it is enough to restrict to such subsets $A$. 

Then, under $H_0$ we have $|S| = t/2$. There are ${t - s \choose t/2 - k}$ ways to choose the remaining $t/2 - k$ elements of $S$ that are not in $R_z \cap T$, which leads to
\[
\Prsub{H_0}{S \cap (R_z \cap T) = A} = \frac{{t - s \choose t/2 - k}}{{t \choose t/2}} \enspace.
\]
By assumption,  $s \leq t/2$.
This means that $t - s \geq t/2$, and hence the above probability is positive for all $A$ with $|A| \leq s$.

On the other hand, under $H_1$ we have $|S| = t/2 + \delta \cdot t = t(1/2 + \delta)$.
There are ${t - s \choose t(1/2 + \delta) - k}$ ways to choose the remaining $t(1/2 + \delta) - k$ elements of $S$ that are not in $R_z \cap T$, which leads to
\[
\Prsub{H_1}{S \cap (R_z \cap T) = A} = \frac{{t - s \choose t(1/2 + \delta) - k}}{{t \choose t(1/2 + \delta)}} \enspace.
\]
Note that for this probability to be nonzero, we need $t(1/2 + \delta) - k \leq t - s$, since we need to have enough elements outside of $R_z \cap T$ to choose from. 
Using the fact that $k \geq s/2 - \sqrt{s/2}$, it is sufficient to have 
\[
\frac{s}{2} + \sqrt{\frac{s}{2}} \leq t\paren[\Big]{\frac{1}{2} - \delta}
\]
Also, using the assumption $s \leq t/2$, it sufficient to have
\[
\frac{t}{4} + \sqrt{\frac{t}{4}} \leq t\paren[\Big]{\frac{1}{2} - \delta} \enspace.
\]
Since $\delta \le C$, for the above to be satisfied, it is sufficient to have
\begin{equation}\label{eq:constraint_1}
C \leq \frac{1}{4} - \frac{1}{2\sqrt{t}} \enspace.
\end{equation}
For the right hand side to be positive, we need $t \geq 5$.
We will choose $C$ such that~\eqref{eq:constraint_1} holds, so for now we can assume that for all $k$ such that $|k - s/2| \leq \sqrt{s/2}$, both probabilities under $H_0$ and $H_1$ are nonzero.

For such a subset $A$ of size $k$, we have
\begin{align*}
    \abs[\Bigg]{\frac{\Prsub{H_1}{S \cap (R_z \cap T) = A}}{\Prsub{H_0}{S \cap (R_z \cap T) = A}} - 1} &= \abs[\Bigg]{\frac{{t - s \choose t(1/2 + \delta) - k}}{{t \choose t(1/2 + \delta)}} \cdot \frac{{t \choose t/2}}{{t - s \choose t/2 - k}} - 1} \\
    &= \abs[\Bigg]{\frac{\prod_{l=0}^{s-k-1}(t(1/2 - \delta) - l) \cdot \prod_{l=0}^{k-1}(t(1/2 + \delta) - l)}{\prod_{l=0}^{s-k-1}(t/2 - l) \cdot \prod_{l=0}^{k-1}(t/2-l)} - 1}\\
    &= \abs[\Bigg]{\prod_{l=0}^{s-k-1} \paren[\Big]{1 - \frac{2\delta t}{t-2l}} \cdot \prod_{l=0}^{k-1} \paren[\Big]{1 + \frac{2\delta t}{t-2l}} - 1}
\end{align*}

We would like to turn these products into a single exponential factor. To do that, we will utilize the following elementary inequality, which can be proved via taking derivatives.
\begin{equation}\label{eq:exp_inequality}
1 + x = \exp(x + b(x)) \quad \text{where} \quad |b(x)| \leq 2x^2 \quad \text{for all $|x| \leq 1/2$} \enspace.
\end{equation}

To apply this inequality, we need to ensure that the absolute value of each factor in the products is at most $1/2$.
For all $l \leq k - 1$ we have
\begin{align}
\abs[\Bigg]{\frac{2\delta t}{t-2l}} &\leq \frac{2\delta t}{t - 2(k-1)} \nonumber \\
&\leq \frac{2\delta t}{t - 2(s/2 + \sqrt{s/2}-1)} \nonumber \\
&\leq \frac{2\delta t}{t/2 -\sqrt{t} + 2} \nonumber\\
&\leq C  \frac{2t}{t/2 -\sqrt{t} + 2} \label{eq:less12}
\end{align}
Therefore, we have the following constraint on $C$.
\begin{equation}
\label{eq:constraint_2}
C \leq \frac{1}{8} - \frac{1}{4\sqrt{t}} + \frac{1}{2t}\enspace.
\end{equation}

We will choose $C>0$ so that~\eqref{eq:constraint_2} is satisfied.
The same exact constraint arises in the case $l \leq s - k - 1$.
Therefore, applying inequality~\eqref{eq:exp_inequality} to each factor in the products, we have that 
\begin{align*}
&\prod_{l=0}^{s-k-1} \paren[\Big]{1 - \frac{2\delta t}{t-2l}} \cdot \prod_{l=0}^{k-1} \paren[\Big]{1 + \frac{2\delta t}{t-2l}}  \\
&= \exp\paren[\big]{\delta_1(t) + \delta_2(t)} \enspace,
\end{align*}
where
\begin{align*}
\delta_1(t)
&= -\sum_{l=0}^{s-k-1} \frac{2\delta t}{t-2l} + \sum_{l=0}^{k-1} \frac{2\delta t}{t-2l}\\
\delta_2(t) &= \sum_{l=0}^{s-k-1} b\paren[\Big]{\frac{2\delta t}{t-2l}} + \sum_{l=0}^{k-1} b\paren[\Big]{\frac{2\delta t}{t-2l}}
\end{align*}

Let us bound $|\delta_1(t)|$ first.
Assume without loss of generality that $k \geq s/2$, so that $k-1 \geq s-k-1$. The other case is symmetric. Then, we have
\begin{align*}
|\delta_1(t)| &=
\abs[\Bigg]{-\sum_{l=0}^{s-k-1} \frac{2\delta t}{t-2l} + \sum_{l=0}^{k-1} \frac{2\delta t}{t-2l}}\\
 &= 
\abs[\Bigg]{\sum_{l=s-k}^{k-1} \frac{2\delta t}{t-2l}} \\
&\leq \abs{2k - s} \cdot \frac{2\delta t}{t-2(k-1)} \\
&\leq \frac{2 \sqrt{2s}\delta t}{t - 2(s + s/2 -1)}\\
\intertext{using $k \leq s/2 + \sqrt{s/2}$}\\
&\leq \frac{2\sqrt{2}  t}{t/2 - \sqrt{t} + 2} \cdot C \enspace.
\end{align*}

We now bound the second error term.
\begin{align*}
|\delta_2(t)| &=
\abs[\Bigg]{\sum_{l=0}^{s-k-1} b\paren[\Big]{\frac{2\delta t}{t-2l}} + \sum_{l=0}^{k-1} b\paren[\Big]{\frac{2\delta t}{t-2l}}} \\
&\leq 2 \cdot \sum_{l=0}^{s-k-1} \paren[\Bigg]{\frac{2\delta t}{t-2l}}^2 + 2 \cdot \sum_{l=0}^{k-1} \paren[\Bigg]{\frac{2\delta t}{t-2l}}^2 \\
&\leq 8(s-k)\delta^2 t^2 \cdot \frac{1}{(t-2(s-k-1))^2} + 8k\delta^2 t^2 \cdot \frac{1}{(t-2(k-1))^2} \\
&\leq 8 s \delta^2  \cdot \frac{t^2}{(t/2 - \sqrt{t} + 2)^2} \\
&\leq C^2 \cdot \frac{8t^2}{(t/2 - \sqrt{t} + 2)^2}\enspace,
\end{align*}
where we have used again the facts $k \leq s/2 + \sqrt{s/2}$ and $s \leq t/2$.
As a result, we have that for any $k$ such that $|k - s/2| \leq \sqrt{s}$,
\begin{align*}
\abs[\Bigg]{\frac{\Prsub{H_1}{S \cap (R_z \cap T) = A}}{\Prsub{H_0}{S \cap (R_z \cap T) = A}} - 1} &= \abs[\Bigg]{\exp\paren[\Bigg]{\delta_1(t) + \delta_2(t)}-1} \\
&\leq \exp\paren[\Bigg]{\abs{\delta_1(t)} + \abs{\delta_2(t)}}-1
\end{align*}
using the fact that $|e^x - 1| \leq e^{|x|} - 1$ for all $x$.
We would like to establish that
\begin{equation}\label{eq:prob_goal}
  \abs[\Bigg]{\frac{\Prsub{H_1}{S \cap (R_z \cap T) = A}}{\Prsub{H_0}{S \cap (R_z \cap T) = A}} - 1} < \frac{1}{2}\enspace,  
\end{equation}

which by the previous calculations is implied by the following inequality
\[
\abs{\delta_1(t)} + \abs{\delta_2(t)} < \log(3/2) \enspace.
\]
Let us denote 
\[
D(t) = \frac{t}{t/2 - \sqrt{t} + 2} \enspace.
\]
Then, combining the bounds on $|\delta_1(t)|$ and $|\delta_2(t)|$, it is sufficient to have
\[
2\sqrt{2}D(t) C + 8 D(t)^2 C^2 < \log(3/2) \enspace.
\]
Solving this quadratic inequality for $C$, we have that it is sufficient to have
\begin{equation}\label{eq:constraint_3}
C < \frac{-\sqrt{2}+ \sqrt{2 + 8  \log(3/2)}}{8 D(t)} \enspace.
\end{equation}

Thus, inequality~\eqref{eq:prob_goal} will hold as long as the three constraints~\eqref{eq:constraint_1},~\eqref{eq:constraint_2} and~\eqref{eq:constraint_3} are satisfied for $C>0$ and $t \geq 5$.
Since all the constraints are improving as $t$ increases, it is sufficient to check that they are satisfied for $t = 5$.
Substituting $t = 5$ into the right hand sides of the three constraints, we get the conditions
\begin{align*}
    C \leq 0.026 \\
    C \leq 0.113 \\
    C \leq 0.049
\end{align*}
Among these, we choose the most restrictive one, which is $C \leq 0.026$.

We next show that the contribution of the sets $A$ with $|k - s/2| \leq \sqrt{s}$ to the total variation distance is non-negligible.
To do that, it suffices to show that the total probability of such sets under $H_0$ is lower bounded by a constant. 
We can prove this is the case using Chebyshev's inequality.
In particular, we have that
\[
|A| = \sum_{i \in R_z \cap T} \indicator{i \in S}
\]
Let us calculate the expectation and variance of this random variable under $H_0$. By linearity of expectation
\[
\Esub{H_0}{|A|} = \sum_{i \in R_z \cap T} \Prsub{H_0}{i \in S} = |R_z \cap T| \cdot \frac{1}{2} = \frac{s}{2} \enspace,
\]
using the fact that $\Prsub{H_0}{i \in S} = 1/2$ for all $i \in R_z \cap T$. 
We also have
\begin{align*}
\Varsub{H_0}{|A|} &= \sum_{i \in R_z \cap T} \Varsub{H_0}{\indicator{i \in S}} + \sum_{i \neq j \in R_z \cap T} \Covsub{H_0}{\indicator{i \in S}, \indicator{j \in S}} \\
&= \frac{s}{4} + s(s-1)\paren[\Big]{\frac{t/2(t/2 - 1)}{t(t-1)} - \frac{1}{4}} \\
&\leq \frac{s}{4} \enspace.
\end{align*}
Therefore, by Chebyshev's inequality, we have that
\[
\Prsub[\Bigg]{H_0}{\abs[\Big]{|A| - \frac{s}{2}} \geq \sqrt{\frac{s}{2}}} \leq \frac{\Varsub{H_0}{|A|}}{\frac{s}{2}} \leq \frac{1}{2} \enspace.
\]
We are now ready to upper bound the total variation distance. In particular, define the events $E,F$ as follows
\begin{align*}
&E = \setb[\Bigg]{A:\abs[\Big]{|A| - \frac{s}{2}} \leq \sqrt{\frac{s}{2}}} \\
&F = \setb[\Bigg]{A : \Prsub{H_0}{S \cap (R_z \cap T) = A} \geq \Prsub{H_1}{S \cap (R_z \cap T) = A}} 
\end{align*}

We have established that $\Prsub{H_0}{E} \geq 1/2$. This means that $\Prsub{H_0}{E \cap F} \geq 1/4$ or $\Prsub{H_0}{E \cap F^c} \geq 1/4$. Without loss of generality, let's assume the first is true, otherwise a similar argument applies.
Then, we can write the total variation distance as
\begin{align*}
d_{\mathrm{TV}}(Q_{H_0}(\cdot|z), Q_{H_1}(\cdot|z)) &= \sum_{A \in F} \Prsub{H_0}{S \cap (R_z \cap T) = A} - \Prsub{H_1}{S \cap (R_z \cap T) = A} \\
&= \sum_{A \in E \cap F} \paren[\Big]{\Prsub{H_0}{S \cap (R_z \cap T) = A} - \Prsub{H_1}{S \cap (R_z \cap T) = A}} \\
&\quad\quad+ \sum_{A \in E^c \cap F} \paren[\Big]{\Prsub{H_0}{S \cap (R_z \cap T) = A} - \Prsub{H_1}{S \cap (R_z \cap T) = A}} \\
&= \sum_{A \in E \cap F} \Prsub{H_0}{S \cap (R_z \cap T) = A} \cdot \paren[\Bigg]{1 - \frac{\Prsub{H_1}{S \cap (R_z \cap T) = A}}{\Prsub{H_0}{S \cap (R_z \cap T) = A}}}\\
&\quad\quad + \sum_{A \in E^c \cap F} \paren[\Big]{\Prsub{H_0}{S \cap (R_z \cap T) = A} - \Prsub{H_1}{S \cap (R_z \cap T) = A}} \\
&< \frac{1}{2} \cdot \Prsub{H_0}{E \cap F} + \Prsub{H_0}{E^c \cap F} \\
&= \Prsub{H_0}{F} - \frac{1}{2} \cdot \Prsub{H_0}{E \cap F} \\
&\leq 1 - \frac{1}{2} \cdot \frac{1}{4}\\
&= \frac{7}{8} \enspace.
\end{align*}

This concludes the proof.
\end{proof}

We are now ready to prove Theorem~\ref{thm:global_lower_bound}.

\begin{proof}[Proof of Theorem~\ref{thm:global_lower_bound}]
Let's fix an arbitrary subset $T \subseteq V_\cH$ of size $t = |T| \geq 5$. Recall that we denote by $\cI(T,\cH)$ the largest independent set in the subgraph of the conflict graph $\cH$ induced by $T$. Assume for now that $|\cI(T,\cH)| \leq |T|/2$.
We define $H_0 = H(1/2,T)$ and $H_1 = H(1/2 + \delta,T)$ for $\delta = 0.026/\sqrt{|\cI(T,\cH)|}$.
By property~\ref{prop:support} of Lemma~\ref{lem:supp_Hp}, we have that the support of $H_0$ and $H_1$ is contained in $\cM(G,\tau,q)$ for any $q \geq 2$.

Next, we establish that for any $z \in \setb{0,1}^n$, we have $|R_z \cap T| \leq |\cI(T,\cH)|$. It suffices to show that the exposures in $R_z \cap T$ is an independent set in the induced subgraph $\cH(T)$.
To see that, note that for any $e_{i,k}, e_{j,l} \in R_z \cap T$, by definition we have $h_i(z) = e_{i,k}$ and $h_j(z) = e_{j,l}$. This means that $e_{i,k},e_{j,l}$ are not connected in the conflict graph $\cH$, thus also not connected in the induced subgraph $\cH(T)$. Thus $R_z \cap T$ is an independent set of $\cH(T)$.

Let $\design$ be an arbitrary design and lett $\mathrm{supp}(\design)\subseteq \setb{0,1}^n$ be the support of $\design$.
As we have remarked earlier, all outcomes in the support of $H_0$ and $H_1$ take values in the set 
\[
\setb{-(|T|/n)^{-1/q},0,(|T|/n)^{-1/q}}\enspace.
\]
Therefore, the observed outcomes $y(Z)$ under $H_0$ or $H_1$ take values in the set 
\[
U_n = \setb{-(|T|/n)^{-1/q},0,(|T|/n)^{-1/q}}^n\enspace.
\]
We can write the total variation distance between the observed distributions under $H_0$ and $H_1$ as follows
\begin{align*}
d_{\mathrm{TV}}(Q_{H_0}^{(\design)}, Q_{H_1}^{(\design)}) &= \frac{1}{2} \sum_{z \in \setb{0,1}^n} \sum_{u \in U_n} \abs[\Bigg]{\Prsub{H_1}{Z = z, y(Z) = u} - \Prsub{H_0}{Z = z, y(Z) = u}} \\
&= \frac{1}{2} \sum_{z \in \mathrm{supp}(\design)} \design(z) \cdot \sum_{u \in U_n} \abs[\Bigg]{\Prsub{H_1}{y(z) = u | Z = z} - \Prsub{H_0}{y(z) = u | Z = z}} \\
&= \sum_{z \in \mathrm{supp}(\design)} \design(z) \cdot d_{\mathrm{TV}}(Q_{H_0}(\cdot|z), Q_{H_1}(\cdot|z)) \\
&\leq \sum_{z \in \mathrm{supp}(\design)} \design(z) \cdot \frac{7}{8}\\
\intertext{using Lemma~\ref{lem:tv_bound} and the fact that $|R_z \cap T| \leq |\cI(T)|$ and $|\cI(T)| \leq |T|/2$}\\
&= \frac{7}{8} \enspace.
\end{align*}
By property~\ref{prop:separation} of Lemma~\ref{lem:supp_Hp} we have that
\[
d(H_0,H_1) = \delta \paren[\Big]{\frac{|T|}{n}}^{1 - \frac{1}{q}} = \frac{0.026}{\sqrt{|\cI(T)|}}\cdot \paren[\Big]{\frac{|T|}{n}}^{1 - \frac{1}{q}} \enspace.
\]

We will use LeCam's method (Proposition~\ref{prop:lecam}) with the same choice of the two distributions $H_0$ and $H_1$ for all designs $\design$. Applying the bound yields
\[
\cR(G,\ate,q) \geq \frac{d(H_0,H_1)^2}{8} \cdot \inf_\design \paren[\Big]{1 - d_{\mathrm{TV}}(Q_{H_0}^{(\design)}, Q_{H_1}^{(\design)})} \geq 1.05 \cdot 10^{-5} \cdot \frac{1}{|\cI(\cH)|} \cdot \paren[\Big]{\frac{|T|}{n}}^{2 - \frac{2}{q}}  \enspace.
\]
This proves the desired lower bound for the case $|\cI(T,\cH)| \leq |T|/2$.
We now handle the case $|\cI(T,\cH)| > |T|/2$. 
First of all, by assumption we have that $(e_{i,0}, e_{i,1}) \in E(\cH(T))$ for all $i \in [n]$, which means that any independent set in $\cH$ contains at most one of the two exposures for each unit. Therefore, we have $|R_z| \leq n$.
This means that $T = V_\cH$ satisfies $|\cI(T,\cH)| \leq |T|/2$, so applying the previously obtained bound with $T = V_\cH$ yields
\[
\cR(G,\ate,q) \geq 1.05 \cdot 10^{-5} \cdot 2^{2 - \frac{2}{q}} \cdot \frac{1}{n} \enspace.
\]
Now, consider an arbitrary subset $T \subseteq V_\cH$ with $|\cI(T,\cH)| > |T|/2$.
We have that 
\begin{align*}
    1.05 \cdot 10^{-5} \cdot \frac{1}{|\cI(T,\cH)|} \cdot \paren[\Big]{\frac{|T|}{n}}^{2 - \frac{2}{q}} &\leq  
    1.05 \cdot 10^{-5} \cdot \frac{2}{|T|} \cdot \paren[\Big]{\frac{|T|}{n}}^{2 - \frac{2}{q}} \\
    &= 1.05 \cdot 10^{-5} \cdot \frac{2|T|^{1 - \frac{2}{q}}}{n^{2 - \frac{2}{q}}} \\
    &\leq 1.05 \cdot 10^{-5} \cdot 2^{2 - \frac{2}{q}} \cdot \frac{1}{n} \\
    &\leq \cR(G,\ate,q) \enspace,
\end{align*}
where in the last step we used the previously obtained bound for $T = V_\cH$.
Thus, the bound holds for all subsets $T \subseteq V_\cH$ with $|T| \geq 5$.
The proof is now complete.
\end{proof}

\subsection{Lower Bound using Local Information (Theorem~\ref{thm:local_lower_bound_general})}\label{sec:local_lower_bound_appendix}

In this section, we present the proof of Theorem~\ref{thm:local_lower_bound_general}. Let's recall that for any exposure $e_{i,k} \in T$, we denote by $d_{i,k}$ the degree of $e_{i,k}$ in the conflict graph $\cH$. We also denote by $d_{\mathrm{min}}(T)$ the minimum degree of any exposure in $T$, and by $\dhar( T,r)$ the harmonic mean of the degrees of the exposures in $T$ with parameter $r$, i.e.
\[
d_{\mathrm{min}}(T) = \min_{e_{i,k} \in T} d_{i,k} \quad , \quad \dhar( T,r) = \frac{\abs{T}}{\paren[\Big]{\sum_{e_{i,k} \in T} 1 / d_{i,k}^r}^{\frac{1}{r}}} \enspace.
\]
We omit the dependence on $\cH$ since it is clear from the context.
We repeat the statement for completeness.

\localLowerBoundGeneral*

Before proving Theorem~\ref{thm:local_lower_bound_general}, we establish an intermediate Lemma that contains the main idea of the argument. It establishes a lower bound on the minimax risk in terms of some tunable parameters $c_{i,k}$.

\begin{lemma}\label{lem:lower_bound_tunable}
    Let $G$ be a graph on $n$ vertices, $\tau$ an arbitrary effect and $\cH$ the conflict graph. 
    Then, for any subset of exposures $T \subseteq V_\cH$, any integer $q \geq 2$, any integers $r_{i,k}$ that satisfy $r_{i,k} \leq d_{i,k}$ for all $i \in [n], k \in \setb{0,1}$ and any positive numbers $c_{i,k} \in (0,1)$ for $i \in [n], k \in \setb{0,1}$ that satisfy
    \[
    \sum_{e_{i,k} \in T} c_{i,k} \geq \lowerboundconst \enspace,
    \]
    we have
    \[
    \cR(G,\tau,q) \geq \frac{\riskconst}{n^{2 - \frac{2}{q}}} \cdot \min \paren[\Bigg]{\min_{e_{i,k} \in T} c_{i,k} r_{i,k}^{2-\frac{2}{q}} , \frac{|T|^{2 - \frac{2}{q}}}{\sum_{e_{i,k} \in T} c_{i,k} }}\enspace.
    \]
\end{lemma}

\begin{proof}
    We define the event $A_{i,k}$ as the event that $i$ gets exposure $e_{i,k}$, i.e.
    \[
    A_{i,k} = \setb[\Big]{h_i(z) = e_{i,k}}
    \]
    
    Recall that for each exposure $e_{i,k} \in T$, $\mathcal{N}(e_{i,k})$ is the set of neighboring exposures of $e_{i,k}$ in $\cH$. By definition, $|\mathcal{N}(e_{i,k})| = d_{i,k}$. Consider an arbitrary subset $U_{i,k} \subseteq \mathcal{N}(e_{i,k})$ of size $r_{i,k}$ for each $e_{i,k} \in T$. 

    Let $\design$ be an arbitrary design. We distinguish between two cases.

    \paragraph{Case 1:} There exists $i \in [n], k \in \setb{0,1} $ with $e_{i,k} \in T$ such that $\design(A_{i,k}) \geq c_{i,k}$. We define $H_0$ and $H_1$ to be supported on potential outcome functions $y^{(0)},y^{(1)}$ respectively, where
\[
y_i^{(1)}(e) = \begin{cases}
    \paren[\Big]{\frac{n}{r_{i,k}}}^{1/q} & \text{if } e = e_{j,1} \in U_{i,k} \\
    -\paren[\Big]{\frac{n}{r_{i,k}}}^{1/q} & \text{if } e = e_{j,0} \in U_{i,k} \\
    0 & \text{otherwise}
\end{cases}
\quad , \quad
y_i^{(0)}(e) = \begin{cases}
-\paren[\Big]{\frac{n}{r_{i,k}}}^{1/q} & \text{if } e = e_{j,1} \in U_{i,k} \\
\paren[\Big]{\frac{n}{r_{i,k}}}^{1/q} & \text{if } e = e_{j,0} \in U_{i,k} \\
0 & \text{otherwise}
\end{cases} \enspace.
\]
Let us first verify that these outcome functions satisfy the moment restriction. We have for $l \in \setb{0,1}$,
\[
\sum_{j=1}^n \paren[\Big]{\abs[\big]{y_j^{(l)}(e_{j,1})}^q + \abs[\big]{y_j^{(l)}(e_{j,0})}^q} = \sum_{e \in U_{i,k}} \paren[\Bigg]{\paren[\Bigg]{\frac{n}{r_{i,k}}}^{1/q}}^q = n \enspace.
\]
Therefore, $y^{(0)},y^{(1)} \in \cM(G,q)$. 
Since $H_0$ and $H_1$ are supported on single potential outcome functions, their separation is simply
\begin{align*}
d(H_0,H_1) &= |\ate(y^{(0)}) - \ate(y^{(1)})| \\
&= \frac{1}{n} \sum_{e \in U_{i,k}} \paren[\Bigg]{\paren[\Bigg]{\frac{n}{r_{i,k}}}^{1/q} - \paren[\Bigg]{-\paren[\Bigg]{\frac{n}{r_{i,k}}}^{1/q}}} \\
&= 2 \cdot \paren[\Bigg]{\frac{r_{i,k}}{n}}^{1 - \frac{1}{q}} \enspace.
\end{align*}

Suppose event $A_{i,k}$ occurs, i.e. we observe the outcome of unit $i$ for exposure $e_{i,k}$. By definition of the conflict graph, this implies that we do not observe the outcomes of any of the exposures in $U_{i,k}$. Since $H_0,H_1$ only differ in the outcomes of exposures in $U_{i,k}$, they are impossible to distinguish under that event. 
In other words, for any $z \in A_{i,k}$, we have $Q_{H_0}(\cdot|z) = Q_{H_1}(\cdot|z)$ and hence
\[
d_{\mathrm{TV}}(Q_{H_0}(\cdot|z), Q_{H_1}(\cdot|z)) = 0 \enspace,
\]
which implies that
\begin{align*}
    d_{\mathrm{TV}}(Q_{H_0}^{(\design)}, Q_{H_1}^{(\design)}) &= \sum_{z \in \mathrm{supp}(\design)} \design(z) \cdot d_{\mathrm{TV}}(Q_{H_0}(\cdot|z), Q_{H_1}(\cdot|z)) \\
    &= \sum_{z \notin A_{i,k}} \design(z) \cdot d_{\mathrm{TV}}(Q_{H_0}(\cdot|z), Q_{H_1}(\cdot|z)) \\
    &\leq 1 - \design(A_{i,k}) \leq 1 -  c_{i,k} \enspace.
\end{align*}
Following the proof of Proposition~\ref{prop:lecam} for this particular design $\design$, we have
\begin{align}
    \inf_{\eate} \mathrm{Risk}(\design,\eate) &\geq \frac{d^2(H_0,H_1)}{8} \cdot \paren[\big]{1 - d_{\mathrm{TV}}(Q_{H_0}^{(\design)}, Q_{H_1}^{(\design)})} \nonumber \\
    &\geq \frac{1}{2} \cdot \paren[\Bigg]{\frac{r_{i,k}}{n}}^{2 - \frac{2}{q}} \cdot c_{i,k} \enspace. \label{eq:case1_general}
\end{align}

\paragraph{Case 2:} For all $i \in [n], k \in \setb{0,1}$ with $e_{i,k} \in T$ we have $\design(A_{i,k}) <  c_{i,k}$. In that case, define the event $A_T$ as
\[
A_T = \setb[\Big]{z \in \setb{0,1}^n: \sum_{e_{i,k} \in T} \indicator{A_{i,k}}(z) \leq 2 \cdot \sum_{e_{i,k} \in T} c_{i,k}} = 
\setb[\Big]{z \in \setb{0,1}^n: |R_z \cap T| \leq 2 \cdot \sum_{e_{i,k} \in T} c_{i,k}} \enspace.
\]
By Markov's inequality, we have
\begin{align*}
\design(A_T^c) &= \design\paren[\Bigg]{\sum_{e_{i,k} \in T} \indicator{A_{i,k}} > 2 \cdot \sum_{e_{i,k} \in T} c_{i,k}}\\
&\leq \frac{\E[\Big]{\sum_{e_{i,k} \in T} \indicator{A_{i,k}}}}{2 \cdot \sum_{e_{i,k} \in T} c_{i,k}} \\
&= \frac{\sum_{e_{i,k} \in T} \design(A_{i,k})}{2 \cdot \sum_{e_{i,k} \in T} c_{i,k}} \\
&< \frac{1}{2} \enspace.
\end{align*}

Thus, $\design(A_T) \geq 1/2$. Now, define $H_0 = H(1/2, T)$ and $H_1 = H(1/2 + \delta, T)$ where 
\[\
\delta = \frac{0.026}{ \sqrt{2  \cdot \sum_{e_{i,k} \in T} c_{i,k}}} \enspace,
\]
where $0.026$ is the constant from Lemma~\ref{lem:tv_bound_general}. Note that for this construction to be well-defined, we need $\delta < 1/2$, which is equivalent to
\begin{equation}\label{eq:c_constraint}
\sum_{e_{i,k} \in T} c_{i,k} > 2 \cdot 0.026^2 \approx 1.352 \cdot 10^{-3} \enspace.
\end{equation}
Thus, we choose the constant $1.36 \cdot 10^{-3}$ in the statement of the lemma to ensure the above inequality is satisfied.

By Lemma~\ref{lem:supp_Hp}
\[
d(H_0,H_1) = \delta \cdot \paren[\Bigg]{\frac{|T|}{n}}^{1 - \frac{1}{q}} = \frac{0.026}{\sqrt{2\sum_{e_{i,k} \in T} c_{i,k}}} \cdot \paren[\Bigg]{\frac{|T|}{n}}^{1 - \frac{1}{q}}\enspace.
\]
For any $z \in A_T$, we have 
\[
|R_z \cap T| \leq 2 \cdot \sum_{e_{i,k} \in T} c_{i,k} \leq \frac{|T|}{2}\enspace.
\]
Therefore, applying Lemma~\ref{lem:tv_bound_general}, we have that for any $z \in A_T$
\[
d_{\mathrm{TV}}(Q_{H_0}(\cdot|z), Q_{H_1}(\cdot|z)) \leq \frac{7}{8} \enspace.
\]
By the law of total probability, we have
\begin{align*}
1 - d_{\mathrm{TV}}(Q_{H_0}^{(\design)}, Q_{H_1}^{(\design)}) &= \sum_{z \in \mathrm{supp}(\design)} \design(z) \cdot (1 - d_{\mathrm{TV}}(Q_{H_0}(\cdot|z), Q_{H_1}(\cdot|z))) \\
&\geq \sum_{z \in A_T} \design(z) \cdot (1 - d_{\mathrm{TV}}(Q_{H_0}(\cdot|z), Q_{H_1}(\cdot|z))) \\
&\geq \design(A_T) \cdot (1 - \frac{7}{8}) \\
&\geq \frac{1}{16} \enspace.
\end{align*}

Therefore, following the proof of Proposition~\ref{prop:lecam} for this particular design $\design$, we have
\begin{align}
    \inf_{\eate} \mathrm{Risk}(\design,\eate) &\geq \frac{d^2(H_0,H_1)}{8} \cdot \paren[\big]{1 - d_{\mathrm{TV}}(Q_{H_0}^{(\design)}, Q_{H_1}^{(\design)})} \nonumber \\
    &\geq \frac{0.026^2 }{2^8 \cdot \sum_{e_{i,k} \in T} c_{i,k}} \cdot  \paren[\Bigg]{\frac{|T|}{n}}^{2 - \frac{2}{q}} \\
    &\geq \frac{2.64 \cdot 10^{-6}}{\sum_{e_{i,k} \in T} c_{i,k} } \paren[\Bigg]{\frac{|T|}{n}}^{2 - \frac{2}{q}}\enspace. \label{eq:case2_general}
\end{align}

Since any possible design $\design$ falls into either Case 1 or Case 2, we can combine the lower bounds from both cases (inequalities~\ref{eq:case1_general} and~\ref{eq:case2_general}) to obtain a lower bound on the minimax risk as follows
\begin{equation}
\cR(G,\tau,q) \geq \riskconst \cdot \min\paren[\Bigg]{\min_{e_{i,k} \in T} c_{i,k}\paren[\Big]{\frac{r_{i,k}}{n}}^{2 - \frac{2}{q}}, \frac{1}{\sum_{e_{i,k} \in T} c_{i,k}}\paren[\Bigg]{\frac{|T|}{n}}^{2 - \frac{2}{q}}} \enspace\label{eq:objective}
\end{equation}
\end{proof}

We are now ready to prove Theorem~\ref{thm:local_lower_bound_general}. The general strategy will be to invoke Lemma~\ref{lem:lower_bound_tunable} with the appropriate parameters. As we shall see, this choice depends crucially on the properties of each subset $T$ of exposures.

\begin{proof}[Proof of Theorem~\ref{thm:local_lower_bound_general}]
We distinguish between various cases based on the properties of the subset $T$.

\paragraph{Case 1:} Suppose $d_{\mathrm{min}}(T) < |T|$. In that case, it suffices to prove that
\begin{equation}\label{eq:case1_goal}
\cR(G,\tau,q) \geq \frac{\riskconst}{n^{2 - \frac{2}{q}}} \cdot \min\paren[\Bigg]{\frac{d_{\mathrm{min}}(T)^{2-\frac{2}{q}}}{4}, \frac{|T|^{1-\frac{1}{q}}}{\sqrt{\sum_{e_{i,k} \in T} \frac{1}{d_{i,k}^{2 - \frac{2}{q}}}}}} \enspace.
\end{equation}
Suppose $e_{\mathrm{min}} \in T$ is an exposure of miminum degree $d_{\mathrm{min}}(T)$ in $T$. 
Set $r_{i,k} = d_{i,k}$ for all $e_{i,k} \in T$.
Let 
\begin{gather*}
t^* = \frac{1}{n^{2 - \frac{2}{q}}} \cdot \min\paren[\Bigg]{\frac{\min_{e_{i,k} \in T} r_{i,k}^{2-\frac{2}{q}}}{4}, \frac{|T|^{1-\frac{1}{q}}}{\sqrt{\sum_{e_{i,k} \in T} \frac{1}{r_{i,k}^{2 - \frac{2}{q}}}}}}\\
c_{i,k} = t^* \paren[\Big]{\frac{n}{r_{i,k}}}^{2 - \frac{2}{q}} 
\end{gather*}
We will invoke Lemma~\ref{lem:lower_bound_tunable} with the above choice of $r_{i,k}$ and $c_{i,k}$. Let us first check that the $c_{i,k}$ satisfy the required assumptions of Lemma~\ref{lem:lower_bound_tunable}. We have
\[
c_{i,k} \leq \frac{1}{n^{2 - \frac{2}{q}}} \cdot \paren[\Big]{\frac{n}{r_{i,k}}}^{2 - \frac{2}{q}} \cdot \frac{\min_{e_{i,k} \in T}  r_{i,k}^{2-\frac{2}{q}}}{4} \leq 1 \enspace.
\]
Furthermore, we have
\begin{align*}
\sum_{e_{i,k} \in T} c_{i,k} &= t^* \cdot n^{2 - \frac{2}{q}} \cdot \sum_{e_{i,k} \in T} \frac{1}{r_{i,k}^{2 - \frac{2}{q}}} \\
&= \min\paren[\Bigg]{\frac{\min_{e_{i,k} \in T} r_{i,k}^{2-\frac{2}{q}}}{4}, \frac{|T|^{1-\frac{1}{q}}}{\sqrt{\sum_{e_{i,k} \in T} \frac{1}{r_{i,k}^{2 - \frac{2}{q}}}}}} \cdot \sum_{e_{i,k} \in T} \frac{1}{r_{i,k}^{2 - \frac{2}{q}}} \\
&= \min\paren[\Bigg]{\frac{\min_{e_{i,k} \in T} r_{i,k}^{2-\frac{2}{q}}}{4} \sum_{e_{i,k} \in T} \frac{1}{r_{i,k}^{2 - \frac{2}{q}}} , |T|^{1-\frac{1}{q}} \cdot \sqrt{\sum_{e_{i,k} \in T} \frac{1}{r_{i,k}^{2 - \frac{2}{q}}}}} 
\end{align*}
From here, we can lower bound the quantity as follows.
\[
\sum_{e_{i,k} \in T} c_{i,k} \geq \frac{1}{4} \geq \lowerboundconst \enspace,
\]
where we lower bounded the sum by the term corresponding to the exposure $e_{\mathrm{min}}$.

We can also upper bound the sum as follows.

\begin{align*}
    \sum_{e_{i,k} \in T} c_{i,k}  &\leq \frac{\min_{e_{i,k} \in T} r_{i,k}^{2-\frac{2}{q}}}{4} \sum_{e_{i,k} \in T} \frac{1}{r_{i,k}^{2 - \frac{2}{q}}} \\
    &\leq \frac{1}{4} \cdot |T| \enspace.
\end{align*}

Therefore, all assumptions of Lemma~\ref{lem:lower_bound_tunable} are satisfied. Invoking the lemma, we get
\begin{align*}
\cR(G,\tau,q) &\geq \frac{\riskconst}{n^{2 - \frac{2}{q}}} \cdot \min\paren[\Bigg]{\min_{e_{i,k} \in T} c_{i,k} r_{i,k}^{2-\frac{2}{q}} , \frac{|T|^{2 - \frac{2}{q}}}{\sum_{e_{i,k} \in T} c_{i,k} }} \\
&= \frac{\riskconst}{n^{2 - \frac{2}{q}}} \cdot \min\paren[\Bigg]{t^* n^{2 - \frac{2}{q}}, \frac{|T|^{2 - \frac{2}{q}}}{t^* n^{2 - \frac{2}{q}} \cdot \sum_{e_{i,k} \in T} \frac{1}{r_{i,k}^{2 - \frac{2}{q}}}}} \\
&\geq \riskconst \cdot \min\paren[\Bigg]{ t^* , \frac{|T|^{1 - \frac{1}{q}}}{ n^{2 - \frac{2}{q}} \cdot \sqrt{\sum_{e_{i,k} \in T} \frac{1}{r_{i,k}^{2 - \frac{2}{q}}}}}}\\
&\geq \riskconst \cdot t^* \enspace,
\end{align*}
where we used the definition of $t^*$ in the last two steps.
Thus, our claim is complete in that case.

\paragraph{Case 2:} Suppose $d_{\mathrm{min}}(T) \geq |T|$. In that case, it suffices to prove that

\begin{equation}\label{eq:case2_goal}
\cR(G,\tau,q) \geq  \frac{\localconst}{n^{2 - \frac{2}{q}}}  \min\paren[\Bigg]{|T|^{2 - \frac{2}{q}},\frac{|T|^{1-\frac{1}{q}}}{\sqrt{\sum_{e_{i,k} \in T} \frac{1}{d_{i,k}^{2 - \frac{2}{q}}}}}}\enspace.
\end{equation}

We distinguish between two sub-cases.

\begin{itemize}
    \item \textbf{Case 2.1:} Suppose
    \[
     \frac{|T|^{1-\frac{1}{q}}}{\sqrt{\sum_{e_{i,k} \in T} \frac{1}{d_{i,k}^{2 - \frac{2}{q}}}}} \leq |T|^{2 - \frac{2}{q}} \enspace,
    \]
    which can equivalently written as
    \begin{equation}\label{eq:case2_subcase1}
        |T|^{1-\frac{1}{q}} \sqrt{\sum_{e_{i,k} \in T} \frac{1}{d_{i,k}^{2 - \frac{2}{q}}}} \geq 1 \enspace.
    \end{equation}
    In that case, set $r_{i,k} = d_{i,k}$ for all $e_{i,k} \in T$ and define
\begin{gather*}
t^* = \frac{1}{n^{2 - \frac{2}{q}}}  \min\paren[\Bigg]{\frac{|T|^{2 - \frac{2}{q}}}{4},\frac{|T|^{1-\frac{1}{q}}}{\sqrt{\sum_{e_{i,k} \in T} \frac{1}{d_{i,k}^{2 - \frac{2}{q}}}}}}\\
c_{i,k} = t^* \paren[\Big]{\frac{n}{d_{i,k}}}^{2 - \frac{2}{q}}
\end{gather*}

We will show that the $c_{i,k}$ satisfy the assumptions of Lemma~\ref{lem:lower_bound_tunable}. We have
\[
c_{i,k} \leq \frac{1}{n^{2 - \frac{2}{q}}}  \paren[\Big]{\frac{n}{d_{i,k}}}^{2 - \frac{2}{q}} \cdot \frac{|T|^{2 - \frac{2}{q}}}{4} \leq  \frac{d_{\mathrm{min}}(T)^{2 - \frac{2}{q}}}{4d_{i,k}^{2 - \frac{2}{q}}} \leq 1 \enspace,
\]
where we have used the fact that $|T| \leq d_{\mathrm{min}}(T)$.
Furthermore, 
\begin{align*}
\sum_{e_{i,k} \in T} c_{i,k} &= t^* \cdot n^{2 - \frac{2}{q}} \cdot \sum_{e_{i,k} \in T} \frac{1}{d_{i,k}^{2 - \frac{2}{q}}} \\
&= \min\paren[\Bigg]{\frac{|T|^{2 - \frac{2}{q}}}{4} \cdot \sum_{e_{i,k} \in T} \frac{1}{d_{i,k}^{2 - \frac{2}{q}}}, |T|^{1-\frac{1}{q}} \cdot \sqrt{\sum_{e_{i,k} \in T} \frac{1}{d_{i,k}^{2 - \frac{2}{q}}}}}\enspace.
\end{align*}
We can lower bound the sum as
\[
\sum_{e_{i,k} \in T} c_{i,k} \geq \frac{1}{4} \geq \lowerboundconst \enspace,
\]
where in the last step we used \eqref{eq:case2_subcase1}.
We can also upper bound the sum as follows.
\begin{align*}
    \sum_{e_{i,k} \in T} c_{i,k} &\leq \frac{|T|^{2 - \frac{2}{q}}}{4} \sum_{e_{i,k} \in T} \frac{1}{d_{i,k}^{2 - \frac{2}{q}}}\\
      &\leq \frac{d_{\mathrm{min}}(T)^{2-\frac{2}{q}}}{4} \sum_{e_{i,k} \in T} \frac{1}{d_{i,k}^{2 - \frac{2}{q}}} \\
    &\leq \frac{1}{4} \cdot |T| \enspace,
\end{align*}
where we used the fact that $|T| \leq d_{\mathrm{min}}(T)$ in the last step.
Thus, all assumptions of Lemma~\ref{lem:lower_bound_tunable} are satisfied. Invoking the lemma, we get
\begin{align*}
\cR(G,\tau,q) &\geq \frac{\riskconst}{n^{2 - \frac{2}{q}}} \cdot \min\paren[\Bigg]{\min_{e_{i,k} \in T} c_{i,k} d_{i,k}^{2-\frac{2}{q}} , \frac{|T|^{2 - \frac{2}{q}}}{\sum_{e_{i,k} \in T} c_{i,k} }} \\
&\geq \frac{\riskconst}{n^{2 - \frac{2}{q}}} \cdot \min\paren[\Bigg]{ t^* n^{2 - \frac{2}{q}}, \frac{|T|^{2 - \frac{2}{q}}}{t^* n^{2 - \frac{2}{q}} \cdot \sum_{e_{i,k} \in T} \frac{1}{d_{i,k}^{2 - \frac{2}{q}}}}} \\
\intertext{using the fact that $d_{\mathrm{min}}(T) \geq |T|$}\\
&\geq \riskconst \cdot \min\paren[\Bigg]{ t^* , \frac{|T|^{1 - \frac{1}{q}}}{ n^{2 - \frac{2}{q}} \cdot \sqrt{\sum_{e_{i,k} \in T} \frac{1}{d_{i,k}^{2 - \frac{2}{q}}}}}}\\
&\geq \riskconst \cdot t^* \enspace,
\end{align*}
where we used the definition of $t^*$ in the last two steps. Inequality~\eqref{eq:case2_goal} now follows.

\item \textbf{Case 2.2:} Suppose
\[
    \frac{|T|^{1-\frac{1}{q}}}{\sqrt{\sum_{e_{i,k} \in T} \frac{1}{d_{i,k}^{2 - \frac{2}{q}}}}} > |T|^{2 - \frac{2}{q}} \enspace,
    \]
which can equivalently be written as
\begin{equation}\label{eq:case2_subcase2}
    |T|^{1-\frac{1}{q}} \sqrt{\sum_{e_{i,k} \in T} \frac{1}{d_{i,k}^{2 - \frac{2}{q}}}} < 1 \enspace.
\end{equation}
We will show that we can pick $r_{i,k} \leq d_{i,k}$ with $r_{i,k} \geq |T|$ for all $e_{i,k} \in T$ such that

\begin{equation}\label{eq:case2_subcase2_subgoal}
    1 \leq |T|^{2-\frac{2}{q}} \sum_{e_{i,k} \in T} \frac{1}{r_{i,k}^{2 - \frac{2}{q}}} \leq 2
\end{equation}

Indeed, we can achieve this by iteratively picking $e_{i,k} \in T$ with $r_{i,k} > |T|$ and setting $r_{i,k} = r_{i,k} - 1$ until the left-hand side of \eqref{eq:case2_subcase2_subgoal} is satisfied. We first prove that this process will terminate in a finite number of steps and $r_{i,k} \geq |T|$ at termination.
Indeed, suppose we reach the point where $r_{i,k} = |T|$ for all $e_{i,k} \in T$. Then, we have
\[
|T|^{2-\frac{2}{q}} \sum_{e_{i,k} \in T} \frac{1}{r_{i,k}^{2 - \frac{2}{q}}} = |T|^{2-\frac{2}{q}} \sum_{e_{i,k} \in T} \frac{1}{|T|^{2 - \frac{2}{q}}} = |T| \geq 1 \enspace.
\]
Therefore, the process will terminate at some point and at termination we have $r_{i,k} \geq |T|$ for all $e_{i,k} \in T$. Furthermore, consider a single update step where we pick $e_{i,k} \in T$ and set $r_{i,k} = r_{i,k} - 1$. The value of the objective changes by at most
\[
|T|^{2-\frac{2}{q}} \cdot \paren[\Bigg]{\frac{1}{(r_{i,k}-1)^{2 - \frac{2}{q}}} - \frac{1}{r_{i,k}^{2 - \frac{2}{q}}}} 
\leq 1
\]
since $r_{i,k} > |T|$ by definition of the process. Thus, the objective increases by at most $1$ at every update. Therefore, the first time the objective exceeds $1$, \eqref{eq:case2_subcase2_subgoal} will be satisfied.
From now on, let $r_{i,k}$ be the values obtained at termination of the above process. We set
\begin{gather*}
t^* = \frac{1}{n^{2 - \frac{2}{q}}}  \min\paren[\Bigg]{\frac{\min_{e_{i,k} \in T} r_{i,k}^{2 - \frac{2}{q}}}{4},\frac{|T|^{1-\frac{1}{q}}}{\sqrt{\sum_{e_{i,k} \in T} \frac{1}{r_{i,k}^{2 - \frac{2}{q}}}}}}\\
c_{i,k} = t^* \paren[\Big]{\frac{n}{r_{i,k}}}^{2 - \frac{2}{q}}
\end{gather*}
We will show that the $c_{i,k}$ satisfy the assumptions of Lemma~\ref{lem:lower_bound_tunable}. We have
\[
c_{i,k} \leq \frac{1}{n^{2 - \frac{2}{q}}}  \paren[\Big]{\frac{n}{r_{i,k}}}^{2 - \frac{2}{q}} \cdot \frac{\min_{e_{i,k} \in T} r_{i,k}^{2 - \frac{2}{q}}}{4} \leq 1 \enspace.
\]
Furthermore,
\begin{align*}
\sum_{e_{i,k} \in T} c_{i,k} &= t^* \cdot n^{2 - \frac{2}{q}} \cdot \sum_{e_{i,k} \in T} \frac{1}{r_{i,k}^{2 - \frac{2}{q}}} \\
&= \min\paren[\Bigg]{\frac{\min_{e_{i,k} \in T} r_{i,k}^{2 - \frac{2}{q}}}{4} \cdot \sum_{e_{i,k} \in T} \frac{1}{r_{i,k}^{2 - \frac{2}{q}}}, |T|^{1-\frac{1}{q}} \cdot \sqrt{\sum_{e_{i,k} \in T} \frac{1}{r_{i,k}^{2 - \frac{2}{q}}}}} \\
&\geq \min\paren[\Bigg]{|T|^{2 - \frac{2}{q}} \cdot \sum_{e_{i,k} \in T} \frac{1}{r_{i,k}^{2 - \frac{2}{q}}}, |T|^{1-\frac{1}{q}} \cdot \sqrt{\sum_{e_{i,k} \in T} \frac{1}{r_{i,k}^{2 - \frac{2}{q}}}}}\\
&\geq \lowerboundconst \enspace,
\end{align*}
where in the last step we used \eqref{eq:case2_subcase2_subgoal}.
We can also upper bound the sum as follows.
\begin{align*}
    \sum_{e_{i,k} \in T} c_{i,k} &\leq \frac{\min_{e_{i,k} \in T} r_{i,k}^{2 - \frac{2}{q}}}{4} \sum_{e_{i,k} \in T} \frac{1}{r_{i,k}^{2 - \frac{2}{q}}}\\
    &\leq \frac{1}{4} \cdot |T| \enspace,
\end{align*}

Thus, all assumptions of Lemma~\ref{lem:lower_bound_tunable} are satisfied. Invoking the lemma, we get
\begin{align}
\cR(G,\tau,q) &\geq \frac{\riskconst}{n^{2 - \frac{2}{q}}} \cdot \min\paren[\Bigg]{\min_{e_{i,k} \in T} c_{i,k} r_{i,k}^{2-\frac{2}{q}} , \frac{|T|^{2 - \frac{2}{q}}}{\sum_{e_{i,k} \in T} c_{i,k} }}\nonumber \\
&\geq \riskconst \cdot \min\paren[\Bigg]{ t^* , \frac{|T|^{1 - \frac{1}{q}}}{ n^{2 - \frac{2}{q}} \cdot \sqrt{\sum_{e_{i,k} \in T} \frac{1}{r_{i,k}^{2 - \frac{2}{q}}}}}}\nonumber\\
&\geq \riskconst \cdot t^* \enspace. \label{eq:risk_final}
\end{align}
Now all we have to do is connect $t^*$ to the desired lower bound of \eqref{eq:case2_goal}. This is because the lower bound is in terms of the original degrees, whereas in $t^*$ some degrees have been ``truncated''. We have
\begin{align*}
    t^* &= \frac{1}{n^{2 - \frac{2}{q}}}  \min\paren[\Bigg]{\frac{\min_{e_{i,k} \in T} r_{i,k}^{2 - \frac{2}{q}}}{4},\frac{|T|^{1-\frac{1}{q}}}{\sqrt{\sum_{e_{i,k} \in T} \frac{1}{r_{i,k}^{2 - \frac{2}{q}}}}}}\\
    &\geq \frac{1}{n^{2 - \frac{2}{q}}}  \min\paren[\Bigg]{\frac{|T|^{2 - \frac{2}{q}}}{4},\frac{|T|^{1-\frac{1}{q}}}{\sqrt{\sum_{e_{i,k} \in T} \frac{1}{r_{i,k}^{2 - \frac{2}{q}}}}}} \\
    &\geq \frac{1}{n^{2 - \frac{2}{q}}}  \min\paren[\Bigg]{\frac{|T|^{2 - \frac{2}{q}}}{4},\frac{|T|^{2-\frac{2}{q}}}{\sqrt{2}}}\\
    \intertext{where we use \eqref{eq:case2_subcase2_subgoal}}\\
    &= \frac{1}{4}\frac{1}{n^{2 - \frac{2}{q}}}  |T|^{2 - \frac{2}{q}} \\
    &\geq \frac{1}{4}\frac{1}{n^{2 - \frac{2}{q}}} \min\paren[\Bigg]{|T|^{2 - \frac{2}{q}},\frac{|T|^{1-\frac{1}{q}}}{\sqrt{\sum_{e_{i,k} \in T} \frac{1}{d_{i,k}^{2 - \frac{2}{q}}}}}} \enspace
\end{align*}
Plugging this back to \eqref{eq:risk_final}, we obtain
\[
\cR(G,\tau,q) \geq \frac{\riskconst}{4}\cdot \frac{1}{n^{2 - \frac{2}{q}}} \min\paren[\Bigg]{|T|^{2 - \frac{2}{q}},\frac{|T|^{1-\frac{1}{q}}}{\sqrt{\sum_{e_{i,k} \in T} \frac{1}{d_{i,k}^{2 - \frac{2}{q}}}}}} \enspace,
\]
which establishes the desired lower bound in that case.
\end{itemize}

We have established the desired lower bound in all cases, therefore the proof is complete.
\end{proof}

We now give the proof of Corollary~\ref{cor:local_lower_bound_general}, which is restated here for completeness. 
Recall that in Section~\ref{sec:local_lower_bound} we defined $F_\cH$ to be the degree distribution function of the conflict graph $\cH$.
We also defined the \emph{critical degree} of a conflict graph $\cH$ with respect to a moment $q \geq 2$ as follows
\[
d^*_q(\cH) = \sup\setb[\Bigg]{d \in [n] : 1 - F_\cH(d) \geq \frac{d^{\frac{2q - 2}{3q - 2}}}{n}} \enspace. 
\]

\localLowerBoundCor*

\begin{proof}
    By definition of $d^*_q(\cH)$, there exist a subset of exposures $T \subseteq V_\cH$ with $|T| \geq (d^*_q(\cH))^{\frac{2(q-1)}{3q - 2}}$ such that $d_{i,k} \geq d^*_q(\cH)$ for all $e_{i,k} \in T$. Therefore, we can invoke Theorem~\ref{thm:local_lower_bound_general} with the above choice of $T$ to get
\begin{align*}
\cR(G,\tau,q) &\geq \frac{\localconst}{n^{2 - \frac{2}{q}}} \cdot \min\paren[\Bigg]{|T|^{2 - \frac{2}{q}},\frac{|T|^{1-\frac{1}{q}}}{\sqrt{\sum_{e_{i,k} \in T} \frac{1}{d_{i,k}^{2 - \frac{2}{q}}}}}, d_{\mathrm{min}}(T)^{2 - \frac{2}{q}}} \\
&\geq \frac{\localconst}{n^{2 - \frac{2}{q}}} \cdot \min\paren[\Bigg]{|T|^{2 - \frac{2}{q}},\frac{|T|^{1-\frac{1}{q}}}{\sqrt{\frac{|T|}{(d^*_q(\cH))^{2 - \frac{2}{q}}}}}, (d^*_q(\cH))^{2 - \frac{2}{q}}} \\
&= \frac{\localconst}{n^{2 - \frac{2}{q}}} \cdot \min\paren[\Bigg]{|T|^{2 - \frac{2}{q}},|T|^{\frac{1}{2} - \frac{1}{q}}\cdot (d^*_q(\cH))^{1 - \frac{1}{q}}, (d^*_q(\cH))^{2 - \frac{2}{q}}} \\
&\geq \frac{\localconst}{n^{2 - \frac{2}{q}}} \cdot \min\paren[\Bigg]{(d^*_q(\cH))^{\frac{2q-2}{3q-2}\paren{2 - \frac{2}{q}}},(d^*_q(\cH))^{\frac{2q-2}{3q-2}\paren{\frac{1}{2} - \frac{1}{q}} + 1 - \frac{1}{q}}, (d^*_q(\cH))^{2 - \frac{2}{q}}} \\
&= \localconst  \cdot \frac{(d^*_q(\cH))^{\frac{4(q-1)^2}{q(3q - 2)}}}{n^{2 - \frac{2}{q}}} \enspace.
\end{align*}
\end{proof}

Finally, we present the proof of Proposition~\ref{prop:critical_to_avg} that lower bounds the critical degree by the average degree of the conflict graph $\cH$.

\begin{restatable}{proposition}{avgDegree}\label{prop:critical_to_avg}
    Let $G$ be a graph and $\tau$ a contrastive effect, with $\cH$ being the associated conflict graph. We denote by $d_{\mathrm{avg}}(\cH)$ the average degree of $\cH$. Then, for any $q \geq 2$, we have
\[
d^*_q(\cH) \geq \frac{d_{\mathrm{avg}}(\cH)}{4} \enspace.
\]
\end{restatable}
\begin{proof}
    Let $A,B \subseteq V_\cH$ be the subsets of exposures that have degree at least $d^*_q(\cH) + 1$ and at most $d^*_q(\cH)$, respectively, i.e.
    \begin{align*}
    A = \setb[\Big]{e_{i,k} \in V_\cH : d_{i,k} \geq d^*_q(\cH) + 1}\\
    B = \setb[\Big]{e_{i,k} \in V_\cH : d_{i,k} \leq d^*_q(\cH)} \enspace.
    \end{align*}
    Clearly $A$ and $B$ partition $V_\cH$. By definition of the critical degree, we have
    \[
    |A| \leq (d^*_q(\cH) + 1)^{\frac{2q - 2}{3q - 2}} \enspace.
    \]
    Thus, we have that
    \begin{align*}
    d_{\mathrm{avg}}(\cH) &= \frac{\sum_{e_{i,k} \in V_\cH} d_{i,k}}{|V_\cH|}  \\
    &= \frac{\sum_{e_{i,k} \in A} d_{i,k} + \sum_{e_{i,k} \in B} d_{i,k}}{|V_\cH|} \\
    &\leq \frac{|A|n + |B|d^*_q(\cH)}{|V_\cH|} \\
    &\leq \frac{2n(d^*_q(\cH) + 1)^{\frac{2q - 2}{3q - 2}} + 2nd^*_q(\cH)}{|V_\cH|} \\
    &= (d^*_q(\cH) + 1)^{\frac{2q - 2}{3q - 2}}+d^*_q(\cH)  \\
    &\leq 2^{\frac{2q-2}{3q - 2}}\cdot (d^*_q(\cH))^{\frac{2q - 2}{3q - 2}} + d^*_q(\cH) \\
    &\leq 4d^*_q(\cH) \enspace,
    \end{align*}
    using the fact that $2^{\frac{2q-2}{3q - 2}} \leq 3$ for all $q \geq 2$. Rearranging the above inequality, we get the conclusion.
\end{proof}

	\section{Upper Bounds on the Minimax Risk}\label{sec:upper_bounds_appendix}

In this section, we present the proofs of our general upper bounds on the minimax risk. We start by presenting the analysis of the generalized Conflict Graph Design in Section~\ref{sec:cgd_appendix} and then continue to the optimal bias-variance tradeoff in Section~\ref{sec:bias_variance_appendix}.

\subsection{Analysis of the Conflict Graph Design (Theorem~\ref{thm:conflict_graph_design})}\label{sec:cgd_appendix}

In this section, we present the proof of Theorem~\ref{thm:conflict_graph_design}, which we restate here for convenience.

\cgd*

For $q=2$, this upper bound was established in \citet{kandiros2024conflict}. In particular, they introduced the Conflict Graph Design (CGD) and provided an unbiased estimator which achieves a worst case variance of $\bigO{\lambda(G)/n}$. 
We will follow the general strategy of the Conflict Graph Design, although as we will see there are important obstacles to overcome for general moments.

We now provide a roadmap for this section. 

\begin{itemize}
\item In Section~\ref{sec:ordering}, we introduce the importance ordering that will be used for our design and establish a crucial property. In Section~\ref{sec:cgd_general_appendix} we describe the design for arbitrary moments $q$. 
\item In Section~\ref{sec:estimator} we introduce the estimator and prove some basic properties of it. 
\item In Section~\ref{sec:probability_bounds} we bound the covariance terms that arise in the variance of the estimator. \item In Section~\ref{sec:matrix_bound} we bound suitable norms of matrices that are related to the covariance matrix of the design. 
\item In Section~\ref{sec:proof_cgd} we present the proof of Theorem~\ref{thm:conflict_graph_design} using all the above ingredients.
\end{itemize}

Throughout this section, we will denote by $G$ the interference graph, $\ate$ the effect of interest, $\cH$ the conflict graph associated with $G$ and $\ate$. For a node $e_{i,k} \in V(\cH)$, we will denote by
$\mathcal{N}(e_{i,k})$ the set of neighbors of $e_{i,k}$ in $\cH$. 
We denote by $d_{i,k} = |\mathcal{N}(e_{i,k})|$ the degree of node $e_{i,k}$ in the conflict graph $\cH$.
Also, for a unit $i \in V(G)$, we denote by $\tilde{N}(i)$ the set consisting of $i$ and its neighbors in $G$.

\subsubsection{Importance Ordering}\label{sec:ordering}

An \emph{importance ordering} is a bijection $\pi:V(\cH) \mapsto [2n]$ that ``ranks'' the vertices in the conflict graph $\cH$ from the most important to the least important.
For an importance ordering $\pi$ and a node $e_{i,k} \in V(\cH)$, we denote by
$\mathcal{N}^\pi_b(e_{i,k})$ the set of neighbors of $e_{i,k}$ in $\cH$ that appear before $e_{i,k}$ in the importance ordering $\pi$, i.e.
\[
\mathcal{N}^\pi_b(e_{i,k}) = \setb{e_{j,l} \in \mathcal{N}(e_{i,k}) : \pi(e_{j,l}) < \pi(e_{i,k})} \enspace.
\]

In \citet{kandiros2024conflict}, they proposed an importance ordering based on the maximum eigenvector in the graph that ensures every node has a bounded number of more important neighbors. This property was crucial for upper bounding the variance of their design.

We will show that there is a different ordering that behaves more favorably for general moments $q$. 
We first provide pseudocode for the procedure that produces the ordering in Algorithm~\ref{alg:ordering}. 

\begin{algorithm}[ht]
\KwIn{Conflict graph $\cH=(V_\cH, E_\cH)$ and moment $q$.}
\KwOut{Mapping $\pi: V_\cH \to [2n]$ representing the importance ordering.}
    $S \gets V_\cH$ \\
    \For{$e_{i,k} \in V_\cH$}{
    Set $
    p_{i,k} \gets \frac{1}{2 \lamH^{\frac{2}{q}} d_{i,k}^{1-\frac{2}{q}}} 
    $
    }
    \For{$c = 1$ to $2n$}{
        \For{$e_{i,k} \in S$}{
            $w_{i,k} \gets \sum_{e_{j,l} \in \mathcal{N}(e_{i,k}) \cap S} p_{j,l}$ 
        }
        $e_{i^*,k^*} \gets \argmin_{e_{i,k} \in S} w_{i,k}$ \label{algline:select_min} \\
        $\pi(e_{i^*,k^*}) = 2n - c + 1$ \label{algline:ordering_update}\\
        $S \gets S \setminus \setb{e_{i^*,k^*}}$\\
    }
\caption{Importance Ordering}
\label{alg:ordering}
\end{algorithm}

The ordering $\pi$ in Algorithm~\ref{alg:ordering} is constructed iteratively, starting from the least important node and ending with the most important node. At each step, $S$ represents the set of nodes that have not been ordered yet. We pick the node whose neighbors in $S$ have the smallest sum of $p_{i,k}$ values and assign it the next available position in the ordering.

In the case $q=2$, we observe that the ordering of Algorithm~\ref{alg:ordering} is constructed by iteratively selecting the node with smallest degree in the remaining induced graph, placing them in the last available spot in the ordering and removing them from the graph. This ordering of the nodes is used in the classical proof of Wilf's Theorem \citep{wilf1967eigenvalues}, which states that the chromatic number of a graph is upper bounded by the largest eigenvalue of its adjacency matrix. In \citet{kandiros2024conflict}, they use a different ordering based on the leading eigenvector of the adjacency matrix, even though they acknowledge that the ordering based on the smallest induced degree also satisfies the desired property.

Since the value of the $p_{i,k}$ variables is different for different values of $q$, the ordering produced by Algorithm~\ref{alg:ordering} will be different for different values of $q$.
Indeed, 
the ordering for $q = \infty$ is determined by the degrees in the original graph, while for
$q = 2$ it is constructed using the degrees in the induced subgraph that is iteratively produced. Thus,
the ordering could be very different for different values of $q$.

In the next lemma, we establish the crucial property of the ordering produced by Algorithm~\ref{alg:ordering} that will be used in the analysis of our design later.


\begin{lemma}\label{lem:ordering}
Let $\cH$ be a conflict graph with maximum eigenvalue $\lamH$.
Let $q \geq 2$ be an integer.
For every node $e_{i,k} \in V(\cH)$ with degree $d_{i,k}$ in the conflict graph $\cH$, let 
\[
p_{i,k} = \frac{1}{2 \lamH^{\frac{2}{q}} d_{i,k}^{1-\frac{2}{q}}} \enspace.
\]
Suppose $\pi: V_\cH \to [2n]$ is the ordering of the vertices in $\cH$ that is produced by Algorithm~\ref{alg:ordering} with input $\cH$ and $q$. Then, for every node $e_{i,k} \in V_\cH$
\[
\sum_{e_{j,l} \in \mathcal{N}_b^\pi(e_{i,k})} p_{j,l} \leq 1/2 \enspace.
\]
\end{lemma}
\begin{proof}
    For an arbitrary subset $S \subseteq [2n]$, we denote by $d_{i,k}(S)$ the number of neighbors of $e_{i,k}$ in $S$, i.e., $d_{i,k}(S) = |\mathcal{N}(e_{i,k}) \cap S|$. Also, for every node $e_{i,k}$, we define the quantity
\[
w_{i,k}(S) = \sum_{e_{j,l} \in \mathcal{N}(e_{i,k}) \cap S} p_{j,l}
\]

We will show that for every subset $S$, there exists a node $e_{i,k} \in S$ with $w_{i,k}(S) \leq 1/2$. It suffices to show that the average value of $w_{i,k}(S)$ over all $e_{i,k} \in S$ is at most $1/2$. Indeed, we have
\begin{align*}
\frac{1}{|S|} \sum_{e_{i,k} \in S} w_{i,k}(S) & = \frac{1}{|S|} \sum_{e_{i,k} \in S} \sum_{e_{j,l} \in \mathcal{N}(e_{i,k}) \cap S} p_{j,l} \\
& = \frac{1}{|S|} \sum_{e_{i,k} \in S} \sum_{e_{j,l} \in \mathcal{N}(e_{i,k}) \cap S} \frac{1}{2 \lambda(G)^{2/q} d_{j,l}^{1-2/q}} \\
& = \frac{1}{|S|} \sum_{e_{j,l} \in S} \sum_{e_{i,k} \in \mathcal{N}(e_{j,l}) \cap S} \frac{1}{2 \lambda(G)^{2/q} d_{j,l}^{1-2/q}} \\
\intertext{by exhanging the order of summation of $i$ and $j$}\\
& = \frac{1}{|S|} \sum_{e_{j,l} \in S} \frac{d_{j,l}(S)}{2 \lambda(G)^{2/q} d_{j,l}^{1-2/q}} \\
& \leq \frac{1}{|S|} \sum_{e_{j,l} \in S} \frac{d_{j,l}(S)}{2 \lambda(G)^{2/q} d_{j,l}(S)^{1-2/q}} \\
\intertext{using the fact that $d_{j,l}(S) \leq d_{j,l}$ for all $S$ by definition}\\
& = \frac{1}{|S|} \sum_{e_{j,l} \in S} \frac{d_{j,l}(S)^{2/q}}{2 \lambda(G)^{2/q}} \\
& = \frac{1}{2 \lambda(G)^{2/q}} \frac{1}{|S|} \sum_{e_{j,l} \in S} d_{j,l}(S)^{2/q} \\
& \leq \frac{1}{2 \lambda(G)^{2/q}} \paren[\Big]{\frac{1}{|S|} \sum_{e_{j,l} \in S} d_{j,l}(S)}^{2/q} \\
\intertext{using Jensen's inequality for the function $x \mapsto x^{2/q}$, which is concave for $q \in (2,\infty)$}\\
& \leq \frac{1}{2 \lambda(G)^{2/q}} \paren[\big]{\lambda(G)}^{2/q} \\
\intertext{using the fact that the average degree of any induced subgraph of $G$ is at most the maximum eigenvalue}\\
& = 1/2 \enspace.
\end{align*}

We now use this property to analyze Algorithm~\ref{alg:ordering}. 
Let $e_{i,k} \in V_\cH$ be an arbitrary node. We know that $e_{i,k}$ will be selected exactly once as $e_{j^*,l^*}$ in line~\ref{algline:select_min}, since this line is executed exactly $|V_\cH|$ times and every time a different node is selected. 
Let $c$ be the iteration in which $e_{i,k}$ is selected as $e_{j^*,l^*}$. When line~\ref{algline:ordering_update} is executed in that iteration, we have $\pi(e_{i,k}) = |V_\cH| - i + 1$. Moreover, at that point the set $S$ contains exactly the nodes that are more important than $e_{i,k}$ in the ordering $\pi$, since all the less important nodes have already been selected in previous iterations and removed from $S$. Thus, we have that $\mathcal{N}^\pi_b(e_{i,k}) = \mathcal{N}(e_{i,k}) \cap S$. 
Therefore, 
\[
\sum_{e_{j,l} \in \mathcal{N}_b^\pi(e_{i,k})} p_{j,l}  = \sum_{e_{j,l} \in \mathcal{N}(e_{i,k}) \cap S} p_{j,l} = w_{i,k}(S) \enspace.
\]

By the property we established above, we have
\[
\sum_{e_{j,l} \in \mathcal{N}_b^\pi(e_{i,k})} p_{j,l}  = w_{i,k}(S) = \min_{e_{j,l} \in S} w_{j,l}(S) \leq \frac{1}{|S|} \sum_{e_{j,l} \in S} w_{j,l}(S) \leq 1/2 \enspace,
\]
as desired. Since $e_{i,k}$ was an arbitrary node, this establishes the desired property for all nodes in the graph.
\end{proof}

In the limit $q = \infty$, Lemma~\ref{lem:ordering} provides a way of selecting the ordering. However, it turns out that in this special case there is a alternative way of ordering the nodes, which is simply to order them in decreasing order of degree. We present this property as a separate lemma for completeness.

\begin{lemma}\label{lem:ordering_infinity}
    For every node $e_{i,k} \in V(\cH)$ with degree $d_{i,k}$ in the conflict graph $\cH$, let 
\[
p_{i,k} = \frac{1}{2 d_{i,k}} \enspace.
\]
Suppose $\pi:V_\cH \to [2n]$ is the ordering of the vertices in decreasing value of their degree in $\cH$ , i.e.
$d_{i,k} \leq d_{j,l}$ if and only if $\pi_{i,k} \geq \pi_{j,l}$.
Then, for every node $e_{i,k}$
\[
\sum_{e_{j,l} \in \mathcal{N}_b^\pi(e_{i,k})} p_{j,l} \leq 1/2 \enspace.
\]
\end{lemma}
\begin{proof}
Let $e_{i,k}$ be an arbitrary node in $V_\cH$. By definition of the ordering, we have that for every $e_{j,l} \in  \mathcal{N}_b^\pi(e_{i,k})$ it holds that $d_{j,l} \geq d_{i,k}$. Thus, we have
\[
\sum_{e_{j,l} \in \mathcal{N}_b^\pi(e_{i,k})} p_{j,l} = \sum_{e_{j,l} \in \mathcal{N}_b^\pi(e_{i,k})} \frac{1}{2 d_{j,l}} \leq \sum_{e_{j,l} \in \mathcal{N}_b^\pi(e_{i,k})} \frac{1}{2 d_{i,k}} \leq \frac{d_{i,k}}{2 d_{i,k}} = 1/2 \enspace,
\]
as desired.
\end{proof}

\subsubsection{Conflict Graph Design for General Moment Restrictions}\label{sec:cgd_general_appendix}

Since Theorem~\ref{thm:conflict_graph_design} aims to upper bound the minimax risk, it suffices to present a design-estimator pair whose risk is upper bounded by the desired quantity.
We start by providing the details of our design for arbitrary moments $q$, given in Algorithm~\ref{alg:CGD}.

\begin{algorithm}[ht]
\KwIn{Maximum eigenvalue $\lamH$, graph $G$, conflict graph $\cH$ and moment $q$.}
	\KwOut{Random intervention $Z \in \mathcal{Z} = \setb{0,1}^n$, variables $U_i$ for all $i \in V(\cH)$.}
    Importance ordering $\pi$ $\gets$ Algorithm~\ref{alg:ordering} with input $\cH$ and $q$ \;
    \For{$e_{i,k} \in V(\cH)$}{
    $
p_{i,k} \gets \frac{1}{2 \lamH^{\frac{2}{q}} d_{i,k}^{1-\frac{2}{q}}} 
$\;
    }
    $C \gets 2$ \;

	Sample desired exposure variables $\setb{U_{i,k} : e_{i,k} \in V(\cH)}$ independently and identically as
	$$
	U_{i,k} \gets
	\begin{cases*}
		e  \quad\text{with probability } \frac{p_{i,k}}{C }\\
		*  \quad\text{with probability } 1 - \frac{p_{i,k}}{C }
	\end{cases*}
	$$
	Initialize the intervention vector $Z \gets \vec{0}$ \;
	\For{$e_{i,k} \in V(\cH)$}{
		\If{$U_{i,k} = e $ and $U_{r,l} = *$ for all $e_{r,l} \in \mathcal{N}^\pi_b(e_{i,k})$\label{algline:exposure-condition}}{
			Update intervention vector: set 
            $
            Z_j = \indicator[\big]{j \in h_i(e_{i,k})} \quad \text{for all } j \in \widetilde{N}(i)
			$
			\label{algline:set-intervention} \;
		}
	}
\caption{Conflict Graph Design (CGD)}
\label{alg:CGD}
\end{algorithm}

The design takes as input an importance ordering $\pi$, which is simply a total ordering of the vertices in the graph $G$, the maximum eigenvalue $\lamH$ of the adjacency matrix of the conflict graph $\cH$, the graph $G$ itself and the moment $q$. Note that even though the goal is to generate the random intervention $Z \in \setb{0,1}^n$, the algorithm also outputs the variables $U_{i,k}$ for all $e_{i,k} \in V(\cH)$. This is because this additional randomness, which is correlated with the treatment assignment, will be useful for defining the estimator. Thus, this falls under the category of a \emph{randomized experimental design}, as was discussed in Section~\mainref{sec:def-minimax-risk} and defined in Section~\ref{sec:randomness_appendix}. 

The general idea of the design is that every exposure $e_{i,k}$ in the conflict graph is assigned some random variable $U_{i,k}$, which determines whether we ``desire to be observe'' this exposure. As dictated in Line~\ref{algline:exposure-condition}, if we desire to observe the outcome under exposure $e_{i,k}$ (i.e. $U_{i,k} = e$), and its more important neighbors in the conflict graph are not desired to be observed (i.e. $U_{r,l} = *$ for all $e_{r,l} \in \mathcal{N}^\pi_b(e_{i,k})$), then we set the intervention $Z$ in such a way that the exposure $e_{i,k}$ is indeed observed (Line~\ref{algline:set-intervention}). 

The success of this design relies on the fact that we get to see the outcomes of all exposures with large enough probability. 
In turn, this depends on the importance ordering $\pi$ that is chosen. Indeed, if an exposure has many more important neighbors, it is unlikely that it will ever be ``allowed'' to be observed.
We will use the importance ordering produced by Algorithm~\ref{alg:ordering}. As we will see, the property established in Lemma~\ref{lem:ordering} for this ordering will be crucial for showing that every exposure gets observed with large enough probability.

\subsubsection{Effect Estimator}\label{sec:estimator}

As we alluded to earlier, our estimator will use the residual randomness of the variables $U_{i,k}$. To that end, for every node $e_{i,k} \in V(\cH)$, define the event $E_{i,k}$ as
\[
E_{i,k} = \setb[\Big]{U_{i,k} = e \text{ and } U_{j,l} = * \text{ for all } e_{j,l} \in \mathcal{N}^\pi_b(e_{i,k})} \enspace.
\]
This is the event that we desire to observe exposure $e_{i,k}$ and all of its more important neighbors ``allow'' it to do so.
We note that similar events were used in the definition of the estimator by \citet{kandiros2024conflict}.

The first observation is that if the event $E_{i,k}$ occurs, then in the assignment $Z$ that is produced by the design, the outcome of exposure $e_{i,k}$ will be observed for unit $i$. A similar fact was established as Lemma 3.2 in \citet{kandiros2024conflict}. We give here the statement and proof for completeness.

\begin{lemma}\label{lem:desired_exposure}
    For every node $e_{i,k} \in V(\cH)$, if the event $E_{i,k}$ occurs, then if $Z \in \setb{0,1}^n$ is the assignment that is produced by the design of Algorithm~\ref{alg:CGD}, $h_i(Z) = e_{i,k}$, i.e. unit $i$ receives the exposure $e_{i,k}$.
\end{lemma}
\begin{proof}
    Suppose event $E_{i,k}$ occurs. Then, by definition, the conditions in line~\ref{algline:exposure-condition} will be satisfied for node $e_{i,k}$, so that line~\ref{algline:set-intervention} in Algorithm~\ref{alg:CGD} will be executed for node $e_{i,k}$. This means that after this execution, the intervention vector $Z$ will be such that $Z_j = \indicator{j \in h_i(e_{i,k})}$ for all $j \in \widetilde{N}(i)$. Thus, after this execution vector $Z$ satisfies $h_{i}(Z) = e_{i,k}$.

    We now argue that the values of $Z$ for $j \in \widetilde{N}(i)$ will not change for the remainder of the algorithm. To see this, note that in order for $Z$ to change, line~\ref{algline:set-intervention-bias} would have to be executed for some different node $e_{r,l} \in V(\cH)$ that comes after $e_{i,k}$ in the importance ordering.
    If the assignment of $Z$ for some $j \in \widetilde{N}(i)$ changes, this means that $j \in \widetilde{\mathcal{N}}(i) \cap \widetilde{\mathcal{N}}(r)$ and $\indicator{j \in h_i(e_{i,k})} \neq \indicator{j \in h_r(e_{r,l})}$. By the definition of the conflict graph $\cH$, this means that $e_{i,k}$ and $e_{r,l}$ are neighbors in $\cH$. In order for line~\ref{algline:set-intervention-bias} to be executed for node $e_{r,l}$, the conditions in line~\ref{algline:exposure-condition-bias} must be met. In particular, we should have $U_{i,k} = *$, since $e_{i,k}$ is a more important neighbor of $e_{r,l}$ in $\cH$. However, event $E_{i,k}$ implies that $U_{i,k} = e$, so this is a contradiction. Therefore, the values of $Z$ for $j \in \widetilde{N}(i)$ will not change, therefore at the end of Algorithm~\ref{alg:CGD} we will have $h_{i}(Z) = e_{i,k}$, as desired.
\end{proof}

Lemma~\ref{lem:desired_exposure} shows that if event $E_{i,k}$ occurs, then we get to observe the outcome for exposure $e_{i,k}$. 
Thus, a natural estimator for the effect $\ate$ would be to do inverse propensity weighting based on the events $E_{i,k}$.
The estimator that arises was referred to as the \emph{modified Horvitz-Thompson estimator} in \citet{kandiros2024conflict} and is defined as follows.

\begin{equation}\label{eq:modified_ht}
\eate = \frac{1}{n} \sum_{i=1}^n \paren[\Bigg]{  \frac{\indicator{E_{i,1}}}{\Pr{E_{i,1}}}  - \frac{\indicator{E_{i,0}}}{\Pr{E_{i,0}}} } \cdot Y_i  \enspace.
\end{equation}
Note that $\eate$ crucially depends on the residual randomness of the variables $U_{i,k}$, since the events $E_{i,k}$ are defined in terms of the $U_{i,k}$. Therefore, strictly speaking $\eate$ is a randomized estimator.

As in \citet{kandiros2024conflict}, let us first establish that $\eate$ is unbiased for $\ate$. 

\begin{lemma}
    For any $y \in \cM(G)$, the estimator $\eate$ is unbiased for $\ate$, i.e. $\E{\eate} = \ate$.
\end{lemma}
\begin{proof}
    By Lemma~\ref{lem:desired_exposure} and the fact that $y \in \cM(G)$, we have that if $E_{i,k}$ occurs, then $h_i(Z) = e_{i,k}$, so that
    $$
    y_i(Z)  = y_i(h_i(Z))= y_i(e_{i,k})\enspace.
    $$
    Thus, we have
\begin{align*}
\E{\eate} & = \frac{1}{n} \paren[\Bigg]{\sum_{i=1}^{n}  \frac{\Pr{E_{i,1}}}{\Pr{E_{i,1}}} y_i(e_{i,1}) - \sum_{i=1}^{n}  \frac{\Pr{E_{i,0}}}{\Pr{E_{i,0}}} y_i(e_{i,0})} 
= \frac{1}{n} \paren[\Bigg]{\sum_{i=1}^{n}  y_i(e_{i,1}) - \sum_{i=1}^{n}  y_i(e_{i,0})} 
= \ate \enspace.
\qedhere
\end{align*}
\end{proof}

\subsubsection{Bounds on Correlation Terms}\label{sec:probability_bounds}

We now present bounds on the diagonal and off-diagonal terms of the covariance matrix of $\eate$.
Our proof will follow the same general structure as the proof for $q=2$ in \cite{kandiros2024conflict}, but with some important modifications. A key challenge in \cite{kandiros2024conflict} was to calculate the marginal and joint probabilities of the events $E_{(i,k)}$. A crucial part of the argument was a linearization step, which used the fact that the importance ordering is such that each $i$ has at most $\lamH$ neighbors that appear before $i$ in the ordering (the so called \emph{more important neighbors}).
Here, we will need a generalization of this property, since the probabilities are non-uniform. In particular, we will use the crucial property of the ordering that was presented in Lemma~\ref{lem:ordering} to justify the linearization that will allow us to calculate the probabilities of the events $E_{i,k}$.

We start with the bound on the diagonal terms of the covariance matrix of $\eate$, which amount to lower bounding the probability of the events $E_{i,k}$.

\begin{lemma}\label{lem:marg_prob}
    Let $q \geq 2$ be a moment restriction. Suppose we sample the $U_{i,k},Z$ variables according to the design in Algorithm~\ref{alg:CGD} with importance ordering $\pi$ given by Lemma~\ref{lem:ordering}. For every node $e_{i,k} \in V(\cH)$
    we have that if $C \geq 2$, then
    \[
    \Var[\Big]{\frac{\indicator{E_{i,k}}}{\Pr{E_{i,k}}}} \leq \frac{C e^{\frac{1}{C}}}{p_{i,k}}\enspace.
    \]
\end{lemma}
\begin{proof}
    We start by noticing that 
    \[
\Var[\Big]{\frac{\indicator{E_{i,k}}}{\Pr{E_{i,k}}}} = \frac{1}{\Pr{E_{i,k}}} - 1 \enspace.
\]
    By the definition of the design, we have
\begin{align*}
\Pr{E_{(i,k)}} & = \Pr[\big]{U_{i,k} = e , U_{j,l} = * \text{ for all } e_{j,l} \in \mathcal{N}_b^\pi(e_{i,k})} \\
&= \Pr{U_{i,k} = e} \prod_{e_{j,l} \in \mathcal{N}_b^\pi(e_{i,k})}\Pr{U_{j,l} = * } \\
& = \frac{p_{i,k}}{C} \prod_{e_{j,l} \in \mathcal{N}_b^\pi(e_{i,k})} \paren[\Big]{1 - \frac{p_{j,l}}{C}} \enspace,
\end{align*}
where we used the independence of the variables $U_{i,k}$. Now, consider the function $f(x) = 1 - x - e^{-2x}$ for $x \in [0,1]$. We have that
\[
f'(x) = -1 + 2 e^{-2x} \quad,\quad f''(x) = -4e^{-2x}
\]
Clearly, $f''(x) < 0$ for all $x \in [0,1]$, so that $f'$ is decreasing. Also, note that the unique root of $f'$ is $\ln(2)/2 > 1/4$. Thus, as long as $x \leq 1/4$, we have $f'(x) > 0$, which implies that $f$ is increasing in the interval $[0,1/4]$. Since $f(0) = 0$, it follows that $f(x) \geq 0$ for all $x \in [0,1/4]$. In particular, we have $1 - x \geq e^{-2x}$ for all $x \in [0,1/4]$. Now, apply this inequality with $x = p_j/C$. By Lemma~\ref{lem:ordering}, we have for all $i$
\[
\sum_{j \in \mathcal{N}_b^\pi(i)} p_j \leq \frac{1}{2}
\]
This implies that $p_{i,k}/ C \leq 1/4$ for all $i$ as long as $C \geq 2$. Therefore, we can apply the inequality $1 - x \geq e^{-2x}$ to each term in the product, yielding
\[
\prod_{e_{j,l} \in \mathcal{N}_b^\pi(e_{i,k})} \paren[\Big]{1 - \frac{p_{j,l}}{C}} \geq \prod_{e_{j,l} \in \mathcal{N}_b^\pi(e_{i,k})} e^{-2p_{j,l}/C} = e^{-2 \sum_{e_{j,l} \in \mathcal{N}_b^\pi(e_{i,k})} p_{j,l} / C} \geq e^{-1/C} \enspace,
\]
from which the Lemma follows.
\end{proof}

Next, we provide bounds on the covariance between two terms corresponding to different exposures.
These involve bounding the joint probabilities of events $E_{i,k}$ and $E_{j,l}$.

\begin{lemma}\label{lem:joint_prob}
    Let $q \geq 2$ be a moment restriction. Suppose we sample the $U_{i,k},Z$ variables according to the design in Algorithm~\ref{alg:CGD} with importance ordering $\pi$ given by Lemma~\ref{lem:ordering} and $C \geq 1$.
    For every pair of nodes $e_{i,k},e_{j,l} \in V(\cH)$, if $d(e_{i,k},e_{j,l})$ denotes the graph distance between $e_{i,k}$ and $e_{j,l}$ in $\cH$, then we have the following:
    \begin{itemize}
        \item If $d(e_{i,k},e_{j,l}) = 1$, then 
        \[
        \Cov[\Big]{\frac{\indicator{E_{i,k}}}{\Pr{E_{i,k}}}, \frac{\indicator{E_{j,l}}}{\Pr{E_{j,l}}}} = -1 \enspace.
        \]
        \item If $d(e_{i,k},e_{j,l}) = 2$, then
        \[
        \abs[\Big]{\Cov[\Big]{\frac{\indicator{E_{i,k}}}{\Pr{E_{i,k}}}, \frac{\indicator{E_{j,l}}}{\Pr{E_{j,l}}}}} \leq \frac{2}{2C - 1}\cdot \sum_{e_{r,s} \in N(e_{i,k}) \cap N(e_{j,l})} p_{r,s} \enspace.
        \]
        \item If $d(e_{i,k},e_{j,l}) > 2$, then
        \[
        \Cov[\Big]{\frac{\indicator{E_{i,k}}}{\Pr{E_{i,k}}}, \frac{\indicator{E_{j,l}}}{\Pr{E_{j,l}}}} = 0 \enspace.
        \]
    \end{itemize}
\end{lemma}

\begin{proof}
    By definition, the event $E_{i,k}$ depends on the values of the variables $U_{i,k}$ and $U_{j,l}$ for all $j \in \mathcal{N}(e_{i,k})$. Thus, if $d(e_{i,k},e_{j,l})> 2$, then the events $E_{i,k}$ and $E_{j,l}$ depend on disjoint sets of variables, so they are independent, which implies that the covariance is zero. If $d(e_{i,k},e_{j,l}) = 1$, then the events $E_{i,k}$ and $E_{j,l}$ are mutually exclusive. Indeed, suppose without loss of generality that $e_{i,k}$ appears before $e_{j,l}$ in the ordering, then if $E_{i,k}$ happens we must have $U_{i,k} = e$, but if $E_{j,l}$ also happens, then we must have $U_{i,k} = *$, which is a contradiction. Thus, 
    \[
    \Cov[\Big]{\frac{\indicator{E_{i,k}}}{\Pr{E_{i,k}}}, \frac{\indicator{E_{j,l}}}{\Pr{E_{j,l}}}} = \E[\Bigg]{\frac{\indicator{E_{i,k}}}{\Pr{E_{i,k}}} \frac{\indicator{E_{j,l}}}{\Pr{E_{j,l}}}} -1 = -1 \enspace.
    \]
    Finally, let us examine the case where $d(e_{i,k},e_{j,l}) = 2$. In this case, we can write
\begin{align*}
& \Cov[\Big]{\frac{\indicator{E_{i,k}}}{\Pr{E_{i,k}}}, \frac{\indicator{E_{j,l}}}{\Pr{E_{j,l}}}} \\
& = \E[\Big]{\frac{\indicator{E_{i,k}}}{\Pr{E_{i,k}}} \frac{\indicator{E_{j,l}}}{\Pr{E_{j,l}}}} -1 \\
& = \frac{\Pr{E_{i,k} \cap E_{j,l}}}{\Pr{E_{i,k}} \Pr{E_{j,l}}} -1 \\
& = \frac{\Pr[\big]{U_{i,k} = U_{j,l} = e_1 , U_{r,s} = * ,\forall e_{r,s} \in \mathcal{N}_b^\pi(e_{i,k}) \cup \mathcal{N}_b^\pi(e_{j,l})}}{\Pr[\big]{U_{i,k} = e_1, U_{r,s} = * ,\forall e_{r,s} \in \mathcal{N}_b^\pi(e_{i,k})}\Pr[\big]{U_{j,l} = e_1, U_{r,s} = * ,\forall e_{r,s} \in \mathcal{N}_b^\pi(e_{j,l})}} - 1\\
&= \frac{1}{\prod_{e_{r,s} \in \mathcal{N}_b^\pi(e_{i,k}) \cap \mathcal{N}_b^\pi(e_{j,l})} \paren[\Big]{1 - \frac{p_{r,s}}{C}}} - 1 
\end{align*}
By the inclusion-exclusion principle, we have that 
\[
\prod_{e_{r,s} \in \mathcal{N}_b^\pi(e_{i,k}) \cap \mathcal{N}_b^\pi(e_{j,l})} \paren[\Big]{1 - \frac{p_{r,s}}{C}} \geq 1 - \sum_{e_{r,s} \in \mathcal{N}_b^\pi(e_{i,k}) \cap \mathcal{N}_b^\pi(e_{j,l})} \frac{p_{r,s}}{C} \enspace.
\]
Hence, we have that 
\begin{align*}
    \frac{1}{\prod_{e_{r,s} \in \mathcal{N}_b^\pi(e_{i,k}) \cap \mathcal{N}_b^\pi(e_{j,l})} \paren[\Big]{1 - \frac{p_{r,s}}{C}}} - 1 &= \frac{1 - \prod_{e_{r,s} \in \mathcal{N}_b^\pi(e_{i,k}) \cap \mathcal{N}_b^\pi(e_{j,l})} \paren[\Big]{1 - \frac{p_{r,s}}{C}}}{\prod_{e_{r,s} \in \mathcal{N}_b^\pi(e_{i,k}) \cap \mathcal{N}_b^\pi(e_{j,l})} \paren[\Big]{1 - \frac{p_{r,s}}{C}}} \\
& \leq \frac{\sum_{e_{r,s} \in \mathcal{N}_b^\pi(e_{i,k}) \cap \mathcal{N}_b^\pi(e_{j,l})} \frac{p_{r,s}}{C}}{1 - \sum_{e_{r,s} \in \mathcal{N}_b^\pi(e_{i,k}) \cap \mathcal{N}_b^\pi(e_{j,l})} \frac{p_{r,s}}{C}} \\
&\leq \frac{2C}{2C - 1} \sum_{e_{r,s} \in \mathcal{N}_b^\pi(e_{i,k}) \cap \mathcal{N}_b^\pi(e_{j,l})} \frac{p_{r,s}}{C} \enspace,
\end{align*}
where in the last step we used the fact that 
\[
\sum_{e_{r,s} \in \mathcal{N}_b^\pi(e_{i,k}) \cap \mathcal{N}_b^\pi(e_{j,l})} \frac{p_{r,s}}{C} \leq \sum_{e_{r,s} \in \mathcal{N}_b^\pi(e_{i,k})} \frac{p_{r,s}}{C} \leq \frac{1}{2C} \enspace,
\]
which follows from Lemma~\ref{lem:ordering}. Finally, we trivially have
\[
\sum_{e_{r,s} \in \mathcal{N}_b^\pi(e_{i,k}) \cap \mathcal{N}_b^\pi(e_{j,l})} p_{r,s} \leq \sum_{e_{r,s} \in \mathcal{N}(e_{i,k}) \cap \mathcal{N}(e_{j,l})} p_{r,s} \enspace,
\]
which concludes the proof.
\end{proof}

\subsubsection{Matrix Norm Bounds}\label{sec:matrix_bound}

Lemmas~\ref{lem:marg_prob} and~\ref{lem:joint_prob} suggest that the covariance matrix associated with the estimator $\eate$ can be split into three matrices, one diagonal matrix corresponding to $e_{i,k} = e_{j,l}$, one corresponding to pairs $e_{i,k}, e_{j,l}$ with $d(e_{i,k},e_{j,l}) = 1$ and one for $d(e_{i,k},e_{j,l}) = 2$. A similar strategy was employed in \citet{kandiros2024conflict}. In particular, we define for all pairs $e_{i,k}, e_{j,l} \in V(\cH)$ 
\[
C_{(i,k),(j,l)} = \Cov[\Big]{\frac{\indicator{E_{i,k}}}{\Pr{E_{i,k}}}, \frac{\indicator{E_{j,l}}}{\Pr{E_{j,l}}}}\enspace.
\]
Then, let us define the matrices $\mat{T}^{(t)} \in \R^{2n \times 2n}$ for $t \in \setb{0,1,2}$ as follows:
\[
\mat{T}^{(t)}_{(i,k),(j,l)} = \begin{cases} |C_{(i,k),(j,l)}| & \text{if } d(e_{i,k},e_{j,l}) = t\\
0 & \text{otherwise}
\end{cases} \enspace.
\]

The following Lemmas upper bound the quadratic forms of the matrices $\mat{T}^{(0)}, \mat{T}^{(1)}, \mat{T}^{(2)}$, evaluated at vectors whose $q$-th norm is at most $1$. These bounds will be crucial for bounding the variance of $\eate$.
The proofs are based on careful applications of H\"older's inequality.
For a vector $x \in \R^{2n}$, we will sometimes use the double indexing $x_{i,k}$ to refer to the coordinate of $x$ corresponding to node $e_{i,k} \in V(\cH)$.

\begin{lemma}\label{lem:C0}
    For any vector $x \in \R^{2n}$ with $\norm{x}_q \leq 1$, we have
    \[
    x^\top \mat{T}^{(0)} x \leq 2C e^{\frac{1}{C}} \cdot n^{1 - \frac{2}{q}} \lamH^{\frac{2}{q}} d_{\mathrm{avg}}(\cH)^{1-\frac{2}{q}}\enspace.
    \]
\end{lemma}
\begin{proof}
    Let $m$ denote the number of edges in the conflict graph $\cH$.
    Using H\"older's inequality, we have that
\begin{align*}
x^\top \mat{T}^{(0)} x &\leq C e^{\frac{1}{C}} \cdot \sum_{e_{i,k} \in V(\cH)} x_{i,k}^2 \frac{1}{p_{i,k}}\\
\intertext{using Lemma~\ref{lem:marg_prob}}\\
& \leq C e^{\frac{1}{C}} \cdot \paren[\Bigg]{\sum_{e_{i,k} \in V(\cH)} x_{i,k}^q}^{2/q} \paren[\Bigg]{\sum_{e_{i,k} \in V(\cH)} \paren[\Big]{\frac{1}{p_{i,k}}}^{\frac{q}{q-2}}}^{(q-2)/q} \\
& \leq C e^{\frac{1}{C}} \cdot \paren[\Bigg]{\sum_{e_{i,k} \in V(\cH)} \paren[\Big]{\frac{1}{p_{i,k}}}^{\frac{q}{q-2}}}^{(q-2)/q} \\
\intertext{by the assumption on $x$}\\
& = C e^{\frac{1}{C}} \cdot \paren[\Bigg]{\sum_{e_{i,k} \in V(\cH)} (2 \lamH^{2/q} d_{i,k}^{1-2/q})^{\frac{q}{q-2}}}^{(q-2)/q} \\
& = 2C e^{\frac{1}{C}} \cdot \lamH^{2/q} \paren[\Big]{\sum_{e_{i,k} \in V(\cH)} d_{i,k}}^{(q-2)/q} \\
& = 2 C e^{\frac{1}{C}} \cdot \lamH^{2/q} (2m)^{1-2/q}\\
&= 2 C e^{\frac{1}{C}} \cdot n^{1 - \frac{2}{q}} \lamH^{\frac{2}{q}} d_{\mathrm{avg}}(\cH)^{1-\frac{2}{q}} \enspace.
\end{align*}
\end{proof}

\begin{lemma}\label{lem:C1}
    For any vector $x \in \R^{2n}$ with $\norm{x}_q \leq 1$, we have
    \[
    x^\top \mat{T}^{(1)} x \leq n^{1 - \frac{2}{q}}\lamH^{\frac{2}{q}} d_{\mathrm{avg}}(\cH)^{1-\frac{2}{q}}\enspace.
    \]
\end{lemma}
\begin{proof}
    We have that 
\begin{align*}
x^\top \mat{T}^{(1)} x & = \sum_{e_{i,k} \in V_\cH} x_{i,k} \sum_{e_{j,l} \in V_\cH} \mat{T}_{(i,k),(j,l)}^{(1)} x_{j,l} \\
&\leq \sum_{e_{i,k} \in V_\cH} |x_{i,k}| \sum_{e_{j,l} \in V_\cH} |\mat{T}_{(i,k),(j,l)}^{(1)}| |x_{j,l}| \\
&\leq \paren[\Big]{\sum_{e_{i,k} \in V_\cH} |x_{i,k}|^q}^{1/q} \paren[\Big]{\sum_{e_{i,k} \in V_\cH} \paren[\Big]{\sum_{e_{j,l} \in V_\cH} |\mat{T}_{(i,k),(j,l)}^{(1)}| |x_{j,l}|}^{\frac{q}{q-1}}}^{(q-1)/q} \\
\intertext{by H\"older's inequality}\\
& \leq \paren[\Big]{\sum_{e_{i,k} \in V_\cH} \paren[\Big]{\sum_{e_{j,l} \in V_\cH} |\mat{T}_{(i,k),(j,l)}^{(1)}| |x_{j,l}|}^{\frac{q}{q-1}}}^{(q-1)/q} \\
\intertext{by the assumption on $x$}\\
& \leq \paren[\Bigg]{\sum_{e_{i,k} \in V_\cH} \paren[\Bigg]{\paren[\Big]{\sum_{e_{j,l} \in V_\cH} |\mat{T}_{(i,k),(j,l)}^{(1)}| |x_{j,l}|^{q/2}}^{2/q} \cdot d_{i,k}^{1 - 2/q}}^{\frac{q}{q-1}}}^{\frac{q-1}{q}}  \\
\intertext{by H\"older's inequality}\\
& = \paren[\Bigg]{\sum_{e_{i,k} \in V_\cH} \paren[\Bigg]{\paren[\Big]{\sum_{e_{j,l} \in V_\cH} |\mat{T}_{(i,k),(j,l)}^{(1)}| |x_{j,l}|^{q/2}}^{\frac{2}{q-1}} \cdot d_{i,k}^{\frac{q-2}{q-1}}}}^{\frac{q-1}{q}}  \\
& \leq \paren[\Bigg]{\paren[\Big]{\sum_{e_{i,k} \in V_\cH} \paren[\Big]{\sum_{e_{j,l} \in V_\cH} |\mat{T}_{(i,k),(j,l)}^{(1)}| |x_{j,l}|^{q/2}}^{2} }^{\frac{1}{q-1}}\cdot \paren[\Big]{\sum_{e_{i,k} \in V_\cH} \paren[\big]{d_{i,k}^{\frac{q-2}{q-1}}}^{\frac{q-1}{q-2}}}^{\frac{q-2}{q-1}}}^{\frac{q-1}{q}} \\
\intertext{by H\"older's inequality}\\
&= \paren[\Bigg]{\paren[\Big]{\sum_{e_{i,k} \in V_\cH} \paren[\Big]{\sum_{e_{j,l} \in V_\cH} |\mat{T}_{(i,k),(j,l)}^{(1)}| |x_{j,l}|^{q/2}}^{2} }^{\frac{1}{q-1}} \cdot \paren[\Big]{\sum_{e_{i,k} \in V_\cH} d_{i,k}}^{\frac{q-2}{q-1}}}^{\frac{q-1}{q}} \\
&= \paren[\Bigg]{\sum_{e_{i,k} \in V_\cH} \paren[\Big]{\sum_{e_{j,l} \in V_\cH} |\mat{T}_{(i,k),(j,l)}^{(1)}| |x_{j,l}|^{q/2}}^{2} }^{\frac{1}{q}}  \cdot (2m)^{\frac{q-2}{q}}\\
& \leq \paren[\Bigg]{\lamH^2 \sum_{e_{i,k} \in V_\cH} \paren[\big]{|x_{j,l}|^{q/2}}^2 }^{\frac{1}{q}}  \cdot (2m)^{\frac{q-2}{q}}\\
\intertext{since $|\mat{T}_{(i,k),(j,l)}^{(1)}|$ is the adjacency entry of graph $\cH$  by Lemma~\ref{lem:joint_prob} and using the definition of the maximum eigenvalue property}\\
& \leq \lambda^{\frac{2}{q}} (2m)^{1-\frac{2}{q}}\\
&= n^{1 - \frac{2}{q}} \lamH^{\frac{2}{q}} d_{\mathrm{avg}}(\cH)^{1-\frac{2}{q}} \enspace.
\end{align*}
\end{proof}

Now, let us handle the third matrix.

\begin{lemma}\label{lem:C2}
    For any vector $x \in \R^{2n}$ with $\norm{x}_q \leq 1$, we have
    \[
    x^\top \mat{T}^{(2)} x \leq \frac{4}{2C-1} \cdot n^{1 - \frac{2}{q}} \lamH^{\frac{2}{q}} d_{\mathrm{avg}}(\cH)^{1-\frac{2}{q}}\enspace.
    \]
\end{lemma}
\begin{proof}
    Let $A_\cH$ be the adjacency matrix of the conflict graph $\cH$.
    We have that
\begin{align*}
x^\top \mat{T}^{(2)} x & = \sum_{e_{i,k} \in V_\cH} x_{i,k} \sum_{e_{j,l} \in V_\cH} \mat{T}_{(i,k),(j,l)}^{(2)} x_{j,l} \\
&\leq \sum_{e_{i,k} \in V_\cH} |x_{i,k}| \sum_{e_{j,l} \in V_\cH} |\mat{T}_{(i,k),(j,l)}^{(2)}| |x_{j,l}| \\
&\leq \paren[\Big]{\sum_{e_{i,k} \in V_\cH} |x_{i,k}|^q}^{1/q} \paren[\Big]{\sum_{e_{i,k} \in V_\cH} \paren[\Big]{\sum_{e_{j,l} \in V_\cH} |\mat{T}_{(i,k),(j,l)}^{(2)}| |x_{j,l}|}^{\frac{q}{q-1}}}^{\frac{q-1}{q}} \\
\intertext{by H\"older's inequality}\\
& \leq \frac{2}{2C-1}\frac{2}{\lamH^{2/q}}  \paren[\Bigg]{\sum_{e_{i,k} \in V_\cH} \paren[\Big]{\sum_{e_{j,l} \in V_\cH} |x_{j,l}| \sum_{e_{r,s} \in \mathcal{N}(e_{i,k}) \cap \mathcal{N}(e_{j,l})} \frac{1}{d_{r,s}^{\frac{q-2}{q}}} }^{\frac{q}{q-1}}}^{\frac{q-1}{q}} \\
\intertext{using the bound on $C_{(i,k),(j,l)}$ of Lemma~\ref{lem:joint_prob}}\\
& \leq \frac{2}{2C-1}\frac{2}{\lamH^{2/q}}  \paren[\Bigg]{\sum_{e_{i,k} \in V_\cH} \paren[\Big]{\sum_{e_{j,l} \in V_\cH} |x_{j,l}| |\mathcal{N}(e_{i,k})\cap \mathcal{N}(e_{j,l})|^{\frac{2}{q}} \paren[\Big]{\sum_{e_{r,s} \in \mathcal{N}(e_{i,k}) \cap \mathcal{N}(e_{j,l})} \frac{1}{d_{r,s}}}^{\frac{q-2}{q}}  }^{\frac{q}{q-1}}}^{\frac{q-1}{q}} \\
\intertext{again using H\"older's inequality}\\
& \leq \frac{2}{2C-1}\frac{2}{\lamH^{2/q}}  \paren[\Bigg]{\sum_{e_{i,k} \in V_\cH} \paren[\Bigg]{\sum_{e_{j,l} \in V_\cH} |x_{j,l}|^{\frac{q}{2}} |\mathcal{N}(e_{i,k})\cap \mathcal{N}(e_{j,l})|}^{\frac{2}{q-1}}\\
&\quad \cdot \paren[\Bigg]{\sum_{e_{j,l} \in V_\cH} \sum_{e_{r,s} \in \mathcal{N}(e_{i,k}) \cap \mathcal{N}(e_{j,l})} \frac{1}{d_{r,s}} }^{\frac{q-2}{q-1}}}^{\frac{q-1}{q}}\\
\intertext{by another application of H\"older's inequality}\\
& \leq \frac{2}{2C-1}\frac{2}{\lamH^{2/q}} \paren[\Bigg]{ \paren[\Bigg]{\sum_{e_{i,k} \in V_\cH} \paren[\Bigg]{\sum_{e_{j,l} \in V_\cH} |x_{j,l}|^{\frac{q}{2}} |\mathcal{N}(e_{i,k})\cap \mathcal{N}(e_{j,l})|}^{2} }^{\frac{1}{q-1}} \\
&\quad\cdot \paren[\Bigg]{\sum_{e_{i,k} \in V_\cH} \sum_{e_{j,l} \in V_\cH} \sum_{e_{r,s} \in \mathcal{N}(e_{i,k}) \cap \mathcal{N}(e_{j,l})} \frac{1}{d_{r,s}} }^{\frac{q-2}{q-1}}}^{\frac{q-1}{q}}\\
\intertext{by yet another application of H\"older's inequality}\\
&= \frac{4}{(2C-1)\lamH^{2/q}}  \paren[\Bigg]{\sum_{e_{i,k} \in V_\cH} \paren[\Bigg]{\sum_{e_{j,l} \in V_\cH} |x_{j,l}|^{\frac{q}{2}} |\mathcal{N}(e_{i,k})\cap \mathcal{N}(e_{j,l})|}^{2} }^{\frac{1}{q}}\\
&\quad \cdot \paren[\Bigg]{\sum_{e_{i,k} \in V_\cH} \sum_{e_{j,l} \in V_\cH} \sum_{e_{r,s} \in \mathcal{N}(e_{i,k}) \cap \mathcal{N}(e_{j,l})} \frac{1}{d_{r,s}} }^{\frac{q-2}{q}}\\
&\leq \frac{4}{(2C-1)\lamH^{2/q}}  \paren[\Bigg]{\lamH^2 \sum_{e_{i,k} \in V_\cH}  |x_{i,k}|^{q}  }^{\frac{1}{q}} \cdot \paren[\Bigg]{\sum_{e_{i,k} \in V_\cH} \sum_{e_{j,l} \in V_\cH} \sum_{e_{r,s} \in \mathcal{N}(e_{i,k}) \cap \mathcal{N}(e_{j,l})} \frac{1}{d_{r,s}} }^{\frac{q-2}{q}}\\
\intertext{using the definition of the maximum eigenvalue of $A_\cH$}\\
& \leq \frac{4}{(2C-1)\lamH^{2/q}} \lamH^{\frac{4}{q}}  \paren[\Bigg]{\sum_{e_{i,k} \in V_\cH} \sum_{e_{j,l} \in V_\cH} \sum_{e_{r,s} \in \mathcal{N}(e_{i,k}) \cap \mathcal{N}(e_{j,l})} \frac{1}{d_{r,s}} }^{\frac{q-2}{q}}\\
\intertext{using the constraint for the vector $x$}\\
& = \frac{4}{2C-1} \lamH^{2/q} \paren[\Bigg]{\sum_{e_{r,s} \in V_\cH} \frac{1}{d_{r,s}} \sum_{e_{i,k} \in N(e_{r,s})} \sum_{e_{j,l} \in N(e_{r,s})} 1 }^{\frac{q-2}{q}}\\
& = \frac{4}{2C-1} \lamH^{2/q} \paren[\Bigg]{\sum_{e_{r,s} \in V_\cH} \frac{d_{r,s}^2}{d_{r,s}} }^{\frac{q-2}{q}}\\
& = \frac{4}{2C-1} \lamH^{2/q} (2m)^{1-2/q} \\
&= \frac{4}{2C-1} n^{1 - \frac{2}{q}} \lamH^{\frac{2}{q}} d_{\mathrm{avg}}(\cH)^{1-\frac{2}{q}} \enspace.
\end{align*}
\end{proof}

\subsubsection{Proof of Theorem~\ref{thm:conflict_graph_design}}\label{sec:proof_cgd}

We are now ready to prove Theorem~\ref{thm:conflict_graph_design}, which is a bound on the variance of $\eate$.

\begin{proof}[Proof of Theorem~\ref{thm:conflict_graph_design}]
To prove the upper bound, it suffices to produce a statistical procedure, whose risk is bounded by the appropriate quantity.
We will use as our design $\design$ the Conflict Graph Design of Algorithm~\ref{alg:CGD} with an appropriate value of $C> 0$ that we will specify later and as our estimator $\eate$, the modified Horvitz--Thompson estimator \eqref{eq:modified_ht}.
Although the Conflict Graph Design and modified Horvitz--Thompson estimator form a randomized statistical procedure (as defined in Section~\ref{sec:randomness_appendix}), its maximal risk provides an upper bound for the minimax risk (Proposition~\ref{prop:randomness_model}, Section~\ref{sec:randomness_appendix}).

Our goal is to bound
\[
\sup_{y \in \cM(G,q)} \Esub{(Z,U) \sim \design}{\paren[\big]{\eate(Z,U,y(Z)) - \ate(y)}^2}  = \sup_{y \in \cM(G,q)} \Var[\Big]{\eate(Z,U,y(Z))} \enspace.
\]
Let $y \in \cM(G,q)$.
Let us define the vector $\vec{\alpha} \in \R^{2n}$ as
\[
\vec{\alpha}_{i,k} = \frac{1}{n^{1/q}}|y_i(e_{i,k})|  \enspace.
\]
We begin by writing the variance of $\eate$ as a quadratic form of the covariance matrix with vector $\vec{\alpha}$, as in \cite{kandiros2024conflict}. In particular, we have
\begin{align*}
\Var{\eate} & = \Var[\Bigg]{\frac{1}{n} \paren[\Bigg]{\sum_{i=1}^{n} \paren[\Big]{\frac{\indicator{E_{i,1}}}{\Pr{E_{i,1}}} y_i(e_{i,1}) - \frac{\indicator{E_{i,0}}}{\Pr{E_{i,0}}} y_i(e_{i,0})}}} \\
& = \frac{1}{n^2} \sum_{e_{i,k} \in V_\cH} \sum_{e_{j,l} \in V_\cH} y_{i}(e_{i,k}) y_{j}(e_{j,l}) (-1)^{i+j}\Cov[\Big]{\frac{\indicator{E_{i,k}}}{\Pr{E_{i,k}}}, \frac{\indicator{E_{j,l}}}{\Pr{E_{j,l}}}} \\
&\leq \frac{1}{n^2} \sum_{e_{i,k} \in V_\cH} \sum_{e_{j,l} \in V_\cH} |y_{i}(e_{i,k})| |y_{j}(e_{j,l})| |C_{(i,k),(j,l)}|\\
& = \frac{n^{\frac{2}{q}}}{n^2} \vec{\alpha}^\top \paren[\Big]{\mat{T}^{(0)} + \mat{T}^{(1)} + \mat{T}^{(2)}} \vec{\alpha}
\end{align*}
Since $y \in \cM(G,q)$, the $q$-th norm of the vector $\vec{\alpha}$ is bounded as
\begin{align*}
    \paren[\Big]{\sum_{i=1}^{n} \sum_{k=0}^{1} |\vec{\alpha}_{i,k}|^q}^{1/q} & = \paren[\Bigg]{\frac{1}{n} \sum_{i=1}^{n} |y_i(e_{i,0})|^q + \frac{1}{n} \sum_{i=1}^{n} |y_i(e_{i,1})|^q}^{1/q} \\
& \leq 1\enspace,
\end{align*}

using the definition of $\cM(G,q)$. Now, applying Lemmas~\ref{lem:C0}, \ref{lem:C1} and \ref{lem:C2} to bound the three terms in the sum, we have
\begin{align*}
\Var{\eate} & \leq \frac{n^{\frac{2}{q}}}{n^2} \paren[\Big]{2C e^{\frac{1}{C}}+ 1+ \frac{4}{2C-1}}\lamH^{\frac{2}{q}} (2m)^{1-\frac{2}{q}}\\
&= \paren[\Big]{2C e^{\frac{1}{C}}+ 1+ \frac{4}{2C-1}} \cdot \frac{\lamH^{\frac{2}{q}} (2m/n)^{1-\frac{2}{q}}}{n} \\
&= \paren[\Big]{2C e^{\frac{1}{C}}+ 1+ \frac{4}{2C-1}} \cdot \frac{\lamH^{\frac{2}{q}} d_{\mathrm{avg}}^{1-\frac{2}{q}}}{n}\enspace,
\end{align*}
where in the last step we used the fact that $d_{\mathrm{avg}} = 2m/n$.
Choosing $C = 2$ yields the desired bound.
\end{proof}

\subsection{Bias-Variance Tradeoff }\label{sec:bias_variance_appendix}

In this section, we first establish in Section~\ref{sec:exponential_bv_appendix} the improved upper bound of Theorem~\ref{thm:bias_variance_exponential}. We then show in Section~\ref{sec:approximation_algo_appendix} that we can find a design-estimator pair that can be implemented in polynomial time, whose risk is within multiplicative logarithmic factors from this improved bound. 

\subsubsection{Optimal Bias-Variance Tradeoff (Theorem~\ref{thm:bias_variance_exponential})}\label{sec:exponential_bv_appendix}

We first introduce some notation that will be particularly useful in this section. 

For any subset of nodes $S \subseteq V_\cH$, we denote by $\cH[V_\cH \setminus S]= (V_\cH \setminus S, E(V_\cH \setminus S))$ the subgraph of $\cH$ that is \emph{induced} by the nodes in $V_\cH \setminus S$, i.e. for $i,j \in V_\cH \setminus S$, $(i,j) \in E(V_\cH \setminus S)$ if and only if $(i,j) \in E_\cH$. We denote by $\lambda(\cH[V_\cH \setminus S])$ the largest eigenvalue of the adjacency matrix of $\cH[V_\cH \setminus S]$.

As usual, we will use the notation $\widetilde{N}(i)$ for the extended neighborhood of a node $i$ in the original graph $G$.
Also, for convenience, we introduce the following function of a subset $S \subseteq V(\cH)$ and a moment parameter $q$.
\[
Q_{q}(S, \cH) = \lambda(\cH[V_\cH \setminus S])^{\frac{2}{q}}\cdot d_{\mathrm{avg}}(\cH[V_\cH \setminus S])^{1 - \frac{2}{q}}+ \frac{|S|^{2 - \frac{2}{q}}}{n^{1 - \frac{2}{q}}}\enspace.
\]
In this section, we will prove the following theorem.

\biasvarianceexponential*

We start by introducing the design that will be used in the proof of Theorem~\ref{thm:bias_variance_exponential}. We give the pseudocode in Algorithm~\ref{alg:CGD_exponential}. 

\begin{algorithm}[ht]
\KwIn{maximum eigenvalue $\lamH$, graph $G$, conflict graph $\cH$ and moment $q$.}
	\KwOut{Random intervention $Z \in \mathcal{Z} = \setb{0,1}^n$, variables $U_i$ for all $i \in V(\cH)$.}
    $S^* \gets  \argmin_{S \subseteq V(\cH)} Q_{q}(S,\cH)$ \label{algline:subset_selection}\\
	$\cH^* \gets \cH[V_\cH \setminus S^*]$ \\
    $\pi^* \gets$ Ordering on $V_\cH \setminus S^*$ produced by Algorithm~\ref{alg:ordering} with input $\cH^*$ and $q$ \\
    \For{$e_{i,k} \in V(\cH^*)$}{
    $
p_{i,k} \gets \frac{1}{2 \lambda(\cH^*)^{\frac{2}{q}} d_{i,k}(\cH^*)^{1-\frac{2}{q}}} 
$\;
    }
	Sample desired exposure variables $\setb{U_{i,k} : e_{i,k} \in V(\cH) \setminus S^*}$ independently and identically as
	$$
	U_{i,k} \gets
	\begin{cases*}
		e  \quad\text{with probability } \frac{p_{i,k}}{2 }\\
		*  \quad\text{with probability } 1 - \frac{p_{i,k}}{2 }
	\end{cases*}
	$$\\
	Initialize the intervention vector $Z \gets \vec{0}$ \\
	\For{$e_{i,k} \in V(\cH) \setminus S^*$}{
		\If{$U_{i,k} = e $ and $U_{j,l} = *$ for all $e_{j,l} \in \mathcal{N}^{\pi^*}_b(e_{i,k})$\label{algline:exposure-condition-bias-exponential}}{
			Update intervention vector: set 
            $
            Z_r = \indicator[\big]{ r \in h_i(e_{i,k})} \quad \text{for all } r \in \widetilde{N}(i)
			$\label{algline:set-intervention-bias-exponential} \;
		}
	}
\caption{Biased Conflict Graph Design}
\label{alg:CGD_exponential}
\end{algorithm}

Algorithm~\ref{alg:CGD_exponential} operates similarly to the CGD of Algorithm~\ref{alg:CGD}, but with one important difference.
In particular, in line~\ref{algline:subset_selection} we select a particular subset of nodes $S^*$ in the conflict graph $\cH$ whose outcomes we will not try to estimate.
Then, an ordering $\pi^*$ of the nodes in $V_\cH \setminus S^*$ is produced by calling the importance ordering procedure from Algorithm~\ref{alg:ordering} on the subgraph $\cH[V_\cH \setminus S^*]$ of the original conflict graph. 
The remaining steps mirror the CGD of Algorithm~\ref{alg:CGD}, where we iterate only over the exposures in $V_\cH \setminus S^*$ and assign treatment.

In order to describe our proposed estimator, we first introduce some definitions for arbitrary subsets of nodes in the conflict graph. 
In particular, for any subset $S \subseteq V_\cH$, let $\pi_S$ be the ordering of the nodes in $V_\cH \setminus S$ produced by the importance ordering procedure from Algorithm~\ref{alg:ordering} with input $\cH[V_\cH \setminus S]$ and $q$ (we omit the dependence on $q$ since it is clear from the context).
For all $e_{i,k} \in V_\cH \setminus S$, define the events
\[
E_{i,k}^{(S)} = \setb{U_{i,k} = e, U_{j,l} = * \text{ for all } e_{j,l} \in \mathcal{N}^{\pi_S}_b(e_{i,k})}\enspace.
\]
Note that we used a similar notation for an analogous event in the original CGD in Section~\ref{sec:cgd_general_appendix}, where the event depends on all neighbors of $e_{i,k}$ in the conflict graph. 
This could be thought of as a special case of the above definition, where $S = \emptyset$ and $\pi_\emptyset$ is the ordering produced by Algorithm~\ref{alg:ordering} on the entire conflict graph $\cH$.

For any subset $S \subseteq V_\cH$, we define the following estimator.
\begin{equation}
	\label{eq:biased_estimator}
	\eate_{S} = \frac{1}{n} \sum_{i=1}^n \paren[\Bigg]{\frac{\indicator{E_{i,1}^{(S)} }}{\Pr{E_{i,1}^{(S)}}} \indicator{e_{i,1} \notin S} - \frac{\indicator{E_{i,0}^{(S)} }}{\Pr{E_{i,0}^{(S)}}} \indicator{e_{i,0} \notin S} }y_i(Z)\enspace.
\end{equation}
All probabilities are calculated with respect to the randomness in the design of Algorithm~\ref{alg:CGD_exponential}.
Note that $\eate_{S}$ is a randomized estimator, in the sense defined in Section~\ref{sec:minimax_theory}, since it uses the values of the random variables $U_{i,k}$ in its definition.

For an arbitrary subset of exposures $S \subseteq V_\cH$, denote by $\ate_S$ the average treatment effect defined only on exposures in $V_\cH \setminus S$, i.e.
\[
\ate_S = \frac{1}{n} \sum_{i=1}^n \paren{ \indicator{e_{i,1} \notin S} y_i(e_{i,1}) - \indicator{e_{i,0} \notin S} y_i(e_{i,0}) }\enspace.
\]
We will use $\eate_{S^*}$ as our estimator for $\ate_{S^*}$, where $S^*$ is the subset of nodes in the conflict graph that minimizes $Q(S,q)$, as defined in line~\ref{algline:subset_selection} of Algorithm~\ref{alg:CGD_exponential}.
Our first proposition establishes that $\eate_{S^*}$ is an unbiased estimator for $\ate_{S^*}$.

\begin{proposition}\label{prop:unbiased_S}
	For any outcome function $y \in \cM(G)$, we have $\E{\eate_{S^*}} = \ate_{S^*}$.
\end{proposition}
\begin{proof}
	First, observe that if event $E_{i,k}^{(S^*)}$ holds, then $y_i(Z) = y_i(e_{i,k})$. The argument is essentially identical to that of Lemma~\ref{lem:desired_exposure} for the original CGD, since Algorithm~\ref{alg:CGD_exponential} essentially runs the conflict graph design on the set of nodes $V_\cH \setminus S^*$. 
Therefore, we have
\begin{align*}
	\E{\eate_{S^*}} & = \frac{1}{n} \sum_{i=1}^n \paren[\Bigg]{\frac{\Pr[\Big]{E_{i,1}^{(S^*)} }}{\Pr[\Big]{E_{i,1}^{(S^*)}}} \indicator{e_{i,1} \notin S^*} y_i(e_{i,1}) - \frac{\Pr[\Big]{E_{i,0}^{(S^*)} }}{\Pr[\Big]{E_{i,0}^{(S^*)}}} \indicator{e_{i,0} \notin S^*} y_i(e_{i,0}) }\\
	&= \frac{1}{n} \sum_{i=1}^n \paren[\Bigg]{\indicator{e_{i,1} \notin S^*} y_i(e_{i,1}) - \indicator{e_{i,0} \notin S^*} y_i(e_{i,0}) }\\
	&= \ate_{S^*} \enspace.
	\qedhere
\end{align*}
\end{proof}

We are now ready to prove Theorem~\ref{thm:bias_variance_exponential}. 

\begin{proof}[Proof of Theorem~\ref{thm:bias_variance_exponential}]
We will prove the upper bound on the minimax risk by bounding the risk of a particular design-estimator pair. We will use the Biased Conflict Graph Design of Algorithm~\ref{alg:CGD_exponential} as the design $\design$ and $\eate_{S^*}$ as the estimator, as defined in equation~\eqref{eq:biased_estimator}.
As observed above, $\design$ is a randomized experimental design and $\eate_{S^*}$ is a randomized estimator, since they both use the extra randomness of the variables $U_{i,k}$.
By Proposition~\ref{prop:randomness_model}, it suffices to bound the risk of this design-estimator pair to obtain an upper bound on the minimax risk.

By Proposition~\ref{prop:unbiased_S}, we have that $\eate_{S^*}$ is an unbiased estimator for $\ate_{S^*}$.
Since $\ate_{S^*}$ is not necessarily equal to $\ate$, the estimator $\eate_{S^*}$ may be biased for $\ate$.
Thus, we can decompose the MSE of $\eate_{S^*}$ as a sum of bias and variance terms, as follows.
\begin{equation}
	\label{eq:bias_var}
	\E[\big]{(\eate_{S^*} - \ate)^2} = \Var[\big]{\eate_{S^*}} + (\ate_{S^*} - \ate)^2\enspace.
\end{equation}
We will now bound each term separately, starting with the bias term.
For any $y \in \cM(G,q)$, we have
\begin{align}
	(\ate_{S^*} - \ate)^2 &= \paren[\Bigg]{\frac{1}{n} \sum_{i=1}^n \paren{ \indicator{e_{i,1} \in S^*} y_i(e_{i,1}) - \indicator{e_{i,0} \in S^*} y_i(e_{i,0}) }}^2\nonumber \\
	&\leq \frac{1}{n^2} \paren[\Bigg]{\sum_{e_{i,k} \in S^*} |y_i(e_{i,k})|}^2\nonumber \\
	\intertext{using triangle inequality} \nonumber \\
	&\leq \frac{1}{n^2} \paren[\Bigg]{\sum_{e_{i,k} \in S^*} |y_i(e_{i,k})|^q}^{\frac{2}{q}} \cdot \paren[\Bigg]{\sum_{e_{i,k} \in S^*} 1}^{\frac{2(q-1)}{q}}\nonumber \\
	\intertext{by H\"older's inequality}\nonumber \\
	&\leq \frac{|S^*|^{2 - \frac{2}{q}}}{n^{2- \frac{2}{q}}} \label{eq:bias_bound} \enspace,
\end{align}
using the fact that $y \in \cM(G,q)$ in the last step.

We now focus on the variance term. As we've already discussed, after selecting the subset $S^*$, Algorithm~\ref{alg:CGD_exponential} proceeds identically to the CGD of Algorithm~\ref{alg:CGD} but only on the subgraph $\cH[V_\cH \setminus S^*]$. Moreover, the estimator $\eate_{S^*}$ is essentially the exact analogue of the estimator $\eate$ defined in equation~\eqref{eq:modified_ht} for the original CGD, but only on the exposures in $V_\cH \setminus S^*$.
Therefore, we can use the analysis of the variance of the original CGD, as given in the proof of Theorem~\ref{thm:conflict_graph_design}, to conclude that for any $y \in \cM(G,q)$, we have

\begin{equation}
	\label{eq:var_bound}
	\Var[\big]{\eate_{S^*}} \leq \constant \cdot\frac{\lambda(\cH(V_\cH \setminus S^*))^{\frac{2}{q}}\cdot d_{\mathrm{avg}}(\cH[V_\cH \setminus S^*])^{1 - \frac{2}{q}}}{n}\enspace.
\end{equation}

Combining the bounds in equations~\eqref{eq:bias_bound} and~\eqref{eq:var_bound} with the bias-variance decomposition in equation~\eqref{eq:bias_var}, we have that 
\begin{align*}
\cR(G,\tau,q) &\leq \constant \cdot\frac{\lambda(\cH(V_\cH \setminus S^*))^{\frac{2}{q}}\cdot d_{\mathrm{avg}}(\cH[V_\cH \setminus S^*])^{1 - \frac{2}{q}}}{n} +  \frac{|S^*|^{\frac{2(q-1)}{q}}}{n^{2-\frac{2}{q}}}\\
&= \constant \cdot \frac{ Q_{q}(S^*,\cH)}{n} \enspace.
\end{align*}
Since $S^*$ was selected in line~\ref{algline:subset_selection} to optimize the right hand side of the above inequality, the result follows.
\end{proof}

\subsection{Polynomial Time Approximation Algorithm (Proposition~\ref{prop:approximation_algorithm})}\label{sec:approximation_algo_appendix}

In this section, we prove that there exists a design-estimator pair that run in polynomial time and achieve estimation rate within logarithmic factors from that of Theorem~\ref{thm:bias_variance_exponential}.
The precise statement follows.

\approximationalgorithm*

In order to establish Theorem~\ref{prop:approximation_algorithm}, it will be useful to introduce a few relevant computational problems.
The first one is known as the \emph{Partial Vertex Cover} problem.

\begin{definition}[Partial Vertex Cover]
The \emph{Partial Vertex Cover} problem is defined as follows.

\textbf{Input:} A graph $G=(V,E)$ and an integer $k$.

\textbf{Output:} A minimum-size subset $S \subseteq V$ such that
$|E(G[V \setminus S])| \le k$.
\end{definition}

For a graph $G$ and integer $k$, we denote by $\mathrm{OPT}_{pvc}(G,k)$ the size of the optimal solution to thePartial Vertex Cover problem on input $(G,k)$.

The second problem is a variant of the first one. We will refer to it as the \emph{Partial Eigenvalue Cover} problem, although this is not a standard terminology in the literature.

\begin{definition}[Partial Eigenvalue Cover]
The \emph{Partial Eigenvalue Cover} problem is defined as follows.

\textbf{Input:} A graph $G=(V,E)$ and an integer $k$.

\textbf{Output:} A minimum-size subset $S \subseteq V$ such that
$\lambda(G[V \setminus S]) \le k$.
\end{definition}

For a graph $G$ and integer $k$, we denote by $\mathrm{OPT}_{pec}(G,k)$ the size of the optimal solution to the Partial Eigenvalue Cover problem on input $(G,k)$.

We now present two results from the literature, which provide polynomial time approximation algorithms for the above problems. 

\begin{theorem}[Theorem 2.1 in \citet{gandhi2004approximation}]
\label{thm:partial_vertex_cover}
There exists a polynomial time algorithm that takes as input a graph $G$ and integer $k$ and outputs a subset $\texttt{PVC-Approx}(G,k) \subseteq V$ such that $|\texttt{PVC-Approx}(G,k)| \le 2 \cdot \mathrm{OPT}_{pvc}(G,k)$.
\end{theorem}

\begin{theorem}[Lemma 5.1 in \citet{saha2015approximation}]
\label{thm:partial_eigenvalue_cover}
There exists a polynomial time algorithm that takes as input a graph $G$ and integer $k$ and outputs a subset $\texttt{PEC-Approx}(G,k) \subseteq V(G)$ such that 
$$
|\texttt{PEC-Approx}(G,k)| = \bigO{\mathrm{OPT}_{pec}(G,k) \cdot \log n}
$$
and 
$$
\lambda\paren[\big]{G[V \setminus \texttt{PEC-Approx}(G,k)]} \le 2 k \enspace.
$$
\end{theorem}

We will use the above two algorithms as subroutines in order to find a subset $S$ of vertices in the conflict graph $\cH$ that approximately minimizes the expression in Theorem~\ref{thm:bias_variance_exponential}.

In particular, we describe the design in Algorithm~\ref{alg:approximation_algorithm}. 

\begin{algorithm}[ht]
\KwIn{Interference network $G$, conflict graph $\cH$ and moment $q$.}
	\KwOut{Random intervention $Z \in \mathcal{Z} = \setb{0,1}^n$, variables $U_{i,k}$ for all $e_{i,k} \in V(\cH)$.}
    \For{$k_1 = 1$ to $|E(\cH)|$}{
    	\For{$k_2 = 1$ to $\lceil \lamH \rceil$}{
            $S_1 \gets \texttt{PVC-Approx}(\cH, k_1)$\;
    		$S_2 \gets \texttt{PEC-Approx}(\cH, k_2)$\;
    		$S_{k_1,k_2} \gets S_1 \cup S_2$\;
    	}
    }
    $\hat{S} \gets \argmin_{k_1, k_2} Q_{q}(S_{k_1,k_2}, \cH)$ \label{algline:subset_selection_efficient}\\
	$\hat{\cH} \gets \cH[V_\cH \setminus \hat{S}]$ \\
    $\pi_{\hat{S}} \gets$ Ordering produced by Algorithm~\ref{alg:ordering} with input $\hat{\cH}$ and $q$ \label{algline:ordering}\\
    \For{$e_{i,k} \in V_\cH \setminus \hat{S}$}{
    $
p_{i,k} \gets \frac{1}{2 \lamH^{\frac{2}{q}} d_{i,k}(\hat{\cH})^{1-\frac{2}{q}}} 
$\;
    }
	Sample desired exposure variables $\setb{U_{i,k} : e_{i,k} \in V_\cH \setminus \hat{S}}$ independently and identically as
	$$
	U_{i,k} \gets
	\begin{cases*}
		e  \quad\text{with probability } \frac{p_{i,k}}{2 }\\
		*  \quad\text{with probability } 1 - \frac{p_{i,k}}{2}
	\end{cases*}
	$$\\
	Initialize the intervention vector $Z \gets \vec{0}$ \\
	\For{$e_{i,k} \in V_\cH \setminus \hat{S}$}{
		\If{$U_{i,k} = e $ and $U_{j,l} = *$ for all $e_{j,l} \in \mathcal{N}^{\pi_{\hat{S}}}_b(e_{i,k})$\label{algline:exposure-condition-bias}}{
			Update intervention vector: set 
            $
            Z_r = \indicator[\big]{ r \in h_i(e_{i,k})} \quad \text{for all } r \in \widetilde{N}(i)
			$
			\label{algline:set-intervention-bias} \;
		}
	}
\caption{Approximation Design}
\label{alg:approximation_algorithm}
\end{algorithm}

The only difference between Algorithm~\ref{alg:approximation_algorithm} and the previously described Algorithm~\ref{alg:CGD_exponential} is in the selection of the optimal subset of exposures. While Algorithm~\ref{alg:CGD_exponential} searches over all subsets of exposures to minimize $Q_q(S,\cH)$, which takes time exponential in $n$, Algorithm~\ref{alg:approximation_algorithm} uses a different procedure that runs in polynomial time.

We next establish that the value $Q_{q}(\hat{S},\cH)$ is only logarithmically larger than the optimal value of $Q_{q}(S^*,\cH)$. 

\begin{lemma}
\label{lem:approximation_quality}
Let $S^* = \argmin_{S \subseteq V(\cH)} Q_{q}(S,\cH)$ as defined in Algorithm~\ref{alg:CGD_exponential} and $\hat{S}$ be the subset selected in line~\ref{algline:subset_selection_efficient} of Algorithm~\ref{alg:approximation_algorithm}. Then, \[
Q_{q}(\hat{S},\cH) = \bigO[\Big]{(\log n)^{2 - \frac{2}{q}}} \cdot Q_{q}(S^*,\cH) \enspace.
\]
\end{lemma}
\begin{proof}
Let $k_1^* = |E(V_\cH \setminus S^*)|$ and $k_2^* = \lceil \lambda(\cH(V_\cH \setminus S^*)) \rceil$. We first note that $\mathrm{OPT}_{pvc}(\cH, k_1^*) \le |S^*|$ and $\mathrm{OPT}_{pec}(\cH, k_2^*) \le |S^*|$. This is because $S^*$ is a feasible solution to both the \texttt{Partial Vertex Cover} problem with input $(\cH, k_1^*)$ and the \texttt{Partial Eigenvalue Cover} problem with input $(\cH, k_2^*)$.

For convenience, let us denote
\[
S_1 = \texttt{PVC-Approx}(\cH, k_1^*)\quad\text{and}\quad S_2 = \texttt{PEC-Approx}(\cH, k_2^*)\enspace.
\]

Then, by Theorem~\ref{thm:partial_vertex_cover}, we have that $|S_1| \le 2 \cdot \mathrm{OPT}_{pvc}(\cH, k_1^*) \le 2 |S^*|$. Also, by Theorem~\ref{thm:partial_eigenvalue_cover}, we have that 

$$|S_2| = \bigO{\mathrm{OPT}_{pec}(\cH, k_2^*) \cdot \log n} = \bigO{|S^*| \cdot \log n} \quad \text{and} \quad \lceil \lambda(\cH[V_\cH \setminus S_2]) \rceil \le 2 k_2^*\enspace.$$
Finally, $S_{k_1^*, k_2^*} = S_1 \cup S_2$ has size $|S_{k_1^*, k_2^*}| \le |S_1| + |S_2| = \bigO{|S^*| \cdot \log n}$.

Therefore, at the iteration where $k_1 = k_1^*$ and $k_2 = k_2^*$, we have that
\begin{align*}
n^{1 - \frac{2}{q}} \cdot Q_{q}(S_{k_1^*, k_2^*},\cH) & =  \lambda(\cH(V_\cH \setminus S_{k_1^*, k_2^*}))^{\frac{2}{q}}\cdot \paren{2|E(V_\cH \setminus S_{k_1^*, k_2^*})|}^{1 - \frac{2}{q}}+  |S_{k_1^*, k_2^*}|^{\frac{2(q-1)}{q}}\\
& \le  (2 k_2^*)^{\frac{2}{q}}\cdot \paren{2k_1^*}^{1 - \frac{2}{q}}+  \bigO{|S^*| \cdot \log n}^{\frac{2(q-1)}{q}}\\
& = \bigO[\Big]{(\log n)^{2 - \frac{2}{q}}}  \cdot \paren[\Bigg]{\lambda(\cH(V_\cH \setminus S^*))^{\frac{2}{q}}\cdot |E(V_\cH \setminus S^*)|^{1 - \frac{2}{q}}+  |S^*|^{\frac{2(q-1)}{q}}}\\
& = \bigO[\Big]{(\log n)^{2 - \frac{2}{q}}} \cdot n^{1 - \frac{2}{q}}  Q_{q}(S^*,\cH)\enspace.
\end{align*}

By definition, we have $Q_{q}(\hat{S},\cH) \le Q_{q}(S_{k_1^*, k_2^*},\cH)$, which implies the result.
\end{proof}

We are now ready to establish Proposition~\ref{prop:approximation_algorithm}. 

\begin{proof}[Proof of Proposition~\ref{prop:approximation_algorithm}]
As our design $\design$, we will use the design of Algorithm~\ref{alg:approximation_algorithm} and the estimator $\eate_{\hat{S}}$, where $\hat{S}$ is defined in line~\ref{algline:subset_selection_efficient} of Algorithm~\ref{alg:approximation_algorithm}.

We first note that Algorithm~\ref{alg:approximation_algorithm} runs in polynomial time. This is because the total number of iterations in the for loops in the beginning is $\bigO{|E(\cH)| \cdot \lamH} = \bigO{n^3}$, and each iteration involves running the algorithms of Theorem~\ref{thm:partial_vertex_cover} and Theorem~\ref{thm:partial_eigenvalue_cover}, which run in polynomial time. The remaining steps of the algorithm also clearly run in polynomial time. 
After having obtained $\hat{S}$, one can also compute the estimator $\eate_{\hat{S}}$ in polynomial time.

Next, we can proceed exactly as in the proof of Theorem~\ref{thm:bias_variance_exponential} to show that

\[
 \mathrm{MMSE}\paren[\big]{\design,\eate_{\hat{S}},\cM(G,q)} \leq \constant \cdot \frac{Q_{q}(\hat{S},\cH)}{n^{2 - \frac{2}{q}}}\enspace.
\]
Finally, by Lemma~\ref{lem:approximation_quality}, we have that 
$$
Q_{q}(\hat{S},\cH) = \bigO[\Big]{(\log n)^{2 - \frac{2}{q}}} \cdot Q_{q}(S^*,\cH)\enspace,
$$ 
where $S^* = \argmin_{S \subseteq V(\cH)} Q_{q}(S,\cH)$ as defined in Algorithm~\ref{alg:CGD_exponential}. The result follows.
    
\end{proof}

	\section{Illustrations of Main Results}\label{sec:applications_appendix}

In this section, we present the proofs of two illustrations of the main result for estimating two well-studied causal effects, as presented in Section~\ref{sec:applications}.
Section~\ref{sec:dte_appendix} contains the proof for the Direct Treatment Effect, while Section~\ref{sec:gate_appendix} contains the proof for the Global Average Treatment Effect.

\subsection{Direct Treatment Effect (Corollary~\ref{corollary:direct-rate})}\label{sec:dte_appendix}

The Direct Treatment Effect (DTE) was formally defined in Section~\ref{sec:preliminaries}.
We state and prove a more formal version of Corollary~\ref{corollary:direct-rate} below.

\begin{corollary}\label{cor:dte_formal}
    The minimax rate of estimating the Direct Treatment Effect $\ate_{\mathrm{DTE}}$ on a network $G$ is bounded as
	\[
	\max\paren[\Bigg]{ \frac{1}{|\indset(G)|} , \frac{\sqrt{d_{\mathrm{avg}}(G)}}{n} } 
	\lesssim 
	\risk(G,\ate_{\mathrm{DTE}},2) 
	\lesssim
	\frac{\lambda(G)+1}{n}
	\enspace,
	\]
	where $\indset(G)$ is the maximum independent set in $G$, $\davg(G)$ is its average degree, and $\lambda(G)$ is the largest eigenvalue of its adjacency matrix.
\end{corollary}

\begin{proof}
    Let $\cH$ be the conflict graph corresponding to $G$ and $\ate_{\mathrm{DTE}}$.
    Let $S = \setb{e_{i,1}: i \in [n]}$ be the subset of size $n$ of the $e_{i,1}$ exposures.
    Applying the upper bound of Theorem~\ref{thm:conflict_graph_design} for $q=2$, as well as the global lower bound of Theorem~\ref{thm:global_lower_bound} with $T = S$ and the local lower bound of Corollary~\ref{cor:local_lower_bound_general} together with Proposition~\ref{prop:critical_to_avg}, we get
    \begin{equation}\label{eq:first_bound_dte}
    \max\paren[\bigg]{\frac{1}{|\cI(\cH,S)|},\frac{\sqrt{d_{\mathrm{avg}}(\cH)}}{n}}\lesssim \cR(G,\ate_{\mathrm{DTE}}) \lesssim \frac{\lamH}{n}
    \end{equation}
    We now connect the quantites on the upper and lower bounds with the desired quantities of the interference network $G = (V,E)$. 

    As we have discussed in Section~\ref{sec:conflict_graph}, the conflict graph $\cH$ for $\ate_{\mathrm{DTE}}$ has the following structure: $(e_{i,k},e_{j,l}) \in E_\cH$ if and only if $(i,j) \in E$ and $(k,l) \neq (0,0)$ or $i =j$ and $k \neq l$. 
    Therefore, the induced subgraph $\cH(S)$ is isomorphic to $G$, which implies that 
    \begin{equation}\label{eq:independence_equation_dte}
    |\cI(\cH,S)| = |\cI(G)| \enspace.
    \end{equation}
    Also, this means that $\cH$ has at least as many edges as $G$, which means that
    \begin{equation}\label{eq:avg_equation_dte}
    d_{\mathrm{avg}}(\cH) = \frac{2|E_\cH|}{2n} \geq \frac{2 |E|}{2n} = \frac{d_{\mathrm{avg}}(G)}{2}\enspace.
    \end{equation}
    Thus, combining~\eqref{eq:first_bound_dte} together with \eqref{eq:independence_equation_dte} and \eqref{eq:avg_equation_dte} yields the lower bound of the statement.

    For the upper bound, let $\mat{A}_\cH,\mat{A}_G$ denote the adjacency matrices of $\cH$ and $G$, respectively. 
    If we order the nodes in $\cH$ starting from the $n$ exposures $e_{i,0}$ and then moving to the $n$ exposures $e_{i,1}$, then by the preceding discussion we can write $\mat{A}_\cH$ in block form as follows.
    $$
    \mat{A}_\cH = 
    \begin{pmatrix}
        0 &I\\
        I &\mat{A}_G
    \end{pmatrix}\enspace.
    $$
    Now, standard matrix bounds imply that
    $$
    \lamH = \|\mat{A}_\cH\|_{\mathrm{op}} \leq 2 \|I\|_{\mathrm{op}} + \|\mat{A}_G\|_{\mathrm{op}} = 2 + \lambda(G)\enspace.
    $$
    This bound combined with \eqref{eq:first_bound_dte} yields the upper bound of the statement.
\end{proof}

\subsection{Global Average Treatment Effect (Corollary~\ref{corollary:gate-rate})}\label{sec:gate_appendix}

The Global Average Treatment Effect (GATE) was formally defined in Section~\ref{sec:preliminaries}.
We state and prove a more formal version of Corollary~\ref{corollary:gate-rate} below.

\begin{corollary}\label{cor:gate_formal}
    The minimax rate of estimating the Global Average Treatment Effect $\ate_{\mathrm{GATE}}$ on a network $G$ is bounded as
	\[
	\frac{\sqrt{d_{\mathrm{avg}}(G^2)}}{n} 
	\lesssim 
	\risk(G,\ate_{\mathrm{GATE}},2) 
	\lesssim
	\frac{\lambda(G^2)+1}{n}
	\enspace,
	\]
	where $G^2$ is the two-hop interference graph, $\davg(G^2)$ is the average degree of the two-hop graph, and $\lambda(G^2)$ is the largest eigenvalue of its adjacency matrix.\end{corollary}

\begin{proof}
    Let $\cH$ be the conflict graph corresponding to $G$ and $\ate_{\mathrm{GATE}}$.
    Applying the upper bound of Theorem~\ref{thm:conflict_graph_design} for $q=2$, as well as the local lower bound of Corollary~\ref{cor:local_lower_bound_general} together with Proposition~\ref{prop:critical_to_avg}, we get
    \begin{equation}\label{eq:first_bound_gate}
    \frac{\sqrt{d_{\mathrm{avg}}(\cH)}}{n}\lesssim \cR(G,\ate_{\mathrm{GATE}},2) \lesssim \frac{\lamH}{n}
    \end{equation}
    We now connect the quantites on the upper and lower bounds with the desired quantities of the 2-hop graph $G^2 = (V,E(G^2))$. 

    As we have discussed in Section~\ref{sec:conflict_graph}, the conflict graph $\cH$ for $\ate_{\mathrm{GATE}}$ has the following structure: $(e_{i,k},e_{j,l}) \in E_\cH$ if and only if $(i,j) \in E(G^2)$ and $k \neq l$ or $i =j$ and $k \neq l$. 
    Therefore, for every edge $(i,j) \in E(G^2)$ in the 2-hop graph $G^2$, there are two edges $(e_{i,1},e_{j,0}), (e_{i,0},e_{j,1}) \in E_\cH$.
    It follows that $|E_\cH| \geq 2 |E(G^2)|$, which implies that
    \[
    d_{\mathrm{avg}}(\cH) = \frac{2|E_\cH|}{2n} \geq \frac{4|E(G^2)|}{2n} = d_{\mathrm{avg}}(G^2)\enspace.
    \]
    This inequality combined with \eqref{eq:first_bound_gate} establishes the lower bound of the statement. 

    For the upper bound, let $\mat{A}_\cH, \mat{A}_{G^2}$ denote the adjacency matrices of $\cH$ and $G^2$ respectively. 
    Then, if we order the $2n$ nodes of $\cH$ starting from the $e_{i,0}$ exposures and then continuing with the $e_{i,1}$ exposures, we can write $\mat{A}_\cH$ in block form as
    \[
    \begin{pmatrix}
        0 &I + \mat{A}_{G^2}\\
        I + \mat{A}_{G^2} &0
    \end{pmatrix}
    \enspace.
    \]
    The identity part comes from the edges $(e_{i,0},e_{i,1})$. 
    Therefore, we have
    \[
    \lamH = \|\mat{A}_\cH\|_{\mathrm{op}} \leq 2 \|I + \mat{A}_{G^2}\|_{\mathrm{op}} \leq 2 + 2 \lambda(G^2)
    \]
    The above inequality combined with \eqref{eq:first_bound_gate} yields the lower bound in the statement.
\end{proof}

	\section{Tightness of Bounds}\label{sec:tightness_appendix}

In this section, we present the proofs of the results in Section~\ref{sec:comparisons}. We start by providing a convenient description of the conflict graph for direct effect when $G$ is an arbitrary $d$-regular graph. We then continue by establishing that both the upper and lower bounds for estimating the direct effect in regular graphs are tight. Finally, we show that for a random regular graph, the upper bound is tight up to a $\log d$ factor with high probability.

\subsection{The Structure of the Conflict Graph}\label{sec:supp_general_d_regular}
We first provide a Lemma that describes the structure of the conflict graph $\cH$ when the estimand of interest is the Direct Effect.
We denote by $\vec{e}_i \in \R^n$  the vector with $1$ in the $i$-th entry and $0$ everywhere else and $\vec{0} \in \R^n$ be the all $0$ vector. 
Also, for a symmetric matrix $A \in \R^{n \times n}$, we denote by $\|A\|_2$ its spectral norm. 
Furthermore, for any graph $G$, we denote by $\mat{A}_G$ its adjacency matrix.

\begin{lemma}
\label{lem:conflict_graph_direct_effect}
Let $G = (V,E)$ be a graph and $\cH = (V_\cH, E_\cH)$ be the conflict graph associated with $G$ and the direct effect $\tau$. Then, we have that
\begin{enumerate}[label=(\roman*)]
    \item \label{prop:degree_bound} If a node $i \in G$ has degree $d_i$, then $e_{i,1} \in \cH$ has degree $2d_i + 1$ and $e_{i,0} \in \cH$ has degree at most $d_i + 1$.
    \item \label{prop:isometry} Let $T = \{e_{i,1} : i \in [n]\}$ be the subset of vertices of $\cH$. Then, $\cH[T]$ is a graph isometric to $G$, where the isometry maps $e_{i,1}$ to $i$ for each $i \in [n]$.
\end{enumerate}
\end{lemma}
\begin{proof}
    Let us start with the first part. The vertices of $\cH$ are the $2n$ exposures $e_{i,1} = \vec{e}_i$ and $e_{i,0} = \vec{0}$ for $i \in [n]$. For every $i \in [n]$, there is an edge $(e_{i,1},e_{i,0}) \in E_\cH$, since we can not observe both exposures for the same unit. Also, for every $i,j \in [n]$ with $(i,j) \in E$, we have that $(e_{i,1},e_{j,0}), (e_{i,0},e_{j,1}), (e_{i,1},e_{j,1}) \in E_\cH$. This is because if $i$ receives the $e_{i,1}$ exposure, then $i$ is treated, meaning that $j$ has a neighbor that is treated, so $j$ can not receive the $e_{j,0}$ or the $e_{j,1}$ exposure. Thus, the first property follows.

    For the second part, it follows from the previous discussion that for any $i,j \in [n]$, we have that $(e_{i,1}, e_{j,1}) \in E_\cH$ if and only if $(i,j) \in E$. Therefore, the subgraph of $\cH$ induced by $T$ is isometric to $G$.
\end{proof}

\subsection{Tightness of Upper Bound for Direct Effect (Proposition~\ref{prop:d_regular_upper_bound_tightness})}\label{sec:supp_tightness_upper}

\upperboundtightness*

\begin{proof}
    Let $d_n$ be a sequence of degrees such that $d_n = \littleO{n}$ and $d_n n$ is even for all $n$. For a given $n$, we will construct a $d_n$-regular graph $G_n$ on $n$ nodes, which we call the \emph{$d$-regular chorded cycle}. 
    We will assume that $n$ is large enough so that $d_n < n/2$, which is sufficient since our statement is asymptotic.
    The $k$-regular chorded cycle on $n$ nodes $H_{n,k}$ is defined as follows. We arrange the $n$ nodes in a cycle, labeled $1,2,\dots,n$. Two nodes $i,j \in [n]$ are connected if and only if $|i- j| \leq k$, where the difference is taken modulo $n$. In other words, each node is connected to the $k$ nodes on each side of it in the cycle.
    From the definition it follows immediately that $H_{n,k}$ is $2k$-regular.

    Let us now define $G_n$. Since $n d_n$ is even, it must be that either $n$ or $d_n$ is even.
    We distinghish between two cases.
    \begin{itemize}
    \item If $d_n$ is even, we define $G_n = H_{n,d_n /2}$.
    \item If $d_n$ is odd, then $n$ has to be even. We define $G_n$ to be the graph obtained by taking the union of $H_{n,(d_n-1)/2}$ and a perfect matching on the $n$ nodes.
    \end{itemize}
    In either case, $G_n$ is $d_n$-regular.
    We now show that $\cR(G_n,\tau,2) = \bigTheta[\Big]{d_n/n}$.

    For the upper bound, we apply Theorem~\ref{thm:conflict_graph_design}.
    Since $G_n$ is $d_n$-regular, by~\ref{prop:degree_bound} of Lemma~\ref{lem:conflict_graph_direct_effect} it follows that every node in $V_\cH$ has degree at most $2d_n + 1$.
    Thus, $\lamH \leq 2d_n + 1$. Thus, applying Theorem~\ref{thm:conflict_graph_design} with $q=2$, we get that
    \[
    \cR(G_n,\tau,2) \lesssim \frac{2d_n + 1}{n} = \bigO[\Big]{\frac{d_n}{n}}\enspace.
    \]
    
    For the lower bound, we apply Theorem~\ref{thm:global_lower_bound} on the subset $T = \{e_{i,1} : i \in [n]\}$ of the vertices of $\cH$. We will upper bound the size of the largest independent set of $\cH[T]$. By the isometry property~\ref{prop:isometry} of Lemma~\ref{lem:conflict_graph_direct_effect}, it suffices to bound the largest independent set of $G$. We will assume for simplicity that $n$ is divided by $\lfloor d_n /2 \rfloor$, since the general case can be handled by a standard modification of this argument. We partition $[n]$ into the subsets $\{T_i\}_{i=1}^{n / \lfloor d_n /2 \rfloor}$, where
\[
T_i = [\lfloor d_n /2 \rfloor (i-1),\lfloor d_n / 2 \rfloor i - 1] \cap [n] \enspace.
\]
Clearly $T_i$ are disjoint and $[n] = \cup_{i=1}^{n /\lfloor d_n / 2 \rfloor} T_i$. Also, within each $T_i$ all nodes are connected to each other in $G$, since for any $i,j \in T_i$ we have that $|i-j| \leq \lfloor d_n / 2 \rfloor \leq d_n$. Therefore, any independent set in $\cH[T]$ can contain at most one node from each $T_i$. Since there are at most 
\[
\frac{n}{\lfloor \frac{d_n}{2} \rfloor} \leq \frac{n}{d_n / 2 - 1} = \bigO[\Big]{\frac{n}{d_n}}
\]
such subsets, 
we have that $|\cI(\cH[T])| =|\cI| = \bigO{\frac{n}{d_n}}$. Thus, applying Theorem~\ref{thm:global_lower_bound} for the subset $T$, we get that
\[
\cR(G_n,\tau,2) \gtrsim \frac{1}{|\cI(\cH[T])|} \gtrsim \frac{d_n}{n} \enspace.
\]
\end{proof}

\subsection{Tightness of Lower Bound for Direct Effect (Proposition~\ref{prop:d_regular_lower_bound_tightness})}\label{sec:supp_tightness_lower}

For Proposition~\ref{prop:d_regular_lower_bound_tightness}, we will use a family of graphs called the \emph{Generalized Quadrangle}. Thus, we first review some of the properties of these graphs that will be relevant for our purposes. 
To do that, it will be helpful to introduce some basic concepts from the algebraic theory of quadratic forms. For a detailed introduction to this topic, see \citet{lam2005introduction}.
Let $s \in \mathbb{N}$ be a prime number and $\F_s$ be the finite field of residues modulo $s$. 
Let $r \in \mathbb{N}$ be a positive integer.
For every $x \in \F_s^r$, the corresponding \emph{projective point} is the set of all scalar multiples of $x$, i.e.
\[
[x] = \{y \in \F_s^r : y = \lambda x \text{ for some } \lambda \in \F_s\}\enspace.
\]
The set of all projective points is called the \emph{projective space} $PG(r-1,s)$.
Let $B : \F_s^r \times \F_s^r \to \F_s$ be a bilinear form.
Let $Q(r-1,s)$ be the set of all projective points $[x] \in PG(r-1,s)$ such that $B(x,x) = 0$. We call $Q(r-1,s)$ the \emph{quadric} associated with $B$ on the vector space $\F_s^r$. 
There are three types of quadrics, depending on the dimension $r$ and the Witt index\footnote{For the definition of Witt index, see Chapter 1 in \citet{lam2005introduction}.} $w$ of the bilinear form $B$. 
\begin{enumerate}
    \item If $r$ is even and $w = r/2$, then $Q(r-1,s)$ is called a \emph{hyperbolic quadric} and denoted $Q^+(r-1,s)$.
    \item If $r$ is even and $w = r/2 - 1$, then $Q(r-1,s)$ is called an \emph{elliptic quadric} and denoted $Q^-(r-1,s)$.
    \item If $r$ is odd and $w = (r-1)/2$, then $Q(r-1,s)$ is called a \emph{parabolic quadric} and denoted $Q(r-1,s)$.
\end{enumerate}

Most of our discussion will focus on the case $r=5$.
Consider the matrix $M \in \F_s^{5 \times 5}$ defined as
\[
M = 
\left(
\begin{array}{cc|cc|c}
0 & 1 & 0 & 0 & 0 \\
1 & 0 & 0 & 0 & 0 \\ \hline
0 & 0 & 0 & 1 & 0 \\
0 & 0 & 1 & 0 & 0 \\ \hline
0 & 0 & 0 & 0 & 1
\end{array}
\right)\enspace.
\]
We associate to $M$ the bilinear form $B : V \times V \to \F_s$ defined as $B(x,y) = x^T M y$ for all $x,y \in V$. 
We will denote by $Q(4,s)$ the quadric associated with $B$ on the vector space $\F_s^5$.
The Witt index of $B$ is $2$, thus by the above discussion $Q(4,s)$ is a parabolic quadric.
The points in $Q(4,s)$ will be the vertices of our graph.

A \emph{line} $L$ in $Q(4,s)$ is a $2$-dimensional subspace of $V$ such that $B(x,x) = 0$ for all $x \in L$. 
We say a point $[x]$ is \emph{incident} to a line $L$ if $x \in L$. 
The \emph{point graph} $\Gamma(Q(4,s))$ associated with $Q(4,s)$ is the graph with vertices corresponding to the points in $Q(4,s)$, where two points are connected by an edge if and only if they are incident to the same line.

We now summarize some basic properties of the graph $\Gamma(Q(4,s))$ related to the number of vertices, edges and spectrum of the adjacency matrix.

\begin{lemma}
\label{lem:gq_properties}
Let $s \geq 2$ be a prime number. The following properties hold for the graph $G = \Gamma(Q(4,s))$.
\begin{enumerate}[label=(\roman*)]
    \item \label{prop:num_nodes} The number of nodes is $n = (s+1)(s^2 + 1)$.
    \item \label{prop:degree} The graph is $d$-regular with $d = s(s+1)$.
    \item \label{prop:adjacency_matrix} The eigenvalues of the adjacency matrix $\mat{A}_G$ are $s(s+1), s-1, -(s+1)$, with multiplicities $1, s(s+1), s^2$ respectively.
\end{enumerate}
\end{lemma}

\begin{proof}
    As established in Section 3.1 of \citet{payne2009finite}, $G$ satisfies the general properties of a point-line geometry.
    From these properties, it follows that $G$ is a $d$-regular graph with $d = s(s+1)$ and $n = (s+1)(s^2 + 1)$ nodes \citep[see e.g.][Lemma 10.8.1]{godsil2013algebraic}. This establishes properties~\ref{prop:num_nodes} and~\ref{prop:degree}. Property~\ref{prop:adjacency_matrix} is exactly Lemma 10.8.2 of \citet{godsil2013algebraic}. 
\end{proof}

In order for us to construct an improved design that achieves the $\sqrt{d}/n$ rate for this family of graphs, it is important to understand the structure of independent sets in $G$. 
We start by defining a \emph{hyperplane} of $\mathrm{PG}(r,q)$, which is a set of the form
\[
H = \{ [x] \in \mathrm{PG}(r,q) : \alpha^\top x = 0 \},
\]
for some $\alpha \in \F_q^{r} \setminus \{0\}$.
Let $H$ be a hyperplane of $\mathrm{PG}(4,s)$ and $I = H \cap Q(4,s)$. If $I$ is isomorphic to the elliptic quadric $Q^-(3,s)$, then we call $I$ an \emph{ovoid} of $Q(4,s)$.
By definition, $I$ will have Witt index $1$, thus it will contain no lines, meaning that $I$ is an independent set in $G$. By Lemma 3.4.1 in \citet{payne2009finite}, ovoids always exist in $Q(4,s)$. 
The following lemma establishes some crucial properties of ovoids that will be used in the construction of our design.

\begin{lemma}
\label{lem:ovoid_properties}
Let $G = \Gamma(Q(4,s))$ and $\cF$ be the family of ovoids in $Q(4,s)$. Then, the following properties hold.
\begin{enumerate}[label=(\roman*)]
     \item \label{prop:independent_family} Every ovoid has size $s^2 + 1$.
    \item \label{prop:independent_node} Every node $i$ belongs in the same number of independent sets in the family $\{I_f\}_{f \in \cF}$.
    \item \label{prop:independent_pair} Every pair of nodes $i,j$ for which $(i,j) \notin E$ belongs in the same number of independent sets in the family $\{I_f\}_{f \in \cF}$.
\end{enumerate}
\end{lemma}
\begin{proof}
    Property~\ref{prop:independent_family} is exactly Lemma 1.8.1 of \citet{payne2009finite}. For the second and third properties, we will use Witt's Extension Theorem, as stated in Theorem 22 of \citet{clark2013quadratic}. 
    This Theorem states that any isometry between two subspaces of a quadratic space can be extended to an isometry of the whole space.
    In particular, let $[x],[y] \in Q(4,s)$ be two different vertices of the graph. Consider the one dimensional subspaces $V_1 = [x]$ and $V_2 = [y]$, which are clearly isomorphic. By Witt's Extension Theorem, there exists an isometry of the whole space that maps $V_1$ to $V_2$, thus maps $[x]$ to $[y]$. Since it is an isometry, it preserves the quadratic form, thus all edges between vertices are preserved. This means that ovoids containing $[x]$ are in 1-1 mapping with ovoids containing $[y]$, establishing property~\ref{prop:independent_node}. 
    
    Also, let $([x],[y])$ and $([z],[u])$ be two pairs of non-adjacent nodes in $G$. 
    Consider the two-dimensional subspaces $V_1 = \mathrm{span}([x],[y])$ and $V_2 = \mathrm{span}([z],[u])$. Since $[x],[y]$ are non-adjacent, we have that $B(x,y) \neq 0$ and $B(x,x)=B(x,y)=0$, thus the restriction of the quadratic form to $V_1$ leads to a hyperbolic quadric. The same is true for $V_2$, thus $V_1$ and $V_2$ are isometric. By Witt's Extension Theorem, there exists an isometry of the whole space that maps $V_1$ to $V_2$. Since edges are preserved by the isometry, ovoids containing $[x],[y]$ are in 1-1 mapping with ovoids containing $[z],[u]$, establishing property~\ref{prop:independent_pair}.
\end{proof}

We will use the family $\cF$ of independent sets in the graph $G$ of Lemma~\ref{lem:ovoid_properties} to construct a design that achieves the $\sqrt{d}/n$ rate for this family of graphs. The pseudocode for the design is given in Algorithm~\ref{alg:isd}.

\begin{algorithm}[ht]
\KwIn{Graph $G$ on $n$ nodes as in Lemma~\ref{lem:gq_properties}.}
	\KwOut{Random intervention $Z \in \setb{0,1}^n$.}
    Initialize the intervention vector $Z \gets \vec{0}$. \\
    Select $f \in \cF$ uniformly at random. \\
    Sample variables $\setb{U_i : i \in [n]}$ independently and identically as
	$$
	U_i \gets
	\begin{cases*}
		0  \quad\text{with probability } 1/2\\
		1  \quad\text{with probability } 1/2
	\end{cases*}
	$$
    \For{$i \in [n]$}{
        \If{$i \in I_f$}{
            $Z_i \gets U_i$
        }
    \Else{
        $Z_i \gets 0$
    }
	}
\caption{Independent Set Design}
\label{alg:isd}
\end{algorithm}

The design works by selecting a uniformly random independent set from the family $\{I_f\}_{f \in \cF}$ and then treating the nodes in that independent set with probability $1/2$ each, while leaving the rest of the nodes untreated. Therefore, we call this the \emph{Independent Set Design}. The properties of the family $\{I_f\}_{f \in \cF}$ will enable a precise characterization of the spectrum of the covariance matrix that arises from this design.

Our estimator $\eate$ will be a Horvitz-Thompson type of estimator.
For $i \in [n], k \in \setb{0,1}$, we define the event $A_{i,k} = \{i \in I_f, Z_i = k\}$, which corresponds to the event that node $i$ is in the independent set that was selected and receives treatment $k$. We will use the following estimator for $\ate$

\begin{equation}\label{eq:ht_variant}
\eate = \frac{1}{n} \sum_{i=1}^n \paren[\Bigg]{\frac{\indicator{A_{i,1}}}{\Pr{A_{i,1}}} - \frac{\indicator{A_{i,0}}}{\Pr{A_{i,0}}}} Y_i \enspace.
\end{equation}

Our next lemma shows that this estimator is unbiased for $\ate$. 

\begin{lemma}
\label{lem:ht_unbiased_gq}
The estimator $\eate$ defined in~\eqref{eq:ht_variant} is an unbiased estimator for $\ate$, i.e. $\E{\eate} = \ate$.
\end{lemma}
\begin{proof}
We start by observing that if $i \in I_f$ and $Z_i = k$, then all neighbors of $i$ are in control thus $h_i(Z) = e_{i,k}$ and $Y_i = y_i(e_{i,k})$. It also follows by Algorithm~\ref{alg:isd} that $\Pr{A_{i,k}}  > 0$ for all $i \in [n], k \in \setb{0,1}$. Therefore, we can write
\begin{align*}
\E{\eate} &= \E[\Bigg]{\frac{1}{n} \sum_{i=1}^n \paren[\Bigg]{\frac{\indicator{A_{i,1}}}{\Pr{A_{i,1}}}y_i(e_{i,1}) - \frac{\indicator{A_{i,0}}}{\Pr{A_{i,0}}}y_i(e_{i,0})}} \\
&= \frac{1}{n} \sum_{i=1}^n \paren[\Bigg]{\frac{\Pr{A_{i,1}}}{\Pr{A_{i,1}}}y_i(e_{i,1}) - \frac{\Pr{A_{i,0}}}{\Pr{A_{i,0}}}y_i(e_{i,0})} \\
&= \ate\enspace.
\end{align*}
\end{proof}

We now provide an auxiliary lemma that contains expressions for the variance and covariance terms arising in the analysis of the variance of the estimator $\eate$.

\begin{lemma}
\label{lem:prob_expressions_gq}
Suppose $G$ is given by Lemma~\ref{lem:gq_properties} and $Z$ is generated according to the Independent Set Design of Algorithm~\ref{alg:isd}. Then, for every $i,j \in [n], k \in \setb{0,1}$ we have that
\begin{enumerate}
    \item 
    \[
    \Var[\Big]{\frac{\indicator{A_{i,k}}}{\Pr{A_{i,k}}}} = s \enspace.
    \] 
    \item If $i \neq j$ 
    \[
\Cov[\Big]{\frac{\indicator{A_{i,k}}}{\Pr{A_{i,k}}}, \frac{\indicator{A_{j,k}}}{\Pr{A_{j,k}}}} = \begin{cases}
    \frac{1}{s} & \text{if } (i,j) \notin E \\
    -1 & \text{if } (i,j) \in E
\end{cases} \enspace.
    \]
\end{enumerate}
\end{lemma}

\begin{proof}
    We start with the first part. We have that
    \begin{align*}
        \Var[\Big]{\frac{\indicator{A_{i,k}}}{\Pr{A_{i,k}}}} &= \frac{1}{\Pr{A_{i,k}}} - 1 
    \end{align*}
    Thus, it suffices to compute $\Pr{A_{i,k}}$. Since the subset $I_f$ is selected uniformly at random from the family $\{I_f\}_{f \in \cF}$, independently of the random variables $\{U_i\}_{i=1}^n$, we have that
    \begin{align*}
        \Pr{A_{i,k}} &= \Pr{i \in I_f, Z_i = k} \\
        &= \Pr{i \in I_f}\Pr{Z_i = k | i \in I_f} \\
        &= \Pr{i \in I_f}\Pr{U_i = k} \\
        &= \frac{1}{2} \frac{|\setb{f: i \in I_f}|}{|\cF|} \enspace.
    \end{align*}
    Thus, it suffices to compute $|\setb{f: i \in I_f}|$. By property~\ref{prop:independent_node} of Lemma~\ref{lem:ovoid_properties}, every node belongs in the same number of independent sets in the family $\{I_f\}_{f \in \cF}$, thus this set has the same size $M$ for all $i \in [n]$. We have that
    \begin{align*}
    M n &= \sum_{i=1}^n |\setb{f: i \in I_f}| \\
    &= \sum_{f \in \cF} |I_f| \\
    \intertext{by exhanging the order of summation}\\
    &= |\cF|(s^2 + 1) \enspace,
    \end{align*}
    where in the last step we used property~\ref{prop:independent_family} of Lemma~\ref{lem:ovoid_properties}. Thus, $M = |\cF|(s^2 + 1)/n$. Therefore, we have that
    \begin{align*}
        \Pr{A_{i,k}} &= \frac{1}{2} \frac{|\setb{f: i \in I_f}|}{|\cF|} \\
        &= \frac{1}{2} \frac{|\cF|(s^2 + 1)/n}{|\cF|} \\
        &= \frac{1}{2} \frac{s^2 + 1}{n} \\
        &= \frac{1}{2} \frac{s^2 + 1}{(s+1)(s^2 + 1)} \\
        &= \frac{1}{2(s+1)} \enspace,
        \end{align*}
    where we used property~\ref{prop:num_nodes} of Lemma~\ref{lem:gq_properties} in the last step. Thus, we have that
    \begin{align*}
        \Var[\Big]{\frac{\indicator{A_{i,k}}}{\Pr{A_{i,k}}}} &= \frac{1}{\Pr{A_{i,k}}} - 1 = 2s + 1
    \end{align*}

    For the second part, we have that
    \begin{align*}
        \Cov[\Big]{\frac{\indicator{A_{i,k}}}{\Pr{A_{i,k}}}, \frac{\indicator{A_{j,k}}}{\Pr{A_{j,k}}}} = \frac{\Pr{A_{i,k} \cap A_{j,k}}}{\Pr{A_{i,k}}\Pr{A_{j,k}}} - 1 
        \end{align*}
    Thus, it suffices to compute $\Pr{A_{i,k} \cap A_{j,k}}$. If $(i,j) \in E$, then $i$ and $j$ can not be in the same independent set, thus $\Pr{A_{i,k} \cap A_{j,k}} = 0$. This proves the first case.
    If $(i,j) \notin E$, then by property~\ref{prop:independent_pair} of Lemma~\ref{lem:ovoid_properties} every pair of non-adjacent nodes belongs in the same number of independent sets in the family $\{I_f\}_{f \in \cF}$. Let $L$ be this number. 
    Following the same strategy as before, we sum over all non-adjacent pairs of nodes and get that
    \begin{align*}
        L \cdot \paren[\Bigg]{\binom{n}{2} - \frac{nd}{2}} &= \sum_{i,j: (i,j) \notin E} |\setb{f: i,j \in I_f}| \\
        &= \sum_{f \in \cF} {|I_f| \choose 2} \\
        \intertext{by exhanging the order of summation}\\
        &= |\cF| \frac{(s^2 + 1)s^2}{2} \enspace,
    \end{align*}
    where in the last step we used property~\ref{prop:independent_family} of Lemma~\ref{lem:ovoid_properties}. Thus, we have that
    \begin{align*}
        \Pr{A_{i,k} \cap A_{j,k}} &= \Pr{i,j \in I_f, Z_i = k, Z_j = k} \\
        &= \Pr{i,j \in I_f}\Pr{Z_i = k, Z_j = k | i,j \in I_f} \\
        &= \Pr{i,j \in I_f}\Pr{U_i = k, U_j = k} \\
        &= \frac{1}{4} \frac{L}{|\cF|} \\
        &= \frac{1}{4} \frac{\frac{(s^2 + 1)s^2}{2}}{\binom{n}{2} - \frac{nd}{2}} \\
        &= \frac{1}{4} \frac{(s^2 + 1)s^2}{n(n-1-d)} \\
        &= \frac{1}{4} \frac{(s^2 + 1)s^2}{(s+1)(s^2 + 1)((s+1)(s^2 + 1) - 1 - s(s+1))} \\
        &= \frac{1}{4} \frac{s^2}{(s+1) s^3} \\
        &= \frac{1}{4s(s+1)} \enspace.
    \end{align*}

Therefore,
\begin{align*}
\Cov[\Big]{\frac{\indicator{A_{i,k}}}{\Pr{A_{i,k}}}, \frac{\indicator{A_{j,k}}}{\Pr{A_{j,k}}}} &= \frac{\Pr{A_{i,k} \cap A_{j,k}}}{\Pr{A_{i,k}}\Pr{A_{j,k}}} - 1 \\
&= \frac{\frac{1}{4s(s+1)}}{\frac{1}{4(s+1)^2}} - 1 \\
&= \frac{(s+1)^2}{s(s+1)} - 1 \\
&= \frac{s+1}{s} - 1 \\
&= \frac{1}{s} \enspace.
\end{align*}
\end{proof}

We are now ready to prove Proposition~\ref{prop:d_regular_lower_bound_tightness}, which we also restate here for convenience.

\lowerboundtightness*

\begin{proof}
    Assume for simplicity that there is a prime number $s$ such that $n = (s+1)(s^2 + 1)$ and $d = s(s+1)$. We will show in the end how to extend to arbitrary values of $n$ and $d$ by adding isolated nodes to the graph. Let $G$ be the graph on $n$ nodes with the properties of Lemma~\ref{lem:gq_properties}. 

    Let us start from the lower bound. Since $G$ is $d$-regular, by~\ref{prop:degree_bound} of Lemma~\ref{lem:conflict_graph_direct_effect} it follows that every node in the conflict graph $\cH$ has degree at least $d$. Therefore, $d^*(\cH) \geq d$. Thus, applying the lower bound of Corollary~\ref{corollary:direct-rate} with $q=2$, we get that
    \[
    \cR(G,\tau,2) \gtrsim \frac{\sqrt{d}}{n} \enspace.
    \]

    Now, let us establish the upper bound. This involves finding a statistical procedure, i.e. a pair of design $P$ and estimator $\eate$, whose risk is $\bigO{\sqrt{d}/n}$. We will use the Independent Set Design of Algorithm~\ref{alg:isd} as our design $P$. Below, when we write probabilities, it will always be with respect to the randomness of the Independent Set Design. 
    
    Our estimator $\eate$ will be the one defined in~\eqref{eq:ht_variant}. As established in Lemma~\ref{lem:ht_unbiased_gq}, this estimator is unbiased for $\ate$. Therefore, it suffices to upper bound the variance of $\eate$.
    To simplify analysis, we break down the estimator into the estimators for the two exposures, i.e. we write $\eate = \eate_1 - \eate_0$, where
\begin{gather*}
\eate_1 = \frac{1}{n} \sum_{i=1}^n \frac{\indicator{A_{i,1}}}{\Pr{A_{i,1}}}y_i(e_{i,1}) \\
\eate_0 = \frac{1}{n} \sum_{i=1}^n \frac{\indicator{A_{i,0}}}{\Pr{A_{i,0}}}y_i(e_{i,0})
\end{gather*}
We have that
\begin{align*}
    \Var{\eate} &= \Var{\eate_1 - \eate_0} \\
    &= \Var{\eate_1} + \Var{\eate_0} - 2\Cov{\eate_1, \eate_0} \\
    &\leq \Var{\eate_1} + \Var{\eate_0} + 2\sqrt{\Var{\eate_1}\Var{\eate_0}} \\
    &\leq 2\Var{\eate_1} + 2\Var{\eate_0} \enspace,
\end{align*}
where we used the Cauchy-Schwarz and AM-GM inequalities. Therefore, it suffices to upper bound $\Var{\eate_1}$ and $\Var{\eate_0}$. We focus on $\eate_1$, since the analysis for $\eate_0$ is identical.
We have
\begin{align*}
\Var{\eate_1} &= \frac{1}{n^2} \Var[\Bigg]{\sum_{i=1}^n \frac{\indicator{i \in I_f, Z_i = 1}}{\Pr{i \in I_f, Z_i = 1}}y_i(e_{i,1})} \\
&= \frac{1}{n^2} \sum_{i=1}^n \sum_{j=1}^n \frac{ \Cov{\indicator{i \in I_f, Z_i = 1}, \indicator{j \in I_f, Z_j = 1}}}{\Pr{i \in I_f, Z_i = 1}\Pr{j \in I_f, Z_j = 1}} y_i(e_{i,1})y_j(e_{j,1})\enspace.
\end{align*}

Define the vector $\vec{\alpha} \in \R^n$ as $\alpha_i = y_i(e_{i,1}) / \sqrt{n}$ for $i \in [n]$. Then, the preceding variance can be written as a quadratic form 
\[
\Var{\eate_1} = \frac{1}{n}\vec{\alpha}^\top \mat{C} \vec{\alpha} \enspace,
\]
where $\mat{C} \in \R^{n \times n}$ is the matrix with entries
\[
C_{i,j} = \frac{ \Cov{\indicator{A_{i,1}}, \indicator{A_{j,1}}}}{\Pr{A_{i,1}}\Pr{A_{j,1}}} \enspace.
\]
Since $y \in \cM(G,\tau,2)$, we have that
\[
\|\vec{\alpha}\|^2 = \frac{1}{n} \sum_{i=1}^n y_i(e_{i,1})^2 \leq 1 \enspace.
\]
It follows that
\begin{equation}\label{eq:var_operator}
\Var{\eate_1} \leq \frac{1}{n} \|\mat{C}\|_2 \enspace.
\end{equation}

We have already calculated all entries of $\mat{C}$ in Lemma~\ref{lem:prob_expressions_gq}. In particular, if $\mat{I}\in \R^{n \times n}$ is the identity matrix, $\vec{1} \in \R^n$ is the all $1$ vector and $\mat{A}_G$ is the adjacency matrix of $G$, then we can write
\begin{align*}
\mat{C} &= 2(s+1)\mat{I} + \frac{1}{s}(\vec{1}\vec{1}^\top - \mat{I} - \mat{A}_G) - \mat{A}_G \\
&= \paren[\Big]{2(s+2) - \frac{1}{s}} \mat{I} - \paren[\Big]{1 + \frac{1}{s}} \mat{A}_G + \frac{1}{s}\vec{1}\vec{1}^\top \enspace.
\end{align*}

The matrix $\vec{1}\vec{1}^\top/s$ has one eigenvalue equal to $n/s$ with eigenvector $\vec{1}$ and $n-1$ eigenvalues equal to $0$. The matrix $\mat{A}_G$ has leading eigenvalue equal to $d$ with eigenvector $\vec{1}$. Thus, the absolute value of the eigenvalue corresponding to the eigenvector $\vec{1}$ for matrix $\vec{1}\vec{1}^\top/s - (1 + 1/s)\mat{A}_G$ is equal to
\[
\abs[\Bigg]{\frac{n}{s} - \paren[\Big]{1 + \frac{1}{s}}d} = \abs[\Bigg]{\frac{(s+1)(s^2 + 1)}{s} - \paren[\Big]{1 + \frac{1}{s}}s(s+1)} = \frac{s^2 - 1}{s} = \bigO{s} \enspace.
\]
By property~\ref{prop:adjacency_matrix} of Lemma~\ref{lem:gq_properties}, the remaining eigenvalues of $\mat{A}_G$ are $s-1$ and $-(s+1)$, which are both $\bigO{s}$ in absolute value. 
Thus, we have established that
\[
\norm[\Bigg]{\frac{1}{s}\vec{1}\vec{1}^\top - \paren[\Big]{1 + \frac{1}{s}}\mat{A}_G}_2 = \bigO{s} \enspace.
\]
We also trivially have that
\[
\norm[\Bigg]{\paren[\Big]{2(s+2) - \frac{1}{s}} \mat{I}}_2 = \bigO{s} \enspace.
\]
Thus, by the triangle inequality we have that $\|\mat{C}\|_2 = \bigO{s}$. Therefore, by~\eqref{eq:var_operator} we concude that $\Var{\eate_1} = \bigO{s/n}$. By the same argument, we also have that $\Var{\eate_0} = \bigO{s/n}$. Thus, we have that $\Var{\eate} = \bigO{s/n}$. Since $s = \bigTheta{\sqrt{d}}$, we have that $\Var{\eate} = \bigO{\sqrt{d}/n}$, which establishes the upper bound.
\end{proof}

\subsection{Random Regular Graphs (Corollaries~\ref{cor:random_d_regular} and~\ref{cor:random_d_regular_gate})}\label{sec:random_d_regular_appendix}

A \emph{random d-regular graph} $G$ is a graph chosen uniformly at random from the set of all $d$-regular graphs on $n$ vertices $\cG_{n,d}$. We denote this as $G \sim \cU(\cG_{n,d})$. In this section, we identify the minimax rate for random $d$-regular graphs up to a logarithmic factor. We will use the following fact about independent sets in $d$-regular graphs.
In this section, we identify the minimax rate for random $d$-regular graphs up to a logarithmic factor. We will use the following fact about independent sets in random $d$-regular graphs. Recall that for a graph $G$, we denote by $\cI(G)$ the largest independent set in $G$.

\begin{theorem}\label{thm:random_regular_independent_sets}
    Suppose $G \sim \cU(\cG_{n,d})$. There exist constants $c_1, c_2 > 0$ such that if $c_1 \leq d < n - c_2 n/\log n$, then with probability $1 - o(1)$ as $n \to \infty$
    \[
        |\cI(G)| = \bigO[\Big]{\frac{n \log d}{d}}\enspace.
    \]
\end{theorem}

\begin{proof}
    \citet{frieze1992independence} showed that there exists a constant $c_1 > 0$ such that if $d \geq c_1$ and $d = o(n^{1/3})$, then for $G \sim \cU(\cG_{n,d})$, with probability $1 - o(1)$ as $n \to \infty$
    \[
        |\cI(G)| = \bigO[\Big]{\frac{2n \log d}{d}}\enspace.
    \]
    Furthermore, from the discussion after Theorem 5.2 in \citet{krivelevich2001random}, it follows there exists a constant $c_2 > 0$ such that for any $\alpha < 1/2$, if $n^\alpha < d < n - c_2 n/\log n$, then with probability $1 - o(1)$
    \[
        |\cI(G)| = \bigO[\Big]{\frac{n \log d}{d}}\enspace.
    \]
    We can thus choose any $\alpha < 1/3$ and combine these two results to get the desired conclusion.
\end{proof}

We are now ready to prove Corollary~\ref{cor:random_d_regular}, which we restate here for convenience.

\randomdregulardte*

\begin{proof}
    The upper bound follows from Corollary~\ref{corollary:direct-rate}. For the lower bound, we can apply Theorem~\ref{thm:random_regular_independent_sets} to get that with probability $1 - o(1)$ as $n \to \infty$
    \[
        |\cI(G)| = \bigO[\Big]{\frac{n \log d}{d}}\enspace.
    \]
    We can then apply the global lower bound of Theorem~\ref{thm:global_lower_bound} to get that with probability $1 - o(1)$ as $n \to \infty$
    \[
        R(\cG_{n,d}) = \bigOmega[\Big]{\frac{\log d}{d}}\enspace,
    \]
    which completes the proof.
\end{proof}

Finally, we present the proof of Corollary~\ref{cor:random_d_regular_gate} for the global effect, which we restate here for convenience.

\randomdregulargate*

\begin{proof}
    Corollary~\ref{corollary:gate-rate} yields that for any graph $G$,
    \[
        \frac{\sqrt{d_{\mathrm{avg}}(G^2)}}{n} \lesssim \cR(G, \ate_{\mathrm{GATE}} ,2) \lesssim \frac{\lambda(G^2)+1}{n} \enspace.
    \]
    If $G$ is a $d$-regular graph, then every node has exactly $d$ neighbors, and each of those neighbors has at most $d-1$ other distinct neighbors. Therefore, every node in $G^2$ has degree at most $d(d-1) + d  = d^2$. It follows that $\lambda(G^2) \leq d^2$ for every $d$-regular graph $G$, from which the upper bound follows. 

    For the lower bound, let $X_r$ denote the number of walks of length $r$ in $G$. Then, if $d_i(G^2)$ denotes the degree of node $i$ in $G^2$, we have that
    \begin{equation}\label{eq:cycles}
    \sum_{i=1}^n d_i(G^2) = nd^2 - 6 X_3 - 4 X_4 \enspace. 
    \end{equation}
    This is because each triangle of nodes $a,b,c$ contributes two ``invalid'' paths of lenght $2$ for each of the three vertices. For example, for vertex $a$, the paths $a \to b \to c$ and $a \to c \to b$ are invalid, since they end up in a 1-hop neibhbor of $a$. Similarly, each 4-cycle of nodes $a,b,c,d$ contributes one invalid path of length $2$ for every vertex. For example, for vertex $a$, the paths $a \to b \to c$ and $a \to d \to c$ both end up in the same vertex, so one of them is considered invalid.

    Using Theorem 3(a) from \citet{mckay2004short}, we have the following bounds on the expected number of triangles and 4-cycles in a random $d$-regular graph $G$, which hold as long as $d = o(\sqrt{n})$. 
    \[
    \E{X_3} = \frac{d^3}{6}(1 + o(1)) \quad \text{and} \quad \E{X_4} = \frac{d^4}{8}(1 + o(1)) \enspace.
    \]
    By Markov's inequality and the fact that $d = o(n)$, if follows that $X_3 = \littleOp{nd^2}$ and $X_4 = \littleOp{nd^2}$. Therefore, by equation~\eqref{eq:cycles}, we have that
    \[
        d_{\mathrm{avg}}(G^2) = \frac{1}{n} \sum_{i=1}^n d_i(G^2) = d^2(1 + \littleOp{1}) \enspace,
        \]
    from which the lower bound follows.
\end{proof}
	
\end{document}